\documentclass{amsart}
\newtheorem{theorem}{Theorem}[section]
\newtheorem{proposition}[theorem]{Proposition}
\newtheorem{lemma}[theorem]{Lemma}
\newtheorem{lemma-definition}[theorem]{Lemma and Definition}
\theoremstyle{definition}
\newtheorem{definition}[theorem]{Definition}
\newtheorem{example}[theorem]{Example}
\newtheorem{xca}[theorem]{Exercise}
\newtheorem{corollary}[theorem]{Corollary}
\newtheorem{prop}[theorem]{Proposition}
\theoremstyle{remark}
\newtheorem{remark}[theorem]{Remark}
\numberwithin{equation}{section}

\def\A{\mathbb A}
\def\B{\mathbb B}
\def\G{\mathbb G}
\def\H{\mathbb H}
\def\C{\mathbb C}
\def\Q{\mathbb Q}
\def\P{\mathbb P}
\def\R{{\mathbb R}}
\def\Z{\mathbb Z}
\def\N{\mathbb N}

\def\ZZ{\mathcal Z}
\def\FF{\mathcal F}
\def\EE{\mathcal E}

\def\HH{\mathcal H}
\def\LL{\mathcal L}

\def\RR{\mathcal R}
\def\OO{\mathcal O}
\def\PP{\mathcal P}

\def\n{\noindent}

\newcommand{\im}{ \hbox{\rm Im} }

\begin{document}
\title {Mixed Hodge Structures}
\author{Fouad El Zein}
\address{Institut de Math\'ematiques de Jussieu\\
Paris, France}
 \email{elzein@math.jussieu.fr}
\author{L\^e D\~ung Tr\'ang}\footnote{\n The notes were partially written 
with the help of the Research fund of Northeastern University} 
\address{CMI, Universit\'e de Provence\\
F-13453 Marseille cedex 13, France}
 \email{ledt@ictp.it}
\maketitle \centerline {Summer School  on Hodge Theory and Related
Topics}
 \centerline {International center for theoretical Physics}
\centerline {14 June 2010 - 2 July 2010}
 \centerline { Trieste -
ITALY}

\begin{abstract}
With a basic knowledge  of cohomology theory, the background necessary to
understand Hodge theory
 and polarization, Deligne's Mixed Hodge
Structure  on cohomology of complex algebraic varieties is
described.
\end{abstract}

\maketitle

\section*{Introduction}

We assume that the reader is familiar with the basic theory of
manifolds, basic algebraic geometry as well as cohomology theory. For instance, the reader
should be familiar with the contents of  \cite{Warn}, \cite{bott-tu} or \cite{tu_m}, 
the beginning of \cite{H}, the notes  \cite{E-T}  in this book.
\vskip.1in

According to Deligne \cite{HII} and \cite{HIII}, the cohomology space $H^n(X,\C)$
of a complex algebraic variety $X$ carries two finite filtrations by
complex sub-vector spaces, the weight filtration $W$ and the Hodge filtration $F$. In
addition, the subspaces $W_j$ are rationally defined, i.e., $W_j$ is
generated by  rational elements in $H^n(X,\Q)$. 

For a
smooth compact variety, the filtration $W$ is trivial; however the filtration $ F$ and
its conjugate $\overline F$, with respect to the rational
cohomology, define  the Hodge decomposition. 
The structure in linear algebra
defined on a complex vector space by such decomposition is called a Hodge
Structure (HS).

For any complex algebraic variety $X$, the filtration $W$ is defined on $H^n(X, \Q)$.
We can define the homogeneous part $Gr^W_j H^n(X, \Q)= W_j/W_{j-1}$
of degree $j$ in the graded vector space 
$Gr^W_* H^n(X, \Q):= \oplus_{j\in \Z} W_j/W_{j-1}$
associated to the filtration $W$ of $H^n(X, \Q)$.
For each integer $j$, there exists
an isomorphism: 
$$Gr^W_j H^n(X, \Q) \otimes_{\Q} \C \simeq Gr^W_j
H^n(X, \C),$$ 
such that conjugation on $Gr^W_j H^n(X, \C)$ with
respect to $Gr^W_j H^n(X, \Q)$ is well-defined.  Deligne showed
that the
filtration $F$ with its conjugate induce  a Hodge
decomposition on $ Gr^W_j H^n(X, \C)$ of weight $j$.
These data define a Mixed Hodge Structure
on the cohomology of the complex algebraic variety $X$. 

On a smooth variety $X$,
the weight filtration $W$ and the Hodge filtration $F$ reflect properties 
of the Normal Crossing
Divisor (NCD) at infinity of a completion of the variety, but are
independent of the choice of the Normal Crossing
Divisor.
  
Inspired by the properties  of \'etale cohomology
 of varieties over
fields with positive characteristic, constructed by A. Grothendieck
and M. Artin, Deligne established the existence of a Mixed Hodge Structure on the
cohomology of complex algebraic varieties, depending naturally 
on algebraic morphisms but not on
continuous maps. The theory has been fundamental in the study of
topological
properties of complex algebraic varieties.

After this introduction, we recall the analytic background  of
Hodge decomposition on the cohomology of K\"{a}hler manifolds.
We mention \cite{g-h} and \cite{Vo}, which are 
references more adapted to algebraic
geometers, and where one can also find 
other sources as classical books by A.
Weil \cite{Weil} and R.O. Wells \cite {wel} for which we refer for full proofs in the first two sections. In the second section, we recall the Hard
Lefschetz Theorem and Riemann bilinear relations in order to define a
polarization on the cohomology of projective smooth varieties.

We give an abstract definition of Mixed Hodge Structures in the 
third section as an object of
interest in linear algebra, following closely Deligne \cite{HII}. The
algebraic properties of Mixed Hodge Structures are 
developed on objects with three opposite
filtrations in any abelian category. In section four, we need to
develop algebraic homology techniques on filtered complexes up to
filtered quasi-isomorphisms of complexes. The main contribution by
Deligne was to fix the axioms of a Mixed Hodge Complex (MHC)
giving rise to a Mixed Hodge Structure on its cohomology. 
A delicate lemma of Deligne on two filtrations needed to prove that the
weight spectral sequence of a Mixed Hodge Complex is in the category 
of Hodge Structure is explained.
This central result is
stated and applied in section five, where in particular, the
extension of Hodge Structures to all smooth compact complex algebraic
varieties is described by
Hodge filtrations on the analytic de Rham complex and not harmonic
forms, although  the proof is by
reduction to K\"{a}hler manifolds. 

For a non-compact smooth algebraic variety $X$,  we
need to take into account the properties at infinity, that is the properties
of the Normal Crossing Divisor, the complement of the variety 
in a proper compactification, and introduce 
Deligne's logarithmic de Rham
complex in section $6$.

If $X$ is singular, we  introduce a smooth simplicial
 covering to construct the Mixed Hodge Complex in section $7$. We mention also an
 alternative construction.
 
Each section starts with the statement of the results. The proofs
 are given in later subsections because they can be technical.
Since a Mixed Hodge Structure involves rational cohomology constructed via
techniques different than de Rham cohomology, when we define a
morphism of Mixed Hodge Structures, we ask for compatibility with the filtrations
as well as the identification of cohomology via the various
constructions.  For this reason it is convenient to work
constantly at the level of complexes in the language of filtered
and bi-filtered derived categories to ensure that the Mixed Hodge Structures
constructed on cohomology are canonical and do not depend on the
particular definition of cohomology used. Hence, we add a
lengthy introduction on hypercohomology in section  $4$.

As applications let us mention Deformations of smooth proper
analytic families which define a linear invariant called Variation
of Hodge Structure (VHS) introduced by P. Griffiths, and limit Mixed Hodge
Structures adapted to study the degeneration of Variation of 
Hodge Structures. Variation of Mixed Hodge
Structure (VMHS) are the topics of other lectures.

Finally we stress that we only introduce the necessary language
to express the statements and their implications in  Hodge theory, 
but we encourage mathematicians
to look into the foundational work of Deligne in references
\cite{HII} and \cite{HIII}  to discover his dense and unique style, since here intricate
proofs in Hodge theory and spectral sequences are only surveyed.
We remark that further developments in the theory occurred
for cohomology with coefficients in Variations of Hodge
Structure and the theory of differential Hodge modules (\cite{BBD},\cite{Se1}).

\vskip 0.3cm
\centerline {Contents}

Introduction \hfill  1 \\
\n 1 Hodge decomposition \hfill  3  \\
 1.1 Harmonic forms on a differentiable manifold \hfill  4 \\
 1.2  Complex manifolds and decomposition into types \hfill   6 \\
 1.3 Hermitian metric, its associated Riemannian metric and $(1,1)$ form \hfill  10 \\
 1.4 Harmonic forms on compact complex manifolds \hfill  12 \\
 1.5  K\"{a}hler manifolds \hfill  13\\
 1.6 Cohomology class of a subvariety and Hodge conjecture \hfill  15 \\
  2 Lefschetz decomposition and Polarized  Hodge structure \hfill  20 \\
 2.1  Lefschetz decomposition and primitive cohomology \hfill  20 \\
 2.2 The  category of Hodge Structures  \hfill  23 \\
 2.3 Examples \hfill  27 \\
  3 Mixed Hodge Structures (MHS) \hfill   29 \\
 3.1 Filtrations \hfill  31 \\
 3.2 Opposite filtrations \hfill  33 \\
 3.3  Proof of the canonical decomposition of the weight filtration \hfill  36\\
 3.4 Complex Mixed Hodge Structures  \hfill  38 \\
  4 Hypercohomology and spectral sequence of a filtered complex   \hfill  39 \\
 4.1 Derived filtered   category \hfill  42 \\
 4.2 Derived   functor on a filtered complex \hfill  45 \\
 4.3 The spectral sequence defined by a  filtered complex \hfill  47 \\
  5 Mixed Hodge Complex (MHC)  \hfill  50 \\
 5.1 HS on the cohomology  of a smooth compact algebraic variety \hfill   54 \\
 5.2 MHS on  the cohomology of a mixed Hodge Complex \hfill 56 \\
  6 Logarithmic complex, normal crossing divisor and the mixed cone   \hfill 59   \\
 6.1 MHS on the cohomology of smooth varieties \hfill 60 \\
 6.2 MHS of a normal crossing divisor (NCD) \hfill  66  \\
 6.3 Relative cohomology and the mixed cone \hfill  68 \\
  7  MHS on the cohomology of a complex  algebraic variety  \hfill  69 \\
 7.1 MHS on  cohomology of simplicial varieties \hfill 70 \\
 7.2 MHS on the cohomology of a complete embedded algebraic  variety \hfill 77 \\
   References \hfill  79

\section{Hodge decomposition}

In this section we explain  the importance of the Hodge
decomposition and its relation to various concepts in geometry.
Let $ \EE^*(X)$ denotes the de Rham complex of differential forms
with complex values on a complex manifold $X$ and $\EE^{p,q}(X)$
the differentiable forms with complex coefficients of type $(p,q)$ (see \ref{type} below). 
The cohomology subspaces of type $(p,q)$
are defined as:
\begin{equation*}
     H^{p,q}(X) = \frac{Z_d^{p,q}(X)}{d \EE^*(X) \cap
     Z_d^{p,q}(X)}\, \, \mbox{ where } \, \, Z_d^{p,q}(X)= Ker\, d
     \cap \EE^{p,q}(X)
 \end{equation*}
\begin{theorem}[Hodge decomposition] Let $X$ be a compact K\"{a}hler
manifold (see \ref{Kahler}). There exists a decomposition, called the Hodge decomposition, 
of the complex cohomology
  spaces into a direct sum of complex subspaces:
  \begin{equation*}
H^i (X, \C ) = \oplus_{p+q=i} H^{p,q}(X), \quad {\rm satisfying }
\quad H^{p,q}(X) = \overline {H^{q,p}(X)}.
 \end{equation*}
 \end{theorem}
 
 Since a smooth complex projective variety is  K\"{a}hler (see \ref{examples}), we
 deduce:
 
 \begin{corollary}There exists a Hodge decomposition on the
 cohomology of a smooth complex projective variety.
\end{corollary}

  The above Hodge decomposition  theorem  uses successive
 fundamental concepts  from  Riemannian geometry
 such as harmonic forms,  complex analytic ma\-nifolds such as
  Dolbeault cohomology and from  Hermitian geometry.  It will be extended
  later to algebraic geometry.
  
 The manifold  $X$ will denote successively in this section  a
differentiable,  complex analytic, and finally a compact
 K\"{a}hler manifold. We give in this section a summary of the theory for the reader who is not
familiar with  the  above theorem, which includes the definition of the {\it Laplacian} in Riemannian geometry and the space of harmonic forms isomorphic to the cohomology of $X$.
On complex manifolds, we start with a {\it Hermitian structure}, its underlying {\it Riemannian structure} and its associated {\it fundamental $(1,1)$ form } used to construct the Lefschetz decomposition. 
It is only on {\it K\"{a}hler manifolds} when the  form is closed that the components of type $(p,q)$ of an  {\it harmonic form}  are  harmonics. For  an extended exposition of the theory with
full proofs, including the subtle linear algebra of Hermitian metrics  needed here, we
refer to \cite {wel}, see also  \cite{g-h}, \cite{Vo}  ( for original  articles see \cite{Hodge} and even \cite{Lefsch}). 

\subsection{Harmonic forms on a differentiable manifold}
  A Riemannian manifold $X$ is endowed with a {\it  scalar product on
its tangent bundle } defining a metric. Another basic result in
analysis states that
 the {\it cohomology of  a  compact
smooth  Riemannian manifold is
 represented by real harmonic global differential forms}.
To define harmonic  forms, we need to introduce Hodge {\it
${\mathbb S}tar$-operator and the Laplacian}.

\subsubsection{Riemannian metric} A bilinear form $g$ on $X$
 is defined at each point $x$ as a product on the
 tangent space  $T_x$  to $X$ at $x$
\begin{equation*}
 g_x( , ) : T_x \otimes_{\R} T_x \to \R
\end{equation*}
where $g_x$ vary smoothly with $x$, that is $h_{ij}(x):= g_x(\frac
{\partial}{\partial x_i} , \frac { \partial}{\partial x_j} ) $ is
differentiable in $x$, then the  product is given as $g_x =
\sum_{i,j} h_{ij}(x) dx_i\otimes d x_j$. It is the local
expression of a global product on fiber bundles $g: T_{X}
\otimes T_{X} \to \R_X $. It is called a metric if moreover the
matrix of the the product defined by $h_{ij}(x)$ is positive
definite, that is: $ g_x(u,u) > 0$ for all $u \neq 0\in
T_x$.

A globally induced metric on the cotangent bundle, is  defined
locally on the fiber $\EE^1_x:= T_x^*:= Hom(T_x, \R)$ and more
generally on the differential forms $\EE^p_x$ as follows.

Let $e_1, \ldots,e_n$ be an orthonormal basis of $T_x$,
and $(e_j^*)_{j\in [1,n]}$ its dual basis of $T_x^*$. The vectors
$e_{i_1}^*\wedge \ldots\wedge e_{i_p}^*$, for all $ i_1<\ldots<i_p$, is a
basis of  $\EE^p_x$, for all $x \in X$.

\begin{lemma-definition}
There exists
 a unique metric globally defined on the bundle $\EE^p_X$ such
that,  for all $x \in X$, the vectors
$e_{i_1}^*\wedge \ldots\wedge e_{i_p}^*$, $ i_1<\ldots<i_p$, is an orthonormal
basis of  $\EE^p_x$ whenever $e_1, \ldots,e_n$ is an orthonormal basis of $T_x$.
\end{lemma-definition}

Indeed, $e_{i_1}^*\wedge \ldots,\wedge
e_{i_p}^*$ defines a local differentiable section of $\EE^p_X$ due to
the fact that the Gram-Schmidt process of normalization construct
an orthonormal basis varying differentiably in
terms of the point $x \in X$.

\begin{definition} We suppose that the manifold $X$ is oriented.
For each point $x \in X$ let  $e_j, j\in [1,n]$ be an orthonormal
 positively oriented basis of $T_{X,x}$.
 The volume form, denoted $vol$, is defined by the metric as the unique
  nowhere vanishing section  of $\Omega^n_X$,
 satisfying $vol_x:= e_1^*\wedge \ldots \wedge e_n^*$ for each point $x \in X$.
\end{definition}

 \begin{xca} The volume form may be defined directly without
 orthonormalization process as follows.
  Let $x_1, \ldots, x_n$ denote a local ordered set of
 coordinates on an open subset $U$ of a covering of $X$,
 compatible with the orientation.  Prove that
\begin{equation*} \sqrt {det(h_{ij})} dx_1 \wedge \cdots \wedge dx_n
 \end{equation*}
 defines the global volume form.
 \end{xca}
 
\subsubsection{ $L^2$ metric}
For $X$ compact, we deduce from the volume form a global pairing
called the $L^2$ metric
\begin{equation*} \forall \psi, \,\eta \in \EE^i (X), (\psi,
\,\eta)_{L^2}
  = \int_X g_x(\psi(x), \, \eta(x)) vol(x)
 \end{equation*}
  
\subsubsection{Laplacian}
 We prove the existence of an operator $d^*:\EE^{i+1}_X
 \to \EE^{i}_X $ called formal adjoint to the differential
 $d:\EE^i_X \to \EE^{i+1}_X$, satisfying:
 
\begin{equation*}
 (d \psi,  \eta)_{L^2} = (\psi, d^* \eta)_{L^2}, \quad \forall \psi \in
  \EE^i (X), \forall \eta \in \EE^{i+1}(X).
 \end{equation*}
 
 The adjoint operator is defined by constructing first an operator $*$ 
 
\begin{equation*}
 \EE^i_X \xrightarrow{*} \EE^{n-i}_X
 \end{equation*}
 
 by requiring at each point $x \in X$:
 
 \begin{equation*}  
 g_x(\psi(x), \,\eta(x))
 vol_x  = \psi_x \wedge * \eta_x, \quad \forall \psi_x, \,\eta_x \in
 \EE^i_{X,x}.
 \end{equation*}
 
 The section $vol$ defines an isomorphism of bundles
   $\R_X \simeq \EE^n_X $, inducing the isomorphisms:
 $ Hom ( \EE^{i}_X, \R_X) \simeq Hom (\EE^i_X, \EE^n_X) \simeq
 \EE^{n-i}_X$.
 Finally, we deduce $*$  by composition with
  the isomorphism $\EE^i_X
 \rightarrow Hom ( \EE^{i}_X, \R_X) $ defined by the scalar product
 on $\EE^i_X$. 
 Locally, in terms of an orthonormal basis,  the formula apply for  $\psi(x)= e_{i_1}^*\wedge \ldots\wedge e_{i_p}^*  =   \eta (x)$; it shows
   \begin{equation*} 
   *( \psi(x)) = \epsilon e_{j_1}^*\wedge \ldots\wedge e_{j_{n-p}}^*
   \end{equation*}
   where $\{j_1, \ldots, j_{n-p}\} $ is the complement of  $\{i_1, \ldots, i_p\}$ in $[1,n]$ and 
$\epsilon$ is the sign of the permutation of  $[1,n]$ defined by  ${i_1, \ldots, i_p, j_1, \ldots, j_{n-p}}$.  We have on $\EE^i_X$:
 
 \begin{equation*}
 *^2 = (-1)^{i(n-i)}Id,\quad d^* = (-1)^{n+in+1} *\circ d \circ * =  (-1)^i *^{-1}\circ d \circ * 
 \end{equation*}
 
\begin{lemma}
We have for all $\alpha, \beta \in \EE^i (X)$:
 \begin{equation*}  (\alpha,
\,\beta)_{L^2}
  = \int_X \alpha \wedge  * \beta.
 \end{equation*}
 \end{lemma}
 
\begin{definition} The  formal adjoint $d^*$ of $d$ is a differential operator
and the operator  $\Delta$  defined as:
 \begin{equation*}
  \Delta = d^*\circ d + d \circ d^*
 \end{equation*}
 is called the  Laplacian.  Harmonic forms are defined as the solutions of the
 Laplacian and denoted by:
 
 \begin{equation*}
  \HH^i(X) = \{ \psi \in \EE^i(X) : \Delta (\psi) = 0 \}
 \end{equation*}
 \end{definition}
 
  The Harmonic forms depend on the choice of the metric. A fundamental result
  due to Hodge and stated here without proof is that harmonic
  forms  represents the cohomology spaces:
  
\begin{theorem} On a  compact
smooth oriented  Riemannian manifold each cohomology class is
 represented by a unique real harmonic global differential form:
 \begin{equation*}
  \HH^i (X) \simeq H^i (X, \R).
 \end{equation*}
\end{theorem}

\subsection{Complex manifolds and decomposition into types}
 An  underlying  real differentiable structure is attached to a complex manifold $X$, then
 a real and a complex  tangent bundles are defined  and denoted by $T_{X,\R}$
 and $T_X$.  The comparison is possible by embedding $T_X$ into $T_{X,\R} \otimes
 \C$, then 
 a delicate point is to describe holomorphic functions among differentiable functions, and more generally the complex tangent
space and  holomorphic differential forms in terms of the real structure
of the manifold. Such description
is a prelude to Hodge theory and is fundamental in Geometry,
whether we mention Cauchy-Riemann equations satisfied by the real
components of a holomorphic function, or the position of the
holomorphic tangent bundle inside the complexified real tangent
bundle. A weaker feature of a complex manifold is the existence of
an almost complex structure $J$ on the real tangent bundle
which is enough
to define the type of differential forms.

 \subsubsection{Almost-complex structure} Let $V$ be a real vector space.
 An almost-complex structure on $V$ is defined by
  a  linear map $J: V \to V$ satisfying $J^2 = - Id$; then $\dim_{\R} V$ is even
  and the eigenvalues are $i$ and $-i$ with associated eigenspaces
  $V^{+}$ and $V^{-}$ in $V_{\C} = V
  \otimes_{\R} \C$:
  
\begin{equation*}
V^{+} = \{ x - i J x : x \in V\} \subset V_{\C}, \quad V^{-} = \{ x +
i J x : x \in V \} \subset V_{\C}.
\end{equation*}
where we write  $ v$ for  $v \otimes 1$ and $i v$ for the element $v \otimes i$. The
conjugation $\sigma$ on $\C$ extends to a real linear map $\sigma$
on $V_{\C}$  such that $\sigma (a + i b) = a - i b$;
 we write $\sigma (\alpha) = {\overline \alpha}$ for $\alpha \in
V_{\C}$. We note that the subspaces $V^{+}$ and $V^{-}$ are
conjugate: $\overline { V^{+}} = V^{-}$ such that  $\sigma$ induces an anti-linear 
isomorphism $\sigma:  V^{+} \simeq V^{-}$ .

\begin{example}
 Let $V$ be a complex vector space, then
the identity bijection  with the underlying real space $V_{\R}$ of
$V$, $\varphi: V \to V_{\R}$  transports the action of $i$ to a
real linear  action $J$ satisfying $J^2 = - Id$, defining an
almost complex structure on $V_{\R}$. A complex basis 
$ (e_j)_{j\in [1,n]}$ defines a
real basis $(\varphi(e_j), i(\varphi(e_j))$ 
for $  j \in [1, n]$,  such that: $ J(\varphi(e_j) := i \varphi(e_j)$. 
For $V = \C^n$, $\varphi (z_1,\ldots,z_n) = (x_1, y_1, \ldots,  x_n, y_n)$ and 
$$ J (x_1, y_1, \ldots,  x_n, y_n) = (- y_1, x_1, \ldots,   -y_n, x_n).$$
Reciprocally, 
an almost-complex structure $J$ on a real vector space $W$ corresponds to a structure of
complex vector space on $W$ as follows. There exists a family of
vectors ($e_j)_{j\in [1,n]}$  in $W$ such that $\{e_j, J
(e_j)\}_{j\in[1,n]}$ form a basis of $W$. First, one chooses a non-zero
vector $e_1$  in $W$ and proves that the space $ W_1 $ generated by $ e_1, J(e_1)$ 
is of dimension two, then we continue with $e_2 \notin W_1$
and so on; then, there is a complex structure on $W$ for which 
$e_j, j\in [1,n]$ is a complex basis with the action of $i \in \C$
defined by $i . e_j := J (e_j)$.
\end{example}

\begin{remark} An almost-complex structure on a real space $V$ is equivalent to a
complex structure. In what follows we work essentially with
$V_{\C}:= V\otimes_{\R} \C$, then the product  with a complex number $\lambda .(v \otimes
1)= v \otimes \lambda $ will be written $\lambda v$. This product is not to
confuse with the product $\lambda * v $ for the  complex structure
on $V$; in particular, $v\otimes i := i v $ while $i * v := J v $.
We note that the map $\varphi: V \to V^{+}  \subset V_{\C}$
defined by $ x \mapsto  x - i J x $ is an isomorphism of complex
vector spaces.
\end{remark}

\subsubsection{Decomposition into types}
Let $V$ be a complex vector space and $(V_{\R}, J)$ the underlying
  real vector space with its involution $J$ defined by
   multiplication by $i$. The following decomposition of $V_{\C}$
   is called a decomposition into types
\[ V_{\C} \simeq V^{1,0} \oplus V^{0,1} \, \hbox{ where }
V^{1,0} = V^{+}, \, V^{0,1} = V^{-}, \, {\overline  V^{1,0}} =
 V^{0,1} \]
  the conjugation is with respect to the real structure of $V$,
 that is inherited from the conjugation on $\C$.
  Let $W := Hom_{\R} (V_{\R}, \R)$ and $W_{\C} := W \otimes_{\R} \C$;
 a morphism in $Hom_{\C} (V,
\C)$ embeds, by extension to $V_{\C} $, into
 $W_{\C}$ and its  image  is denoted by  $ W^{1,0}$, then $W^{1,0}$ vanishes on
$V^{0,1}$. Let  $W^{0,1}$ be the subspace vanishing on
$V^{1,0}$,  thus
  we deduce from the almost complex structure a decomposition
$W_{\C} \simeq W^{1,0} \oplus W^{0,1}$  orthogonal to the
decomposition of $V_{\C} $.
  Such decomposition extends to the exterior power:
  $$ \wedge^n W_{\C} = \oplus_{p+q=n} W^{pq}, \quad {\rm where}
  \,\, W^{pq} = \wedge^p W^{1,0}\otimes \wedge^q
  W^{0,1}, {\overline  W^{p,q}} =
 W^{q,p}.$$
 
 \begin{remark} On a fixed vector space, the notion of almost-complex 
 structure is equivalent to complex structure. The
  interest in this notion will become clear for a
  differentiably  varying complex structure on the tangent bundle
of a manifold, since not all such structures determine an
analytic complex structure on $M$.  In order to lift an almost-complex
 structure on $T_M$  to  a complex structure
on $M$, necessary and sufficient
conditions on 
are stated in Newlander-Nirenberg
theorem (see \cite{Vo} Theorem 2.24).
\end{remark}

\subsubsection{Decomposition into types on the complexified tangent bundle}\label{type}
 The existence of a complex
structure on the manifold  gives an almost-complex structure on
the differentiable tangent bundle $T_{X,\R}$ and a decomposition
into types of the complexified tangent bundle $T_{X,\R}
\otimes_{\R} \C$ which splits as a direct sum $T_X^{1,0} \oplus
T_X^{0,1}$ where $T_{X,z}^{1,0} = \{u -i Ju: u \in T_{X,\R,z}\}$
and $\overline {T_X^{1,0}} = T_X^{0,1}$. \\
For a complex manifold $X$ with local complex coordinates $z_j =
x_j + i y_j, j \in [1,n],$ we define  $J (\frac{\partial}{\partial
x_j}) =\frac{\partial}{\partial y_j}$ on the real tangent space
$T_{X,\R,z}$ of the underlying differentiable manifold. Since the change of coordinates is
analytic, the definition of $J$ is compatible with the
corresponding change of differentiable coordinates, hence $J$ is
globally  defined on the tangent bundle $T_{X,\R}$. 

The
holomorphic tangent space $T_{X,z}$ embeds isomorphically onto
$T_X^{1,0}$ generated by $ \frac{\partial}{\partial z_j}:=
\frac{1}{2} (\frac{\partial}{\partial x_j} - i
\frac{\partial}{\partial y_j})$ of the form $u_j - i J(u_j)$ such that
 the basis ($\frac {\partial}{\partial z_j}, j\in [1,n]$) in $T_{X,\R,z}\otimes \C$
is complex  dual to ($dz_j = dx_j+i dy_j, j\in [1,n]$). From now on we
identify the two complex bundles $T_X$ and $T_X^{1,0}$ via the isomorphism:
\begin{equation*}
T_{X,\R,z} \xrightarrow {\sim} T_{X,z}^{1,0}: \frac{\partial}{\partial x_j}
\mapsto \frac{\partial}{\partial z_j} = \frac {1}{2} (\frac
{\partial}{\partial x_j} - i \frac {\partial}{\partial y_j}),
 \frac{\partial}{\partial y_j} \mapsto i\frac{\partial}{\partial z_j} = \frac {1}{2} (\frac
{\partial}{\partial y_j} + i \frac {\partial}{\partial x_j}) 
\end{equation*}
which transforms the action of $J$ into the product with $i$ since
$J(\frac{\partial}{\partial x_j}) = \frac{\partial}{\partial
y_j}$. The  inverse is defined by twice the real part: 
$u+iv \to 2u$.

The dual of the above decomposition of the tangent space is
written as: 
\begin{equation*}
\EE^1_X\otimes_{\R} \C \simeq \EE^{1,0}_X \oplus
\EE^{0,1}_X
\end{equation*}
which induces a decomposition of the sheaves of
differential forms into types:
\begin{equation*}
 \EE^k_X\otimes_{\R} \C \simeq \oplus_{p+q=k} \,\EE^{p,q}_X
 \end{equation*}
 In terms of complex local coordinates on an open set $U$, $\phi
 \in \EE^{p,q} ( U )$ is written as a linear combination with differentiable
 functions as coefficients of:
 $$dz_I \wedge d \overline { z }_J:= dz_{i_1}\wedge \cdots \wedge
  dz_{i_p}\wedge d\overline {z}_{j_1}\wedge \cdots \wedge d \overline
  {z}_{j_q} \quad {\rm where} \,\,I = \{ i_1, \cdots, i_p \},
  J = \{ j_1, \cdots, j_q \} .$$
 
\subsubsection{The double complex}  The differential $d$ decomposes  
   as well  into $d= \partial +
   {\overline \partial}$, where:
    \begin{equation*}
\begin{split} \partial:\EE^{p,q}_X \to \EE^{p+1,q}_X: & f dz_I \wedge
 d \overline { z }_J \mapsto \sum_i\frac{\partial}{\partial z_i}
 (f) dz_i \wedge dz_I \wedge d \overline { z }_J, \\
 & \overline \partial:\EE^{p,q}_X \to \EE^{p,q+1}_X: f
dz_I \wedge d \overline { z }_J \mapsto
 \sum_i (\frac{\partial}{\partial \overline z_i})  (f)
 d{\overline z_i} \wedge dz_I \wedge d \overline { z }_J
 \end{split}
\end{equation*}
are compatible with
  the decomposition up to a shift on the bidegree, we deduce from
  $d^2 = \partial^2 + {\overline \partial}^2 +
 \partial\circ {\overline \partial} + {\overline \partial}\circ
  \partial = 0 $ that $ \partial^2 = 0$, $ {\overline \partial}^2  = 0$ and $
 \partial\circ {\overline \partial} + {\overline \partial}\circ
  \partial = 0 $. It follows that $(\EE_X^{*,*}, \partial,
 \overline{\partial})$ is a double complex (see \cite{g-h} p 442).
 Its
associated simple 
 complex, denoted $s(\EE_X^{*,*}, \partial,
 \overline{\partial})$, is isomorphic to
   the de Rham complex:
 \begin{equation*}
 (\EE^*_X\otimes_{\R} \C,d)\simeq s(\EE_X^{*,*}, \partial,
 \overline{\partial}).
\end{equation*}
In particular we recover the notion of the subspace  of cohomology
of type $(p, q)$:
\begin{definition} The subspace of de Rham cohomology represented
by a global  closed form of type ($p,q$) on $X$ is denoted as:
\begin{equation*}
 H^{p,q}(X):= \{ {\rm cohomology \, classes \, representable \, by \,  a
 \, global \, closed \, form \, of \, type}\,\, (p,q)\, \}
\end{equation*}
\end{definition}
\noindent We have: $ \overline{ H^{p,q}(X)} = H^{q,p}(X)$.

\begin{remark}
 The definition
shows that $\sum_{p+q = i} H^{p,q}(X) \subset H^i_{DR}(X)$.
Although the differential forms decompose into types, it does not
follow that the cohomo\-logy also decomposes into a direct sum of
subspaces $H^{p,q}(X)$ , since the differentials
are not compatible with types.  It is striking that there is a
 condition, easy to state,  on K\"{a}hler manifolds, that imply
  the full decomposition. Hodge theory on such
manifolds shows that: 
$$H^i_{DR}(X) = \oplus_{p+q = i} H^{p,q}(X).$$
\end{remark}

\begin{example}
 Let $\Lambda \subset \C^r$ be a rank-$2r$ lattice
 \[ \Lambda = \{ m_1 \omega_1 + \ldots+  m_{2r} \omega_{2r};
  \quad m_1,\ldots, m_{2r} \in \Z\}\]
  where $\omega_i$ are complex vectors in $\C^r$, linearly
  independent over $\R$. The quotient group $ T_{\Lambda} = \C^r /\Lambda$
  inherits a complex structure locally
 isomorphic to $\C^r$, with change of coordinates induced by
 translations by elements of $\Lambda$.  Moreover $ T_{\Lambda}$
 is  compact. The quotient $ T_{\Lambda} = \C^r /\Lambda$ is a compact
 complex analytic manifold called a complex torus. The projection
 $\pi: \C^r \to T_{\Lambda}  $
 which is holomorphic and locally  isomorphic, is in fact a
 universal covering. 
 
 \vskip.1in 
 \noindent{\it Family of tori.} The above torus depends on the choice of the lattice.
  Given two lattices ${\Lambda_1}$ and $ {\Lambda_2}$, there exists
  an $\R-$linear isomorphism $g \in GL (2r,\R) $ such that
  $g({\Lambda_1}) = {\Lambda_2}$, hence it induces a diffeomorphism
  of the underlying differentiable structure of the tori
  $g: T_{\Lambda_1} \simeq T_{\Lambda_2}$. However, it is not true
  that all tori are isomorphic as complex manifolds.
  
  \begin{lemma} Two complex tori $ T_{\Lambda_1}$ and $ T_{\Lambda_2}$
  are isomorphic if and only if there exists a complex linear map
  $g \in GL ( r,\C)$  inducing an isomorphism of lattices $g:
  {\Lambda_1}\simeq {\Lambda_2}$.
  \end{lemma}
  
\begin{proof} Given two complex torus $ T_{\Lambda_1}$ and $ T_{\Lambda_2}$,
 an analytic morphism $f: T_{\Lambda_1}\to T_{\Lambda_2}$ will
 induce a morphism on the fundamental groups and there exists an holomorphic
  lifting to the
 covering  $F : \C^r \to \C^r$  compatible with
 $f$ via the projections: 
 $$\pi_2 \circ F = f \circ \pi_1.$$ 
 Moreover, given a point $z_1 \in \C^r$
 and $\lambda_1 \in \Lambda_1$, there exists $\lambda_2 \in \Lambda_2$
 such that  $F(z_1 + \lambda_1) = F(z_1) + \lambda_2$. 
 
 By continuity, this relation
 will remain true by varying $z \in \C^r$; then $F$ and its
derivative are periodic. This derivative depends holomorphically
on parameters of $T_{\Lambda_1}$, hence  it is necessarily constant; then, $F$ is
 complex affine on $\C^r$, since its derivative is constant.
We may suppose that $F$ is complex linear after composition with
a translation.
\end{proof}

\noindent{\it Moduli space}. We may parameterize all lattices as follows:\\
- the group $GL(2r, \R)$ acts transitively on the set of all lattices of $\C^r$.\\
We choose   a basis $\tau = (\tau_1,
\ldots,\tau_{2i-1},\tau_{2i}\ldots\tau_{2r}),$  $i \in [1, r]$, of
a lattice $L$, then it defines a basis of $\R^{2r}$ over $\R$. An element $\varphi$ of 
$GL(2r, \R)$ is given by the linear transformation which sends $\tau$ into
the basis $\varphi(\tau)=\tau'$ of $\R^{2r}$ over $\R$. The
element $\varphi$ of $GL(2r, \R)$ is carrying the lattice $L$ onto 
the lattice $L'$ defined by the $\tau'$.\\
- The isotopy group of this action is $GL(2r, \Z)$, since 
$\tau$ and $\tau'$ define the same lattice if and only if
$\varphi\in GL(2r, \Z)$.\\
Hence the space of lattices is the quotient group $GL(2r, \R)/GL(2r, \Z)$.\\
- Two tori defined by the lattice $L$ and $L'$ are analytically isomorphic 
if and only if there is an element of $GL(r,\C)$ which transform 
the lattice $L$ into the lattice $L'$, as the preceding lemma states.

It follows
that the parameter space of complex tori is the quotient: 
$$GL(2r,
\Z) \backslash GL(2r, \R)/ GL(r, \C)$$ 
where $GL(r, \C)$ is
embedded naturally in $GL(2r, \R)$ as a  complex linear map is
$\R-$linear. 

For $r =1$, the quotient $GL(2, \R)/ GL(1, \C)$ is
isomorphic to $\C - \R $, since, up to complex isomorphisms, a
lattice is generated by $1, \tau \in \C  $ independent over $\R$,
hence
 completely determined by $ \tau \in \C - \R $.
The moduli space is the orbit space of the action
of $GL(2,\Z)$ on the space $GL(2, \R)/ GL(1, \C)=\C - \R $.
Since $GL(2,\Z)$ is the disjoint union of
$SL(2,\Z)$ and the integral $2\times 2$-matrices of determinant
equal to $-1$, that orbit space is the one of the action of
$SL(2,\Z)$ on the upper half plane:
 $$ ( \left(\begin{array}{cc}
  a & b \\
  c & d \\
\end{array} \right), z) \mapsto \frac{az+b}{cz+d}.$$

 \vskip 0,2cm
 
\n {\it Hodge decomposition on complex tori}. The cohomology of a
complex torus $T_{\Lambda}$ is easy to compute by K\"{u}nneth
formula, since it is diffeomorphic to a product of circles:
$T_{\Lambda} \simeq (S^1)^{2r}$. Hence $H^1(T_{\Lambda}, \Z)
\simeq \Z^{2r}$ and $H^j(T_{\Lambda}, \Z) \simeq \wedge^j
H^1(T_{\Lambda}, \Z) $.

The cohomology with complex coefficients
may be computed by de Rham cohomology. In this case,  since the
complex tangent space is trivial, we have natural  cohomology
elements given by the translation-invariant forms, hence with
constant coefficients.
 The finite complex vector space $T_0^*$ of constant holomorphic $1-$forms
 is isomorphic to $\C^r$ and generated by
 $d z_j, j \in[1,r]$  and the Hodge  decomposition reduces to prove
 $H^j(X, \C)\simeq  \oplus_{p+q = j} \wedge^p T_0^*\otimes \wedge^q
  \overline {T_0^*}, p \geq 0, q \geq 0 $ which is here a consequence
  of the above computation of the cohomology.
\end{example}
\subsubsection{Dolbeault cohomology}
\hbox{}
\

\noindent For  $r \geq 0$, the complex $(\EE_X^{r,*}, \overline{\partial})$, is 
called the $r$-Dolbeault complex.
 \begin{definition} The  $i$-th cohomology of the  global
sections of the $r$-Dolbeault complex is called the
$\overline{\partial}$ cohomology of $X$ of type $(r,i)$
\begin{equation*} H_{\overline{\partial}}^{r,i}(X):=
  H^i(\EE_X^{r,*}(X), \overline{\partial}) =
  \frac { Ker (\overline{\partial}^{i}: \EE^{r,i}(X)\to \EE^{r,i+1}(X))}
 {Im (\overline{\partial}^{i-1}: \EE^{r,i-1}(X)\to \EE^{r,i}(X))}
\end{equation*}
\end{definition}
On complex manifolds, we consider the sheaves of holomorphic forms $\Omega^r_X :=
\wedge \Omega1_X $.
\begin{lemma}\label{dol}
 The Dolbeault complex $(\EE_X^{r,*}, \overline{\partial})$ for
$ r\geq 0$, is a fine resolution of $\Omega^r_X$, hence
 \begin{equation*}
 H^i(X,\Omega^r_X) \simeq  H^i(\EE_X^{r,*}(X),
  \overline{\partial}) := H_{\overline{\partial}}^{r,i}(X)
\end{equation*}
 \end{lemma}

For the proof, see the $\overline{\partial}$-Poincar\'e lemma in
\cite{g-h} p. 25 (continued p. 45). \\
A cohomology class of $X$ of type $(r,i)$ defines a
$\overline{\partial}-$cohomology class of the same type, since ($
d \varphi = 0) \Rightarrow (\overline{\partial}\varphi = 0)$. 

On a
K\"{a}hler manifold (see below in \ref{Kahler}) 
we can find a particular representative
$\omega$ in each class of $\overline{\partial}$-cohomology, called
harmonic form for which: 
$$(\overline{\partial}\omega = 0)
\Rightarrow (d \omega = 0).$$

\subsubsection{Holomorphic Poincar\'e lemma} 
 The holomorphic version of Poincar\'e lemma
 shows that the de Rham complex of
holomorphic forms $\Omega^*_X$ is a resolution of the constant
sheaf $\underline{\C}_X$. However, since the resolution is not
acyclic, cohomology spaces are computed only as
 hypercohomology of the global sections functor, that is
  after acyclic resolution.
 \begin{equation*}
H^i(X,\C) \simeq \H^i(X,\Omega^*_X):= R^i\Gamma
(X,\Omega^*_X)\simeq H^i( \EE^*(X)\otimes_{\R}\C )
\end{equation*}

\subsection{Hermitian metric, its associated Riemannian metric and $(1,1)$ form}
 The Hermitian product $h$ on the complex manifold $X$
 is defined at each point $z$ as a product on the holomorphic
 tangent space  $T_z$  to $X$ at $z$
\begin{equation*}
 h_z: T_z \times T_z \to \C
\end{equation*}
 satisfying $h(u,v) =  \overline {h(v,u)}$, linear in $u$ and anti-linear in $v$, moreover
 $h_z$ vary smoothly with $z$. In terms of the basis
 $(\frac { \partial}{\partial z_j}, j \in [1,n])$ dual to
 the basis $(dz_j,j \in [1,n])$ of $T_z^*$,
   the matrix $h_{ij}(z):= h_z(\frac {
\partial}{\partial z_i} , \frac { \partial}{\partial z_j} ) $
is differentiable in $z$ and $h_{ij}(z)= \overline {h_{ji}(z)}$.
 In terms of {\it the above isomorphism  $T_z  \xrightarrow
 {\sim} T^{1,0}_z$} defined by $ \frac{\partial}{\partial z_j} = \frac {1}{2} (\frac
{\partial}{\partial x_j} - i \frac {\partial}{\partial y_j})$,
  it is
equivalent to define the Hermitian product as
\begin{equation*}
 h_z:T^{1,0}_z \times T^{1,0}_z
 \to \C,\,\, h_z = \sum_{i,j} h_{ij}(z) dz_i\otimes d \overline {z_j}
 \in  \EE^{1,0}_z
\otimes_{\C} \EE^{0,1}_z
\end{equation*}
where the Hermitian product on $T^{1,0}_z$ is viewed as a linear form on 
$ T^{1,0}_z \otimes \overline {T^{1,0}_z}$. The product is called a Hermitian metric if moreover the real
number $h_z(u, \overline {u}) > 0$ is positive definite for all $u
\neq 0\in T_z$; equivalently  the matrix $h$  of the product defined by
$h_{ij}(z)$ (equal to its conjugate transpose: $h = {^t{\overline h}}$ is positive
definite: $ {^t{\overline u}} h u > 0 $.

\subsubsection{ Associated Riemannian structure} The Hermitian metric may be defined on
 $T_{X,\R,z}$ identified with $ T^{1,0}_z$ via  $\frac { \partial}{\partial x_i} \to 
 \frac { \partial}{\partial z_i}$ and extended via the complex structure s.t. 
 $\frac { \partial}{\partial y_i} = J (\frac { \partial}{\partial x_i}) :=  i \frac { \partial}{\partial  x_i} \to 
 i \frac { \partial}{\partial z_i} $, then using the decomposition  
  as the sum of its
real and imaginary parts
\begin{equation*}
 h_{z |T_{X,\R,z}} = Re \,( h_z) + i Im \,( h_z)
\end{equation*}
we associate to the metric  a Riemannian structure on $X$
 defined by $g_z:= Re \,( h_z) $.
 
 In this way, the complex manifold $X$ of dimension $n$ is viewed as a Riemann manifold
 of dimension  $2n$ with metric $g_z$.
 Note that $g$ is defined over the real numbers, since it is represented by a
  real matrix over a basis of the real tangent space.
  
  Since $h_z$ is hermitian, the metric satisfy:
  $g_z (Ju,Jv)= g_z (u,v)$. In fact we can recover the Hermitian product from
  a Riemannian structure satisfying this property.\\
  From now on, the Hodge operator star will be defined on $T_{X,\R,z}$ with respect to $g$
  and extended complex linearly to $T_{X,\R,z}\otimes  \C$.
  
\begin{lemma-definition}[ The (1,1)-form $ \omega $]
The hermitian metric defines  a
 real exterior $2-$form
 $ \omega := - \frac { 1}{2} Im \, ( h) \,$ of type ($1,1$) on  $\, T_{X,\R,z}\,$,
 satisfying:\\
 $\omega ( Ju,v) = - \omega (u,J v)$ and $  \omega ( Ju, Jv) =  \omega (u, v)$; moreover: $2 \omega (u,J v) =  g(u,v)$ and 
$$ \omega  =
\frac { i}{2} \sum_{i,j} h_{ij}(z) dz_i\wedge d \overline {z_j} \quad
\hbox{if} \quad  h_z = \sum_{i,j} h_{ij}(z) dz_i\otimes d \overline {z_j}
$$.
 \end{lemma-definition}
 
 \begin{proof}
 We have an exterior $2$-form since $2 \omega (v,u) = - Im \, h (v,u) =
  Im \, h (u, v)
 = - 2 \omega (u,v)$. We check:
$2 \omega (u,J v) = - Im \, h (u,i v)= - Im \,(-i h (u, v)) = Re
\,(h (u,v) = g(u,v)$, and $ 2 \omega (J u, v) = - Im \, h (i u, v) =
- Im \, (ih ( u, v)) = - Re\, h ( u, v)= - g(u,v)$,
while $Im \, h = \frac { 1}{2i} \sum_{i,j} h_{ij}(z) (dz_i\otimes d \overline {z_j}
-  d\overline {z_j}\otimes d z_i) = -i \sum_{i,j} h_{ij}(z) dz_i\wedge d \overline {z_j}$
 \end{proof}
\begin{example}
 1) The Hermitian metric on $\C^n$ is defined by
 $h:= \sum_{i=1}^n dz_i\otimes d \overline {z_i}$:
\begin{equation*}
 h(z,w)= \sum_{i=1}^n z_i {\overline w}_i
\end{equation*}
  where $ z = (z_1, \ldots , z_n)$ and  $ w = (w_1, \ldots , w_n)$.
Considering the underlying real structure on  $\C^n \simeq
\R^{2n}$ defined by $(z_k = x_k + i y_k) \mapsto (x_k, y_k)  $,
 $Re \,h$ is the Euclidean inner product on $\R^{2n}$ and $Im \, h
$ is a non-degenerate
alternating bilinear form, i.e., a symplectic form: $\Omega =  \frac{i}{2}
 \sum_{i=1}^n dz_i\wedge d \overline {z_i} = \sum_{i=1}^n dx_i\wedge d y_i$.

Explicitly, on $\C^2$, the standard Hermitian form is expressed on $\R^4$ as:
\begin{equation*}
\begin{split}
 h((z_{1,1}, z_{1,2}),(z_{2,1}, z_{2,2}))= & x_{1,1}x_{2,1}+
 x_{1,2}x_{2,2}+ y_{1,1}y_{2,1} + y_{1,2}y_{2,2}\\
&+ i (x_{2,1}y_{1,1}- x_{1,1}y_{2,1}
 + x_{2,2}y_{1,2}-x_{1,2}y_{2,2}).
 \end{split}
\end{equation*}
 2) If $\Lambda \subset \C^n$ is a full lattice, then the
 above metric induces a Hermitian metric on the torus
$\C^n/\Lambda$.

\n 3) The {\it Fubini -Study metric on the projective space $\P^n$}.
Let $\pi:\C^{n+1}-\{0\} \to \P^n$ be the  projection  and consider a
section of $\pi$ on an  open subset  $Z: U \to \C^{n+1}-\{0\}$,
then the form $\omega =\frac{i}{2\pi}
\partial {\overline \partial} log \vert \vert Z \vert \vert ^2$
is a  fundamental $(1,1)$ form, globally defined on $\P^n$, since the forms 
glue together on the intersection of two open subsets (this formula is related to 
the construction of the first Chern class of the canonical line bundle on the projective space,
in particular the coefficient $\frac{i}{2}$ leads later to a positive definite form, while $\pi$
leads to an integral cohomology class). 
 For the
section $Z=(1, w_1,\ldots,w_n)$, 
\begin{equation*}
\omega =
\frac{i}{2\pi}\frac{(1+\vert w \vert^2)\sum_{i=1}^n d w_i\wedge d \overline {w_j} - 
\sum_{i,j =1}^n  \overline {w_j}  w_i d w_i\wedge d \overline {w_j}}
{(1+\vert w \vert^2)^2}
\end{equation*}
which reduces at  the point $ O =(1, 0,\ldots, 0)$,  to $\omega =
\frac{i}{2\pi}\sum_{i=1}^n d w_i\wedge d \overline {w_i}
> 0$, hence  a positive real form , i.e., it takes real values on real tangent vectors, and
is associated to a positive Hermitian form. 
Since $\omega$ is invariant under the transitive holomorphic action of $SU(n+1)$ on $\P^n$,
we deduce that the hermitian form 
\begin{equation*}
h =
(\sum_{i,j =1}^n h_{i,j}(w) d w_i \otimes d \overline {w_i})(1+\vert w \vert^2)^{-2}, \, 
\pi h_{i,j}(w) = (1+\vert w \vert^2)\delta_{ij} -  \overline {w_j} w_i
\end{equation*}
 is a positive definite Hermitian form. Hence
the projective space has a Hermitian metric.
\end{example}

\subsection{Harmonic forms on compact complex manifolds} In this
section, a study of the $\overline{\partial} $ operator, similar
to the previous study of $d$,  leads to the representation of
Dolbeault cohomology by $\overline{\partial}$-harmonic forms 
(Hodge theorem).
\subsubsection{ $L^2$-inner product } We choose a Hermitian metric $h$ with associated $(1,1)$-form 
 $\omega$ and  volume form $vol = \frac{1}{n!} \omega^n$.
 We extend  the $L^2$-inner product  $(\psi,\eta)_{L^2}$ ( resp. the Hodge  star * operator) defined on the real forms by  the underlying Riemannian metric,
 first linearly into a complex bilinear form  $<\psi,\eta>_{L^2}$ on  $\EE^*_X:= \EE^*_{X, \R} \otimes \C $ 
 ( resp.  $ \tilde *: \EE^p_{X, \R} \otimes \C \to \EE^{2n-p}_{X, \R} \otimes \C$).
 However, it is more appropriate to introduce the conjugate  operator:
\begin{equation*}
*: \EE^{p,q}_X \to \EE^{n-p, n-q}_X, \quad   * (\varphi):= { \tilde *} (\overline \varphi)  \end{equation*}
and then we consider the Hermitian product : 
 $ (\psi,\eta)_{L^2} := <\psi,\overline \eta>_{L^2} $, hence:
 $$  (\psi,\eta)_{L^2} =  \int_X \psi \wedge * \eta, \,  \psi , \eta \in \EE^p (X), \,
 \hbox{and} \, (\psi,\eta)_{L^2} =  0  \,
 \hbox{if} \,  \psi  \in \EE^p (X),  \eta \in \EE^q (X), p \not= q. $$
 This definition enables us to obtain a Hermitian metric on the complex vector space of global forms 
 $\EE^* (X)$ on $X$,  out of the real metric since we have $ v \wedge *
   v = \sum_{| I | = p } |v_I |^2 *(1) $ where $*(1) = vol, v = \sum_{| I | = p } v_I e_I $ for 
   ordered multi-indices $ I = i_1 < \ldots < i_p $ of an orthonormal basis $e_i$ , hence the integral vanish iff all $v_I $ vanish everywhere (\cite{wel}  p.$174$). 
With respect to such Hermitian metric,
 the adjoint operator to  $ \overline{\partial}$ is defined:
  \begin{equation*}\overline{\partial}^*:\EE^{p,q}_X \to
  \EE^{p,q-1}_X, \, \,  \overline{\partial}^* = - * \circ\overline{\partial}\circ * \, : \, (\overline{\partial}^*\psi, \eta)_{L^2} = (\psi,
\overline{\partial} \eta)_{L^2}.
\end{equation*}
 Then, we introduce  the $\overline{\partial}$-Laplacian:
 \begin{definition} The $\overline{\partial}$-Laplacian  is defined as:
\begin{equation*}
  \Delta_{\overline{\partial}} = \overline{\partial}\circ
  \overline{\partial}^* + \overline{\partial}^* \circ \overline{\partial}
 \end{equation*}
 Harmonic forms of type $(p,q)$ are defined as the solutions of the
 $\overline{\partial}$-Laplacian
 \begin{equation*}
  \HH_{\overline{\partial}}^{p,q}(X) = \{ \psi \in \EE^{p,q}(X)\otimes \C:
   \Delta_{\overline{\partial}} (\psi) = 0\}.
 \end{equation*}
 \end{definition}
 A basic result, which  proof involves deep analysis in Elliptic operator theory and it is well documented
 in various books, for instance   (\cite{wel}  chap. IV, sec. 4) or   \cite{g-h}, p. 82 - 84, is stated here without proof:
 \begin{theorem}[Hodge theorem] On a  compact
complex  Hermitian manifold each $\overline{\partial}$-cohomology
class of type $(p,q)$ is represented by a unique
$\overline{\partial}$-harmonic global
 complex differential form of type
$(p,q)$:
 \begin{equation*}
 H_{\overline{\partial}}^{p,q}(X)  \simeq   \HH_{\overline{\partial}}^{p,q}(X)
 \end{equation*}
 moreover the space of ${\overline{\partial}}$-harmonic forms is
 finite dimensional, hence the space of
 ${\overline{\partial}}$-cohomology also.
\end{theorem}
\begin{corollary} For a  compact
complex  Hermitian manifold:
 \begin{equation*}
 H^q(X,\Omega^p_X)\simeq \HH_{\overline{\partial}}^{p,q}(X)
\end{equation*}
\end{corollary}
In general there is a spectral sequence relating Dolbeault and de Rham cohomology groups \cite{Fr} and \cite{Del-Ill} . It will follow from Hodge theory  that this spectral sequence  degenerates at rank $1$.

\subsection{K\"{a}hler manifolds}\label{Kahler} 
The exterior product of harmonic forms is not in general harmonic
neither the restriction of harmonic forms to a submanifold is
harmonic for the induced metric. In general, the
$\overline{\partial}$-Laplacian $\Delta_{\overline{\partial}}$ and
the Laplacian $\Delta_{d}$ are not related. The theory  becomes
 natural when we add the K\"{a}hler condition on the metric, when the fundamental form 
 is  closed  with respect to the differential $d$;  we refer then to (\cite{wel}, ch. V, section 4) to
 establish the relation between 
 $\Delta_{\overline{\partial}}$ and $\Delta_{d}$ and for full proofs in general,
most importantly the type  components of harmonic forms  become
harmonic. The proofs in Hodge theory involve the theory of
elliptic differential operators on a manifold and will not be
given here as we aim just to give the statements.

\begin{definition} The Hermitian metric   is K\"{a}hler
if its associated $(1,1)$-form $\omega$ 
is closed: $d \omega = 0$. In this case, the manifold $X$
 is called a K\"{a}hler manifold.
 \end{definition}
 
\begin{example}\label{examples}
1) A Hermitian metric on a compact Riemann surface (of complex dimension $1$)
 is K\"{a}hler
since  $d \omega$ of degree $ 3 $ must vanish. \\
2) The Hermitian metric on a compact complex  torus is K\"{a}hler. However, in
  general a complex tori is  not a projective variety. Indeed, by the projective embedding theorem of
  Kodaira (see \cite{Vo} Theorem 7.11 p.164), the cohomology must be polarized, or the K\"{a}hler form  is integral in order to get an algebraic torus,
   called also an  abelian variety since it is a commutative group.   This will be the case for $r = 1$, indeed a
  complex tori of dimension $1$  can be always embedded as a cubic curve in the
  projective plane via Weierstrass function and its derivative.\\ 
3) The projective space with the Fubini-Study metric is K\"{a}hler.\\
4) The restriction of a K\"{a}hler metric on a submanifold is
K\"{a}hler with  associated $(1,1)$-form  
induced by the  associated $(1,1)$-form 
on the ambient manifold.\\
5) The product of two K\"{a}hler manifolds is K\"{a}hler.\\
6)  On the opposite,  a Hopf surface is an example of a compact complex smooth surface which is not K\"{a}hler. Indeed such surface is the orbit space of the action of a group isomorphic to $\Z$ on $\C^2 - \{0\}$ and generated  by the automorphism:
  $(z_1,z_2) \to (\lambda_1 z_1, \lambda_2 z_2)$ where  $\lambda_1,\lambda_2
 $ are complex  of module  $< 1$ (or $ > 1$). 
 
 For example, consider 
 $ S^3 = \{z = (z_1,z_2) \in \C^2 : \vert  z_1 \vert ^2 + | z_2 |^2 = 1 \} $ and  consider
  $$   
 f: S^3 \times \R \simeq  \C^2 - \{ 0 \}: f(z_1,z_2,t) = ( e^t z_1,e^t z_2)
 $$
then the action of $\Z$ on $\R$  defined for $m \in \Z $ by $t \to t+m$ is transformed into the action by $ (\lambda_1^m, \lambda_2^m) = (e^m, e^m)$. We deduce that  the action has no fixed point and the quotient surface is homeomorhic to $S^3 \times S^1$, hence compact. Its fundamental group is isomorphic to $\Z$ as the surface is a quotient of a simply connected surface and  its cohomology in degree $1$ is also  isomorphic to $\Z$,
  while the cohomology of a K\"{a}hler variety in odd degree must have an even dimension as a consequence of the conjugation property in the Hodge decomposition. 
 \end{example}
 On
 K\"{a}hler manifolds, the following key relation between  Laplacians  is just  stated here (see \cite{wel} for a proof, or   \cite{g-h} p. 115):
\begin{lemma}
 \begin{equation*}
  \Delta_{d} = 2 \Delta_{\overline{\partial}} = 2 \Delta_{\partial}
  \end{equation*}
  \end{lemma}
 
 We denote by $ \HH^{p,q}(X )$ the space of harmonic forms of type
  ($p,q$), that is the global sections  $\varphi  \in \EE^{p,q}_X
  (X)$ such that $\Delta_{d}(\varphi) = 0$, which is equivalent to
  $\Delta_{\overline{\partial}} \varphi = 0 $ 
as a consequence of  the lemma.\\
  Let  $\varphi := \sum_{p+q =i}\varphi^{p,q} \in \oplus_{p+q =i}
  \EE^{p,q}(X)$ be harmonic,
 \[\Delta_{d}\varphi = 2 \sum_{p+q
 =r}(\Delta_{\overline{\partial}}\varphi^{p,q}) = 0 \in \oplus_{p+q =r}
  \EE^{p,q}(X)\Rightarrow \forall p,q, \, 2\Delta_{\overline{\partial}}
  \varphi^{p,q} = \Delta_{d} \varphi^{p,q} = 0\]
 since $\Delta_{\overline{\partial}}$ is compatible with type.
  Hence, we deduce:
  
\begin{corollary} i) The $\overline{\partial}$-harmonic forms of type
 ($p,q$) coincide with the harmonic forms of same type:
  $ \HH^{p,q}_{\overline{\partial}}
  (X) =\HH^{p,q}(X )$.\\
  ii) The projection on the $(p,q)$ component of
an harmonic form is harmonic and we have a natural decomposition:
\begin{equation*}
   \HH^r(X)\otimes\C = \oplus_{p+q=r} \HH^{p,q}(X), \,
   \HH^{p,q}(X)= \overline {\HH^{q,p}(X)}.
  \end{equation*}
\end{corollary}

\n  Since $\Delta_d $   is real, we deduce the conjugation
property.

\subsubsection{Hodge decomposition} 

Let  $H^{p,q}(X)$ denotes the Dolbeault groups 
represented by ${\overline{\partial}}-$closed $(p,q)-$ forms. In general, such forms need not 
to be $d-$closed, and reciprocally the $(p,q)$ components of a $d-$ closed  form are not necessarily ${\overline{\partial}}-$closed nor $d-$ closed, however this cannot happens on compact K\"{a}hler manifolds.
\begin{theorem} Let $X$ be a compact K\"{a}hler
manifold. There is an isomorphism of cohomology classes of type
$(p,q)$ with harmonic forms of the same type
\begin{equation*}
     H^{p,q}(X) \simeq \HH^{p,q}(X).
 \end{equation*}
 \end{theorem}
Indeed,  let  $ \varphi = \varphi^{r,0} + \cdots +\varphi^{0,r} $, then $\Delta_{\overline{\partial}} \varphi =   \Delta_{\overline{\partial}} \varphi^{r,0} + \cdots + \Delta_{\overline{\partial}}\varphi^{0,r} =  0$ is equivalent to  $\Delta_d  \varphi^{p,q} = 0$ for each $p+q = r$,  since 
$\Delta_d = 2 \Delta_{\overline{\partial}}$.  
 
 \subsubsection{Applications to Hodge theory} 
 
 We remark first the
 isomorphisms:
 \begin{equation*}
 H^{p,q}(X)\simeq H_{\overline{\partial}}^{p,q}(X)\simeq H^q(X, \Omega_X^p)
 \end{equation*}
 which shows in particular for $q=0$ that the space of global  holomorphic
 $p$-forms  $H^0(X, \Omega_X^p)$ injects into $ H^p_{DR}(X)$
with image  $H^{p,0}(X)$.\\
 - Global holomorphic $p$-forms are closed and harmonic for any K\"{a}hler
  metric. They vanish if and only if they are exact.\\
 - There are no non-zero global holomorphic forms on $\P^n$,
 besides locally constant functions,  and
  more generally:
  \begin{equation*}
 H_{\overline{\partial}}^{p,q}(\P^n)\simeq H^q(\P^n, \Omega_{\P^n}^p)=
 \begin{cases}
  0, & \mbox{if } p\neq q \\ \C & \mbox{if }p = q\end{cases}
 \end{equation*}
 
\begin{remark} i) The holomorphic form $z\,dw$ on $\C^2$ is not closed.\\
ii) The spaces $H^{p,q} $ in $H_{DR}^{p+q}$ are isomorphic to the
holomorphic invariant $H^q( X, \Omega^p)$, but the behavior
relative to  the fiber of a proper holomorphic   family $X \to T$
for $t \in T$ is different for each space: $H^q( X_t, \Omega^p)$
is holomorphic in $t$, de Rham cohomology $H_{DR}^{p+q}(X_t)$ is
locally constant but the embedding of $H^{p,q}(X_t) $ into de Rham
cohomology $H_{DR}^{p+q}(X_t)$ is not holomorphic.
\end{remark}

\begin{xca} Let $X$ be a compact K\"{a}hler
manifold and $\omega :=
 - Im \, h $ its  real $2$-form of type $(1,1)$
 ($\omega \in \EE_X^{1,1}\cap \EE^2_X$).
 The  volume form $vol \in \EE^{2n}(X)$ defined by $g $ on $X$,
  can be defined by $\omega$ and
 is equal to $\frac{1}{n!} \omega^n$.\\
 Indeed, if we consider an orthonormal  complex basis $e_1, \ldots,e_n$ of the tangent space 
 $\, T_{X,\R,x}\,$,  then $e_1, Ie_1, \ldots,e_n, I e_n$  is a real 
 oriented orthonormal basis for $g$
 and we need to prove  $\frac{1}{n!} \omega^n (e_1 \wedge I e_1\wedge \cdots \wedge I e_n ) = 1$.\\
  Let $ \varphi_i, i \in [1,n]$ denote a local  basis of the  complex cotangent bundle $T^{*(1,0)}$,
  unitary for the Hermitian metric, then it is orthogonal  and $\| \varphi_j\|^2 = 2$ with respect to the Riemannian metric (for example on $\C^n$, the complex unitary basis $dz_j$ is written as $ dx_j + i dy_j$  with $ dx_j $ and $dy_j$ elements of an orthonormal basis for the Riemannian structure). We deduce from the decomposition 
 $\varphi_j = \psi'_j + i \psi''j$ into real components, the following formula of the volume:
\begin{equation*}
\omega = \frac{i}{2}\sum_j \varphi_j \wedge {\overline  \varphi}_j, \quad
vol = \frac {\omega^n}{n!}  =  \psi'_1 \wedge \psi''_1 \wedge \cdots
\wedge  \psi'_n \wedge \psi''_n.
\end{equation*}
since $ \varphi_j \wedge {\overline  \varphi}_j = (\frac{2}{i}) \psi'_j \wedge \psi''_j$. In particular, $\omega^n$ never vanish.
\end{xca}

\begin{lemma}
For a  compact K\"{a}hler manifold $X$:
$$H^{p,p}(X,\C) \not= 0\quad\hbox{ for } 0 \leq p \leq \dim X.$$
\end{lemma}

In fact, the integral of the volume form  $\int_X \omega^n > 0$.
It follows that the cohomology class  $\omega^n \not= 0 \in
H^{2n}(X,\C)$, hence the cohomology class  $\omega^p \not= 0 \in
H^{p,p}(X,\C)$ since its cup product with $\omega^{n-p}$ is not
zero.

\subsection{Cohomology class of a subvariety and Hodge conjecture}
To state the Hodge conjecture,
 we construct  the  class in de Rham cohomology
 of a closed complex algebraic
subvariety (resp. complex analytic subspace) of codimension $p$ in
a smooth complex projective variety (resp. compact K\"{a}hler
manifold); it is of Hodge type ($p,p$).

\begin{lemma}
  Let  $X$  be a  complex manifold and $Z$  a compact complex
  subanalytic space of dim $m$ in $X$. The integral  of the
  restriction of
a form $\omega$ on $X$ to the smooth
  open subset of $Z$ is convergent and  defines a linear function
  on forms of degree $2m$. It induces a linear map on cohomology
\[cl (Z): H^{2m}(X, \C) \to \C, \quad [\omega] \mapsto
\int_{Z_{smooth}} \omega_{|Z }\] 
Moreover, if $X$ is compact K\"ahler,  $cl(Z)$ vanish on all
components of the Hodge decomposition which are distinct from $H^{m,m}$.
\end{lemma}

If $Z$ is compact and smooth, the integral is well defined, because, by
Stokes theorem, if $ \omega = d \eta$, the integral vanish since
it is equal to the integral
   of $\eta$ on the boundary $\partial Z = \emptyset$ of $Z $.
   
If $Z$ is not smooth, the easiest proof is to use a
desingularisation (see \cite{Hi}) $\pi: Z' \to Z$ inducing an
isomorphism $Z'_{smooth} \simeq Z_{smooth}$ which implies an equality
of the integrals of $\pi^* (\omega_{|Z })$ and $\omega_{|Z }$, which
is independent of the choice of $Z'$. The restriction of $\omega$
of degree $2m$ vanish unless it is of type $m,m$ since
$Z_{smooth}$ is an analytic manifold of dimension $m$.

\subsubsection{Poincar\'e duality}  On compact oriented 
differentiable manifolds, we use the
wedge product of differential forms to define the cup-product on
de Rham cohomology and integration to define the trace,  so that
we can state Poincar\'e duality \cite{g-h}.

\vskip.1in
\n {\it Cup-product.}
For a manifold $X$, the cup product is a natural bilinear
  operation on de Rham cohomology
\begin{equation*} H^i(X,\C)\otimes  H^j(X,\C)\xrightarrow{\smile}
H^{i+j}(X,\C)
\end{equation*}
  defined by the
wedge product on the level of differential forms. 

  It is a topological product defined on cohomology with coefficients in $\Z$,
   but its
definition on cohomology groups with integer coefficients is less
straightforward.

\vskip.1in
\n {\it The Trace map}.  On a compact oriented
mani\-fold $X$ of dimension $n$, the integral over $X$ of a
differential $\omega$ of highest degree $n$ depends only on its
class modulo boundary by Stokes theorem, hence it defines a map
called the trace:
\[ Tr: H^n(X, \C) \to \C \quad [\omega] \mapsto \int_X \omega.\]
The following theorem is stated without proof:

\begin{theorem}[Poincar\'e duality ]  Let $X$ be a compact
oriented mani\-fold of dimension $n$.   The cup-product
and the trace map:
\begin{equation*} H^j(X,\C)\otimes  H^{n-j}(X,\C)\xrightarrow{\smile}
H^n(X,\C) \xrightarrow {Tr} \C
\end{equation*}
define  an isomorphism:
$$H^j(X,\C)\xrightarrow {\sim} Hom (H^{n-j}(X,\C), \C).$$
\end{theorem}

If the compact complex analytic space $Z$ is of codimension $p$ in
the smooth compact complex manifold $X$ its class $cl(Z) \in
H^{2n-2p}(X,\C)^*$ corresponds, by Poincar\'e duality  on $X$, to
a fundamental cohomology class $[\eta_Z] \in H^{2p}(X,\C)$. Then
we have by definition:

\begin{lemma-definition} For a complex compact manifold $X$,
the fundamental cohomology class $[\eta_Z]  \in H^{p,p}(X,\C)$ of
a closed complex submanifold $Z$ of codimension $p$ satisfies the
following relation:
$$ \int_X \varphi \wedge \eta_Z =  \int_Z \varphi_{|Z }, \quad \forall
 \varphi \in \EE^{n-p,n-p}(X)
.$$
\end{lemma-definition}

\begin{lemma} For a  compact K\"{a}hler manifold $X$,
the cohomology class of a compact complex analytic closed
submanifold $Z$ of codim $p$ is a non-zero element  $[\eta_Z]
\not= 0 \in H^{p,p}(X,\C)$, for $ 0 \leq p \leq $ dim $X$.
\end{lemma}

\begin{proof} For a compact  K\"{a}hler  manifold $X$,
let $\omega$ be a K\"{a}hler form,  then the integral on $Z$ of
the restriction $ \omega_{|Z }$ is positive since it is a
K\"{a}hler form on $Z$, hence $[\eta_Z] \not= 0$:
$$ \int_X (\wedge^{n-p}\omega) \wedge \eta_Z = \int_Z
\wedge^{n-p}(\omega_{|Z })  > 0.$$
\end{proof}

 \subsubsection{Topological construction} The  dual vector space
 $Hom_{\C} (H^{2m}(X, \C), \C)$   is naturally isomorphic
 to  the homology vector space which suggests that the
  fundamental class is defined naturally
 in homology. Indeed, Hodge conjecture has gained so much attention that
 it is of fundamental importance to discover what conditions can be made
 on classes of algebraic subvarieties, including the definition of the classes
  in cohomology   with coefficients in $\Z$.
  
 The theory of homology and
cohomology traces its origin to the work of Poincar\'e in the late
nineteenth century. There are actually many different theories,
for example, simplicial and singular. In 1931, Georges de Rham
proved a conjecture of Poincar\'e on a relationship between cycles
and differential forms that establishes for a compact orientable manifold
 an isomorphism between singular cohomology with real
coefficients and what is known now as de Rham cohomology. In fact
the direct homological construction  of the class of $Z$ is
natural and well known.

 \medskip
\n  {\it Homological class}. The idea that homology should
represent the classes of topological subspaces has been probably
at the origin of homology theory, although the effective
construction of homology groups is different and more elaborate. The simplest way to construct Homology
groups is to use singular simplices. But the definition of
homology groups $H_j (X, \Z)$ of a triangulated space is quite
natural. Indeed, the sum of oriented triangles of highest
dimensions of an oriented triangulated topological space $X$ of
real dimension $n$, defines an homology class
 $[X] \in H_n(X, \Z)$ (\cite{g-h}, Ch 0, paragraph 4). 
 
We take as granted here that  a  closed  subvariety $Z$ of  dimension $m$ in
a compact complex algebraic variety $X$ can be triangulated such
that the sum of its oriented triangles of highest dimensions
 defines its class in  the homology group $ H_{2m}(X, \Z)$.

 \subsubsection{ Cap product}
 Cohomology groups of a topological space $H^i(X, \Z)$ are  dually defined
, and there exists a topological operation on homology and
  cohomology, called the cap product:
\[\frown: H^{q}(X, \Z) \otimes  H_{p}(X, \Z) \rightarrow H_{p-q}(X, \Z) \]
  We can  state now a duality theorem of Poincar\'e, frequently used
  in geometry.

\begin{theorem}[Poincar\'e duality isomorphism]  Let $X$ be a compact
oriented topological manifold of dimension $n$. The cap product
with the fundamental class $[X] \in H_{n}(X,\Z)$ defines an
isomorphism, for all $j$, $0\leq j\leq n$:
\begin{equation*}D_X: H^j(X,\Z)\xrightarrow{\frown [X]} H_{n-j}(X,\Z)
\end{equation*}
\end{theorem}

The homology class $i_*[Z]$ of a  compact complex analytic subspace $i: Z \to X $  of codimension $p$ in
the smooth compact complex manifold $X$ of dimension $n$  corresponds by the inverse of Poincar\'e duality
isomorphism $D_X$, to a fundamental cohomology class:
$$[\eta_Z]^{top} \in H^{2p}(X,\Z).$$

\begin{lemma}
The canonical morphism $H_{2n-2p}(X, \Z) \to H^{2n-2p}(X, \C)^*$
 carry the topological  class $[Z]$  of  an analytic subspace
 $Z$ of codimension $p$ in $X$ into the
 fundamental class $cl (Z)$.\\
 Respectively.  the morphism  $H^{2p}(X,\Z)\to  H^{2p}(X,\C)$  carry the topological  class 
 $[\eta_Z]^{top}$ to $[\eta_Z]$.
 \end{lemma}
 
 \subsubsection{ Intersection product in topology and geometry}
 On
a triangulated space, cycles are defined as  a sum of triangles
with boundary zero, and homology is defined by classes of cycles
modulo boundaries.
It is always possible to represent  two homology classes of degree $p$ and $q$
 by two cycles of codim. $p$ and $q$  in
`transversal position' , so that
 their intersection is defined as a cycle of
 codim.$p+q$. Moreover, for two representations by
 transversal cycles,  the intersections cycles  are  homologous
 \cite{G-M}, (\cite{E-N} 2.8). Then a theory
of intersection product on homology can be deduced:
\begin{equation*} H_{n-p}(X,\Z)\otimes H_{n-q}(X,\Z)  \xrightarrow{\cap}
H_{n-p-q}(X,\Z)
\end{equation*}
In geometry, two closed submanifolds
$V_1$ and $V_2$ of a compact oriented  manifold $M$ can be deformed into a
transversal position so that their intersection can be defined as
a submanifold $W$ with a sign depending on the orientation (\cite{G-P}, ch 2), then $W$ is defined up to a
deformation such that the homology classes satisfy the relation $
[V_1] \cap [V_2] = \pm[W]$. 
The deformation class of $W$ with sign is called
$V_1 \cap V_2$.

\vskip.1in
\n {\it Poincar\'e duality in Homology} (\cite{g-h}, p 53) . \\
 The intersection pairing:
\begin{equation*} H_j(X,\Z)\otimes H_{n-j}(X,\Z)  \xrightarrow{\cap}
H_0(X,\Z) \xrightarrow {degree} \Z
\end{equation*}
 for  $0\leq j\leq n$ is unimodular:  the induced morphism
$$H_j(X,\Z)\rightarrow  Hom (H_{n-j}(X,\Z), \Z) $$
is surjective and its kernel is the torsion subgroup of
$H_j(X,\Z)$.

 \subsubsection{Relation between Intersection and cup products}
 The cup product previously defined in de Rham cohomology
is a topological product defined with coefficients in $\Z$, but
its definition on cohomology groups is less straightforward. The
trace map can also be defined with coefficients in $\Z$, and we
also have the corresponding duality statement,

\vskip.1in
\n {\it Poincar\'e duality in cohomology}.
 
\n The cup-product and the trace map:
\begin{equation*} H^j(X,\Z)\otimes  H^{n-j}(X,\Z)\xrightarrow{\smile}
H^n(X,\Z) \xrightarrow {Tr} \Z
\end{equation*}
define a unimodular pairing inducing an isomorphism:
$$H^j(X,\Q)\xrightarrow {\sim} Hom (H^{n-j}(X,\Q), \Q)$$
compatible with the  pairing in de Rham cohomology.

\vskip.1in
\n {\it Poincar\'e duality isomorphism transforms the intersection
pairing into the cup product:}\\
The following result is proved in \cite{g-h} (p. 59) in the case
$k' = n-k$:

 \n Let
 $\sigma $ be a $k-$cycle on an oriented manifold $X$ of  real dimension
 $n$ and $\tau$ an $k'-$cycle on $X$ with Poincar\'e duals
 $\eta_{\sigma} \in H^{n-k}(X)$ and $\eta_{\tau}\in H^{n-k'}(X)$,
 then:
 \[ \eta_{\sigma} \smile \eta_{\tau} = \eta_{\sigma \frown \tau} \in
 H^{n-k-k'}(X) \]
 
\subsubsection{} The Hodge type $(p,p)$ of the fundamental
class of analytic compact submanifold of codimension $p$ is an
analytic condition. The search for properties characterizing
classes of cycles  has been motivated by a question of Hodge.

\begin{definition} Let $A = \Z$ or $\Q$, $p \in \N$ and
 let $ \varphi: H^{2p}(X,A)\to H^{2p}(X,\C)$ denotes the canonical map.
 The group of cycles
 \[ H^{p,p}(X,A):= \{ x \in H^{2p}(X,A): \varphi (x) \in H^{p,p}(X,\C)\}\]
 is called the group of Hodge classes of type $(p,p)$.
\end{definition}

\begin{definition} An $r-$cycle of an algebraic variety  $X$ is
a formal finite linear combination $\sum_{i\in [1,h]} m_i Z_i$
of closed irreducible subvarieties $Z_i$ of dimension $r$ with
integer coefficients $m_i$. The group of $r-$cycles is denoted by
$\ZZ_r(X)$.
\end{definition}

For a compact complex algebraic manifold, the class of 
closed irreducible subvarieties of codimension $p$
extends into a linear morphism:
$$cl_A: \ZZ_p(X)\otimes A \to H^{p,p}(X,A): \sum_{i\in [1,h]} m_i Z_i \mapsto
\sum_{i\in [1,h]} m_i \eta_{Z_i}, \forall m_i \in A$$
 The
elements of the image of $cl_{\Q}$ are called rational algebraic
Hodge classes of type $(p,p)$.

\subsubsection{Hodge conjecture} Is the map $cl_{\Q}$ surjective
when $X$ is a projective manifold? In other terms, is any rational
Hodge class algebraic?

\vskip.1in
The question is to construct a cycle of codimension $p$ out of a
rational cohomology element of type ($p,p$) knowing that
cohomology is constructed via topological technique.

 Originally, the  Hodge conjecture was stated with $\Z$-coefficients,
 but there  are torsion elements
which cannot be  represented by algebraic cycles. There exists
compact K\"{a}hler complex manifolds not containing  enough
 analytic subspaces to represent all Hodge cycles \cite{Vo1}.

\begin{remark}[Absolute Hodge cycle] In fact
Deligne added another property for algebraic cycles by introducing
the notion of  absolute Hodge cycle (see \cite{Del82}). An
algebraic cycle $Z$ is necessarily defined over a field extension
$K$ of finite type over $\Q$. Then its cohomology class in the de Rham
cohomology of $X$ over the field $K$ defines, for each embedding $\sigma: K
\to \C$, a cohomology class $[Z]_{\sigma}$ of
type $(p,p)$  in the cohomology  of the complex manifold
$X_{\sigma}^{an}$.
\end{remark}

\begin{remark}[Grothendieck fundamental class] For an algebraic
subvariety $Z$ of codimension $p$ in a variety $X$ of dimension
$n$, the fundamental class can be defined as an element of the
group $Ext^p (\OO_Z, \Omega^p_X)$ (see \cite{Gr1}, \cite{H1}). Let
$U$ be an affine subset and suppose that $Z\cap U$ is defined as a complete
intersection by $p$ regular functions $f_i$, if we use the Koszul
resolution of $\OO_{ Z \cap U}$ defined by the elements $f_i$ to
compute the extension
groups, then the cohomology class is defined by a symbol: \\
\centerline {$\left[
\begin{array}{c}
  d f_1\wedge \cdots \wedge d f_p \\
  f_1 \cdots f_p
\end{array}
\right] \in Ext^p (\OO_{ Z \cap U}, \Omega^p_U).$} This symbol is
independent of the choice of generators of $\OO_{ Z \cap U}$, and
it is the restriction of a unique class defined over $Z$ which
defines the cohomology class of $Z$ in the de Rham cohomology
group $\H^{2p}_Z (X, \Omega^*_X)$ with support in $Z$ (\cite{EZ}).
The extension groups and cohomology groups with support in closed
algebraic subvarieties form the   basic idea to construct the
dualizing complex
 $K^*_X$ of $X$ as part of Grothendieck duality theory (see \cite{H1}).
\end{remark}

\section{Lefschetz decomposition and Polarized  Hodge structure}

In  this chapter, we give more specific results  for smooth
complex {\it projective} varieties, namely the Lefschetz decomposition
and primitive cohomology subspaces, Riemann bilinear relations and
their abstract formulation into polarization of Hodge Structures. We start 
with the statement of the results. Proofs and complements follow.

\subsection{Lefschetz decomposition  and primitive cohomology}
 Let $(X, \omega) $ be a compact
K\"{a}hler manifold of class $[\omega] \in H^2(X,\R)$ of Hodge
type $(1,1)$. The cup-product with $[\omega]$ defines morphisms:
\begin{equation*}
 L: H^q(X,\R)\to H^{q+2}(X, \R), \quad L: H^q(X,\C)\to H^{q+2}(X, \C)
\end{equation*}
Referring to de Rham cohomology, the action of $L$ is represented
on the level of forms as $\varphi \mapsto \varphi \wedge \omega $. Since
$ \omega$ is closed, the image of a closed form (resp. a boundary) is
 closed (resp. a boundary). Let
$n = \dim X$.

\begin{definition} The primitive cohomology subspaces are defined
as:
\begin{equation*}
  H_{prim}^{n-i}(\R) := Ker ( L^{i+1}:H^{n-i}(X,\R)\to
  H^{n+i+2}(X,\R))
\end{equation*}
 and similarly for  complex coefficients
$H_{\rm prim}^{n-i}(\C)\simeq H_{\rm prim}^{n-i}(\R)\otimes_{\R} \C$.
\end{definition}

 The operator $L$ is compatible with Hodge decomposition since it sends the subspace
$H^{p,q}$ to $H^{p+1,q+1}$. We shall say that the morphism $L$  is of Hodge type $(1,1)$; hence
 the  action $L^{i+1}:H^{n-i}(X,\C)\to
  H^{n+i+2}(X,\C)$ is a morphism of Hodge type $(i+1,i+1)$, and the kernel is endowed with an induced
 Hodge decomposition.  This is a strong 
 condition on the primitive subspace,
 since if we
 let: 
 $$H^{p,q}_{\rm prim}:= H^{p+q}_{\rm prim} \cap H^{p,q}(X, \C),$$ 
 then:
$$H_{\rm prim}^i(X, \C) = \oplus_{p+q = i}H^{p,q}_{prim}.$$

The following isomorphism, referred to as Hard Lefschetz Theorem, puts a
strong condition on the cohomology of projective, and more generally
compact K\"{a}hler,  manifolds
 and gives rise to a decomposition of the cohomology in terms of
primitive subspaces:

\begin{theorem} Let $X$ be a compact K\"ahler manifold.\\ 
i) {\rm Hard Lefschetz Theorem}.
 The  iterated linear operator
$L$ induces  isomorphisms for each $i$:
 \begin{equation*}
 L^i: H^{n-i}(X,\R)\xrightarrow {\sim} H^{n+i}(X,\R), \,
  L^i: H^{n-i}(X,\C)\xrightarrow {\sim} H^{n+i}(X,\C)
\end{equation*}
ii) {\rm Lefschetz Decomposition}. The cohomology decomposes into
a direct sum of image of primitive subspaces by $L^r, \, r \geq 0$:
\begin{equation*}
  H^q(X,\R)= \oplus_{r \geq 0} L^r H_{\rm prim}^{q-2r}(\R), \quad
   H^q(X,\C)= \oplus_{r \geq 0} L^rH_{\rm prim}^{q-2r}(\C)
\end{equation*}
The Lefschetz decomposition is compatible with Hodge decomposition.\\
iii) If $X$ is moreover projective, then the action of $L$ is
defined with rational coefficients and the decomposition applies to
rational cohomology.
\end{theorem}

\subsubsection{ Hermitian product on cohomology} 

 From the isomorphism in the Hard Lefschetz theorem and
  Poincar\'e duality, we deduce    a scalar product on cohomology of smooth complex projective varieties compatible
 with Hodge Structures and  satisfying relations known as Hodge
 Riemann relations leading to a polarization of the primitive
 cohomology which is an additional highly rich  structure characteristic of such varieties.\\ Representing
  cohomology classes by differential forms, we define a bilinear
 form:
\begin{equation*}
  Q ( \alpha, \beta )= (-1)^{\frac{j(j-1)}{2}}\int _X
 \alpha \wedge \beta \wedge \omega^{n-j}, \quad \forall [\alpha],
  [\beta] \in H^j(X,\C)
\end{equation*}
   where  $\omega$ is the  K\"{a}hler class, the product of $\alpha$
   with $\omega^{n-j}$ represents the action of $L^{n-j}$ and the integral
   of the product with $\beta$ represents Poincar\'e duality.
   
   \vskip 0.1in
   {\it Properties of the product.}
   The above product $\hbox{Q}( \alpha, \beta )$
   depends only on the class  of $\alpha$ and $\beta$.
   The following properties are satisfied:\\
i) the product $\hbox{Q}$ is real  ( it takes real values on real forms) since $\omega$ is real, in other terms the matrix of $\hbox{Q}$ is real,
skew-symmetric if $j$ is odd and symmetric
if $j$ is even;\\
ii) It is non degenerate, by Lefschetz isomorphism and Poincar\'e
 duality;\\
 iii) By consideration of type, the Hodge  and  Lefschetz
 decompositions are  orthogonal relative to $\hbox{Q}$:
\begin{equation*}
\hbox{Q}( H^{p,q}, H^{p',q'}) = 0, \quad \mbox{unless} \,
  \,  p = p', q = q'.
\end{equation*}
 On projective varieties the K\"{a}hler class is
  in the integral lattice defined by cohomology with coefficients in $\Z$, hence the product
  is defined on rational cohomology and preserves the integral lattice.
In this case we have more precise positivity relations in terms of
the primitive component $H_{prim}^{p,q}(\C)$ of the cohomology
$H^{p+q}(X,\C)$.

 \begin{prop}[Hodge-Riemann bilinear relations]
 The  product $i^{p-q}
Q( \alpha,\overline{\alpha})$ is positive definite on the
primitive component $H_{prim}^{p,q}$:

\begin{equation*}
 i^{p-q}
Q( \alpha,\overline{\alpha})> 0, \quad \forall  \alpha
\in
 H_{prim}^{p,q},\,\alpha\neq 0
\end{equation*}
\end{prop}

 \subsubsection{ Summary of the proof of Lefschetz decomposition }
  \
 
\n 1) First, we  consider  the action of $L$ on sheaves,
 $L = \wedge \omega: \EE^r_X \to \EE^{r+2}_X$,  then
we introduce its formal adjoint operator with respect to the
Hermitian form: $\Lambda = L^*: \EE^r_X \to \EE^{r-2}_X$ which can
be defined, using the Hodge star operator, by $\Lambda =
*^{-1}\circ L \circ *$. Note that the operator:
\begin{equation*}
 h = \sum_{p = 0}^{2n} (n-p) \Pi^p, \quad h(\sum_p \omega_p) = \sum_p
 (n-p) \omega_p, \quad {\rm for }\,\, \omega_p \in \EE^p_X,
\end{equation*}
where $\Pi^p$ is the projection in degree $p$ on $\oplus_{p \in
[0, 2n]}\EE^p_X$ and $n =\dim X$, satisfies the relation:
\begin{equation*}
 [\Lambda, L] = h.
 \end{equation*}
  from which we can deduce the injectivity of the morphism:
 \begin{equation*}
 L^i: \EE_X ^{n-i}\rightarrow  \EE_X^{n+i}.
\end{equation*}
For this we use the following formula, for $\alpha\in \EE_X^{k}$:
\begin{equation*}
 [L^r, \Lambda] (\alpha)= (r(k-n) + r(r-1)) L^{r-1}(\alpha),
\end{equation*}
which is proved by induction on $r$. \\
The morphism $L^j$ 
{\it commutes with the Laplacian } and, since cohomology classes can be
represented by
 global harmonic sections, $L^j$  induces an isomorphism on
 cohomology vector spaces which are of finite equal dimension by Poincar\'e
 duality:
 \begin{equation*}
 L^j: H^{n-j}(X,\R)\xrightarrow {\sim} H^{n+j}(X,\R)
\end{equation*}
Moreover the extension of the operator $L$  to complex
coefficients  is compatible with the bigrading $(p,q)$ since
$\omega$ is of type ($1,1$).
 The decomposition of the cohomology into direct sum
of image of primitive subspaces by $L^r, \, r \geq 0$ follows from
the previous isomorphisms.\\
2) Another  proof is based on the representation theory of the Lie
algebra $sl_2$. We represent $L$ by the action on global sections
of the operator: 
$$L = \wedge \omega: \EE^{p,q} (X) \to
\EE^{p+1,q+1} (X)$$  
and  its adjoint $\Lambda = L^*: \EE^{p,q} (X)
\to \EE^{p-1,q-1} (X)$ is defined by $\Lambda = *^{-1}\circ L
\circ *$.
 We have the relations:
 
\begin{equation*}
[\Lambda, L] = h, \quad  [h, L] = -2 L, \quad [h,\Lambda] = 2
\Lambda
\end{equation*}

\n Hence, the operators $L$, $\Lambda$, $h$ generate a Lie algebra
isomorphic to $sl_2$. We deduce from the action of the operators $L, \Lambda, h $ on the
space of harmonic forms $\oplus \HH^*_d(X)\simeq \oplus H^*(X)$, a
representation of the Lie algebra $sl_2$ by identifying the
matrices:
$$\left(
\begin{array}{cc}
  0 & 1\\
  0 & 0 \\
\end{array}
\right),\,\left(
\begin{array}{cc}
  0 & 0\\
  1 & 0 \\
\end{array}
\right), \,\left(
\begin{array}{cc}
  1 & 0\\
  0 & -1 \\
\end{array}
\right)\, \mbox{resp. with} \, \Lambda, \,L,\, h.$$ Then the
theorem follows from the general structure of  such
representation (see \cite{g-h} p. 122-124).

\subsubsection{ Proof of Hodge-Riemann bilinear relations} 
We recall  the notion of a primitive form.
 A differential form  $\alpha \in \Omega^k_{X,x,\R}, k \leq n$ is said to be primitive if
 $L^{n-k+1}\alpha = 0$; then every form $\beta \in  \Omega^k_{X,x,\R}$ decomposes in
 a unique way into a direct sum  $\beta = \sum_r L^r  \beta_r$ where each  $\beta_r$ is primitive of degree $k-2r$ with the condition $k-2r \leq$inf$(k, 2n-k)$. Moreover, the decomposition 
 applies to the complexified forms and it is compatible with bidegrees  $(p,q)$ of forms.

 The following lemma  is admitted and is useful in the proof
(\cite{Vo}  prop.6.29) 
 \begin{lemma} Let $\gamma \in \Omega^{q,p}_{X,x} \subset
 \EE^j_{X,x}$ be a primitive element, then
 \begin{equation*}
* \gamma = (-1)^{\frac{j(j+1)}{2}}i^{q-p} \frac{L^{n-j}}{(n-j)!}\gamma
\end{equation*}
 \end{lemma}
 It is applied to $\overline{\alpha}$ in the definition of the product $Q$ as follows.
 We represent a primitive cohomolgy class by a primitive
  harmonic form $\tilde \alpha$ of degree $j = p+q$, then  we write
 the product in terms of the $L^2-$norm, as follows
\begin{equation*}
\hbox{F}(\alpha,\alpha)=i^{p-q} Q(\alpha,\overline{\alpha}) = 
 (n-j)!
\int_X \tilde \alpha \wedge  * \overline{\tilde \alpha}\\
   = (n-j)! \vert \vert \tilde \alpha \vert \vert^2_{L^2}
\end{equation*}

 Indeed, since $(-1)^j i^{q-p}
 = (-1)^{q-p} i^{q-p}= i^{p-q}
 $  we deduce from the lemma
  $L^{n-j}\overline{ \tilde \alpha}=
i^{q-p}(n-j)!(-1)^{\frac{j(j-1)}{2}} (*\overline{\tilde \alpha})$.

This result suggest to introduce the Weil operator $C$ on
cohomology:
\begin{equation*}
   C( \alpha )= i^{p-q} \alpha,\quad \forall \alpha  \in
 H^{p,q}
\end{equation*}
Notice that $C$ is a real operator since for a real vector $v =
\sum_{p+q = j} v^{p,q}$, $v^{q,p}= \overline {v^{p,q}}$, hence
$\overline {C v} = \sum \overline {i^{p-q}v^{p,q}} = \sum
\overline i^{p-q} v^{q,p} =  i^{q-p} v^{q,p}$, as $\overline
i^{p-q} = i^{q-p}  $. It depends on the decomposition, in
particular the action of $C$ in a variation of Hodge structure is
differentiable in the parameter space. We deduce from $Q$ a new
non degenerate Hermitian product:
\begin{equation*}
 \hbox{F}( \alpha, \beta )=  Q ( C(\alpha), \overline {\beta} ),\,\,
 \hbox{F}( \beta, \alpha )= \overline { \hbox{F}( \alpha, \beta )}
   \quad \forall [\alpha],  [\beta] \in H^j(X,\C)
\end{equation*}
 We use  $\overline {Q ( \alpha,
 \beta)} =  Q (\overline { \alpha},\overline{\beta})$ since $Q$ is real,
 to check for $\alpha, \beta \in H^{p,q}$:\\
$\overline {\hbox{F}( \alpha, \beta )}=\overline {Q (i^{p-q} \alpha,
\overline {\beta})} = \overline {i}^{p-q} Q (\overline { \alpha},
\beta)=(-1)^j i^{p-q} Q ( \overline { \alpha},\beta)
 =(-1)^{2j} i^{p-q} Q (\beta, \overline { \alpha})=\hbox{F}( \beta, \alpha )$.
\begin{remark}
 i) When the class $[\omega] \in
H^j(X,\Z)$ is integral, which is the case for algebraic varieties,
the product $Q$ is integral, i.e., with integral value on integral
classes.\\
ii) The  integral bilinear form $Q$ extends by linearity to the
complex space $H^k_{\Z}\otimes \C$. Its associated form
$\hbox{F}(\alpha,\beta) := Q(\alpha, \overline {\beta})$ is not Hermitian
if $k$ is odd. One way to overcome this sign problem is to define
$\hbox{F}$ as  $\hbox{F}(\alpha,\beta) := i^kQ(\alpha, \overline {\beta})$,
still this form will not be positive definite in general.
 \end{remark}
\subsubsection{ Projective case}
  In the special projective case
we can choose the fundamental class of an hyperplane section to
represent the K\"{a}hler class $[\omega]$, hence we have an
integral representative of the class $[\omega]$ in the image of  $ H^2 (X, \Z) \to H^2 (X, \C)$,
which has an important consequence since the operator $ L:
H^q(X,\Q)\to H^{q+2}(X, \Q)$ acts on rational cohomology.
 This fact characterizes projective varieties among compact K\"{a}hler
 manifolds since a theorem of Kodaira 
 ( \cite{wel} chap. VI)
 states that a K\"{a}hler variety
 with an integral class $[\omega]$ is projective, i.e., it can be
embedded as a closed analytic subvariety in a projective space,
hence by Chow Lemma it is necessarily an algebraic subvariety.
 \begin{remark} [Topological interpretation] 
  In the projective case, the class $[\omega]$ corresponds to the
 homology class of an hyperplane section $[H] \in H_{2n-2}(X,
 \Z)$,  so that the operator $L$ corresponds to
  the intersection with  $[H]$ in  $X$ and
  the result is  an isomorphism:
\begin{equation*}
 H_{n+k}(X)\xrightarrow {(\cap [H])^k} H_{n-k}(X)
\end{equation*}
The primitive cohomology $H_{prim}^{n-k}(X)$ corresponds to the
image of:
 $$H_{n-k}(X - H, \Q ) \rightarrow H_{n-k}(X, \Q). $$
 In his original work, Lefschetz used the hyperplane section of a projective variety
 in homology to investigate the decomposition theorem proved later in the above analysis 
 setting.
\end{remark}
 \begin{definition} The cohomology of a projective  complex smooth
 variety with its Hodge decomposition and the above positive definite Hermitian product is called a polarized Hodge Structure.
\end{definition}

 \subsection{The   category of Hodge Structures} It is rewarding to
  introduce  the Hodge decomposition  as a formal structure in
  linear algebra without any reference to its construction.
  
\begin{definition} A Hodge structure (HS) of weight $n$ is defined
by the  data: \\
i) A finitely generated abelian group $H_{\Z}$; \\
 ii) A decomposition by complex subspaces:
 \begin{equation*}
H_{\C} := H_{\Z}\otimes_{\Z} \C = \oplus_{p+q=n} H^{p,q}\,\, \, \,
\mbox {satisfying }\, \,\,\, H^{p,q} = \overline {H^{q,p}}.
\end{equation*}
\end{definition}

The conjugation on $H_{\C}$ makes sense with respect to
$H_{\Z}$. A subspace $ V \subset H_{\C} := H_{\Z}\otimes_{\Z}
\C\,$ satisfying $\overline V = V$ has a real structure, that is
$V = (V \cap H_{\R})\otimes_{\R} \C$. In particular $H^{p,p}=
(H^{p,p}\cap H_{\R})\otimes_{\R} \C$.
 We may suppose that $H_{\Z}$ is a free abelian group (the lattice),
if we are interested only in its image  in $ H_{\Q} :=
H_{\Z}\otimes_{\Z} \Q$.

 With such an abstract definition we can perform  linear algebra
 operations on a Hodge Structure and define morphisms.  We remark that  {\it Poincar\'e duality}
  is compatible with the Hodge Structure defined by Hodge decomposition, after an ade\-quate shift in the weight (see below).
We deduce from the previous results the notion of polarization:

\begin{definition}[Polarization of HS] A Hodge structure
$(H_{\Z},(H_{\C}\simeq \oplus_{p+q = n}H^{p,q}))$ of weight $n$ is
polarized if a non-degenerate scalar product $Q$ is defined on
$H_{\Z}$, alternating if $n$ is odd and symmetric if $n$ is even,
such that the terms $H^{p,q}$ of the Hodge decomposition are
orthogonal to each other relative to the
 Hermitian form $\hbox{F}$ on $H_{\C}$ defined as
 $\hbox{F}(\alpha,\beta) := i^{p-q} Q(\alpha, \overline {\beta})$ for
 $\alpha, \beta \in H^{p,q}$ and $\hbox{F}(\alpha,{\alpha})$ is positive definite
on the component of type $(p,q)$, i.e., it satisfies the Hodge-Riemann
bilinear relations.
  \end{definition}

\subsubsection{ Equivalent definition of HS } (see \ref{EHS} below) Let ${\mathbb S} (\R) $ denotes the
subgroup:
\[ {\mathbb S} (\R) = \left\{ M(u,v) = \begin{pmatrix} u & -v \\ v & u \end{pmatrix}
\in GL(2, \R), \quad u, v \in \R \right\}.\] 
It is isomorphic to
$\C^*$ via the group homomorphism $M(u,v) \mapsto z = u+iv$.\\
The interest in this isomorphism is to put the structure of a real
algebraic group on $\C^*$; indeed the set of matrices $M(u,v)$ is
a real algebraic subgroup of $GL (2, \R)$.

\begin{definition} A rational Hodge Structure of weight $m \in \Z$, is defined by
 a $\Q-$vector space $H$ and a representation of real algebraic
groups $\varphi: {\mathbb S}(\R) \to GL (H_{\R})$ such that
 for $t \in \R^* $, $\varphi (t)(v) = t^m v$ for all $v \in H_{\R}$.
\end{definition}

\subsubsection{The Hodge filtration.}
 To study variations of  Hodge Structure, Griffiths introduced an equi\-valent
  structure defined by the Hodge filtration
  which varies holomorphically with parameters.
  \begin{definition}
Given a Hodge Structure ($H_{\Z}, H^{p,q}$)
 of weight $n$, we define a decreasing filtration $F$ by subspaces of $H_{\C}$:
 
\begin{equation*}
F^pH_{\C}:= \oplus_{r\geq p}H^{r,n-r}.
\end{equation*}
\end{definition}

\n  Then, the following decomposition is satisfied:

\begin{equation*}
H_{\C}= F^pH_{\C} \oplus \overline {F^{n-p+1}H_{\C}}.
\end{equation*}
The Hodge decomposition may be recovered from the filtration by
the formula:

\begin{equation*}
H^{p,q} = F^p H_{\C} \cap  \overline {F^q H_{\C}},\quad {\rm for}
\, p+q = n.
\end{equation*}
Hence, we obtain an equivalent definition of Hodge Structures, which played an
important role in the development of Hodge theory, since we will
see in chapter $5$ that the Hodge filtration exists naturally on
the cohomology of any complex proper smooth algebraic variety $X$
and it is induced by the trivial filtration on  the de Rham
complex which may be defined algebraically.
However,  the Hodge  decomposition itself is still deduced
from the existence of a birational projective variety over $X$ and
the decomposition on the cohomology of this projective variety
viewed as a K\"{a}hler manifold.

 \subsubsection{ Linear algebra operations on Hodge structures}
 The interest of the above abstract definition of Hodge Structures is to apply
 classical linear algebra operations.
 
\begin{definition} A morphism $f: H=(H_\Z,H^{p,q}) \rightarrow H'=(H'_\Z,H'^{p,q})$ 
of Hodge Structures, is an
homomorphism $f \colon H_{\Z} \rightarrow {H'}_{\Z}$ such that
$f_{{\C}} \colon H_{{\C}} \rightarrow {H'}_{{\C}}$ is compatible
with the decompositions, i.e, for any $p,q$, $f_\C$ induces a $\C$-linear map
from $H^{p,q}$ into $H'^{p,q}$.
 \end{definition}
 
 Therefore Hodge Structures form a category whose objects are Hodge Structures and morphisms
 are morphisms of Hodge Structures that we have just defined. We have the following important
 result:
 
\begin{proposition} The Hodge Structures of same weight $ n $ form an abelian category. 
\end{proposition}

In
particular, the decomposition on the kernel (resp. cokernel) of a
morphism $\varphi: H \to H'$ is induced by the decomposition of
$H$ (resp. $H'$). For a general sub-vector space $V$ of $H$, the
induced subspaces $H^{p,q} \cap V$ do not define a decomposition
on $V$. If $H$ and $H'$ are of distinct weights, $f$ is
necessarily $0$.

\subsubsection{Tensor product  and Hom }  Let  $H$  and $H'$ be
two HS of weight $n$ and $n'$.\\
1) We define their Hodge Structure tensor product
$H\otimes H'$  of weight $n+n'$ as follows:\\
i) $ (H \otimes H')_{\Z} = H_{\Z} \otimes H'_{\Z}$ \\
ii) the bigrading
  of $(H\otimes H')_{{\C}} = H_{{\C}}\otimes H'_{{\C}}$ is the
   tensor product of the bigradings
   of $H_{{\C}}$ and $H'_{{\C}}$: $(H\otimes H')^{a,b}:=
   \oplus_{p+p' = a, q+q' = b} H^{p,q}\otimes H'^{p',q'}$.
   
   \begin{example} Tate Hodge structure $\Z(1)$ is a Hodge Structure of weight $-2$
 defined by:
\begin{equation*}
H_{\Z}= 2i\pi \Z \subset \C,\quad H_{\C} = H^{-1, -1}
\end{equation*}
 It is purely  bigraded  of type $(-1,-1)$ of rank $1$.
    The  $m-$tensor product  $ \Z(1)\otimes\cdots
 \otimes \Z(1)$ of $\Z(1)$ is a Hodge Structure of weight $-2m$ denoted by $\Z(m)$:
 \begin{equation*}
  H_{\Z}= (2i\pi)^m \Z \subset \C,\quad H_{\C} = H^{-m, -m}
 \end{equation*}
Let $ H = (H_{\Z}, \oplus_{p+q=n} H^{p,q})$ be a Hodge Structure of weight $n$,
its $m-$twist is a Hodge Structure of weight $n-2m$ denoted  $ H(m)$  and
defined by
\begin{equation*}
  H(m)_{\Z}:=  H_{\Z}\otimes (2i\pi)^m \Z,\quad
   H(m)^{p,q}:=  H^{p+m,q+m}
  \end{equation*}
  \end{example}
  \vskip.1in
 \n 2)  Similarly, $\wedge^p H$ is a Hodge Structure of weight $pn$.
 
  \vskip.1in
\n 3) We define a Hodge Structure on  $Hom(H,H')$ of weight  $n'-n$ as follows:\\
i) $ Hom(H, H')_{\Z} := Hom_{\Z} ( H_{\Z}, H'_{\Z})$ ;\\
 ii) the decomposition of $ Hom(H, H')_{\C} := Hom_{\Z}
  ( H_{\Z}, H'_{\Z})\otimes \C
 \simeq Hom_{\C} ( H_{\C}, H'_{\C})$ is given by
$ Hom(H, H')^{a,b}:=\oplus_{p'-p = a, q'-q = b}
Hom_{\C}(H^{p,q},H^{p',q'}) $.\\
  In particular    the dual  $H^*$ to $H$
is a Hodge Structure of weight $-n$.\

 \begin{remark}
The group of morphisms of Hodge Structures is called the internal morphism group
of the category of Hodge Structures and denoted by
 $Hom_{HS} (H,H')$; it is the sub-group of
  $Hom_{{\Z}}(H_{{\Z}},H'_{\Z})$
  of elements of type $(0,0)$ in the Hodge Structure on  $ Hom(H,H')$.
   A homomorphism of  type
$(r,r)$  is a morphism of the Hodge Structures: $ H \to H'(-r)$.
\end{remark}

\n {\it Poincar\'e duality.} If we consider the following trace
map:
\begin{equation*}
H^{2n}( X, \C) \xrightarrow{\sim} \C(-n), \quad \omega \to
\frac{1}{(2i\pi)^n}\int_X \omega.
\end{equation*}
then Poincar\'e duality will be compatible with Hodge Structures:
\begin{equation*} H^{n-i}(X,\C)\simeq
Hom ( H^{n+i}(X,\C), \C(-n) ).
\end{equation*}
The Hodge Structure on homology is defined by duality:
 $$(H_i(X,\C), F) \simeq Hom ((H^i(X,\C),F),\C)$$ 
 with $\C$ pure of
 weight $0$, hence $H_i(X,\Z)$ is pure of weight $-i$.
 Then, Poincar\'e duality becomes an isomorphism  of Hodge Structures:
 $H^{n+i}(X,\C)\simeq H_{n-i}(X, \C)(-n) $.
\vskip 0.1in 
\n {\it Gysin morphism.} Let $f: X \to Y$ be an
algebraic morphism with $\dim X = n$ and $\dim Y = m$, since $f^*:
H^i(Y,\Q) \to H^i(Y,\Q)$ is compatible with Hodge Structures, its Poincar\'e
dual $Gysin (f): H^j(Y,\Q) \to H^{j+2(m-n)}(Y,\Q)(m-n)$ is
compatible with Hodge Structures after a shift by $-2(m-n)$ on the right term.

\subsubsection{Proof of the equivalence with the action of the group ${\mathbb S}$}
\label{EHS}

\begin{lemma} Let
$(H_{\R}, F)$ be a real Hodge Structure of weight  $m$, defined by the
decomposition $H_{\C} = \oplus_{p+q = m} H^{p,q}$, then the action
of ${\mathbb S} (\R) = \C^*$ on $H_{\C}$, defined by: 
$$(z, v) \mapsto
\sum_{p+q = m} z^p {\overline z}^q v_{p,q},$$ 
for $v = \sum_{p+q =
m}v_{p,q}$, corresponds to  a real representation $\varphi: {\mathbb S}
(\R) \to GL(H_{\R})$ satisfying $\varphi (t)(v) = t^m v$ for $t
\in \R^* $.
\end{lemma}

We prove that $\overline { \varphi(z)} = \varphi(z), z \in \C^*$,
hence it is defined by a real matrix. We deduce from $\overline v
= \sum_{p+q = m}\overline {v_{p,q}}\in \oplus H^{q,p}$ that
$(\overline
v)_{q,p} = \overline {v_{p,q}} \in H^{q,p}$, then: \\
$\overline { \varphi(z)}(v) := \overline { \varphi(z)(\overline
v)} = \overline { \sum_{p+q = m} z^q {\overline z}^p \overline
{v_{p,q}}} =  \sum_{p+q = m} z^p {\overline z}^q
v_{p,q} = \varphi(z)(v)$.\\
In particular, for $t \in \R$, $\varphi(t) (v)= t^m v$ and
$\varphi(i) (v) =
\sum_{p+q = m} i^{p-q} v_{p,q}$ is the Weil operator $C$ on $H$.

Reciprocally, a  representation of the multiplicative commutative
torus: 
$${\mathbb S} (\C) \to GL(H_{\C})$$ 
splits into a direct sum of
representations ${\mathbb S} (\C) \to GL(H^{p,q})$ acting via $ v \to z^p
{\overline z}^q v$, moreover if the representation is real, then
$\overline {H^{p,q}} = H^{q,p}$. The real group $ \R^*$ acts on
$H^{p,q}$ by multiplication by $t^{p+q}$, so that  the sum
$\oplus_{p+q = m}H^{p,q}$  defines a sub-Hodge Structure of weight $m$.

\begin{remark} i) The
complex points of:
\[{\mathbb S} (\C) = \left\{ M(u,v) = \begin{pmatrix} u & -v \\ v & u \end{pmatrix}
\in GL(2, \C), \quad u, v \in \C \right\}\] 
are the matrices with
complex coefficients $u, v \in \C$ with  determinant $u^2+v^2
\not= 0$. Let $z= u+iv, z' = u-iv$, then $zz'= u^2+v^2 \not= 0$
such that $  S(\C) \xrightarrow {\sim} \C^* \times\C^*: (u,v)
\mapsto (z,z')$ is an isomorphism satisfying $z' = \overline z$,
if $u,v\in\R$; in particular $\R^*
\hookrightarrow \C^* \times\C^*: t \mapsto (t,t)$.\\
ii) We can write $\C^* \simeq S^1 \times \R^*$ as the product of
the real points of the unitary subgroup $U(1)$ of ${\mathbb S}$ defined by
$u^2 + v^2 = 1 $ and of the multiplicative subgroup $\G_m (\R) $
defined by $ v = 0 $ in ${\mathbb S}(\R)$, and ${\mathbb S}(\R)$ is the semi-product
of $S^1$ and $\R^*$. Then the representation gives rise to a
representation of $\R^*$ called the scaling since it fixes the
weight and of $U(1)$ which fixes the decomposition into $H^{p,q}$
and on which $\varphi(z)$ acts as multiplication by $z^{p-q}$.
\end{remark}

\subsection{Examples}
{ \it Cohomology of projective spaces.} The polarization is
defined by the first Chern class of the
 canonical line bundle $H = c_1({\OO}_{\P^n}(1))$, dual to the homology
 class of a hyperplane.
\begin{prop}
$H^i(\P^n, \Z) = 0$ for $i$ odd and $H^i(\P^n, \Z) = \Z$ for $i$
even with generator $[H]^i$, the cup product to the power $i$ of
the cohomology class of an hyperplane, hence: $H^{2r}(\P^n, \Z) =
\Z(-r)$ as Hodge Structure.
\end{prop}

 \subsubsection{Hodge Structures of dimension 2 and weight 1}
     Let $H$ be a real vector space of dimension $2$ endowed with a
     skew symmetric  quadratic form, $(e_1, e_2)$ a basis in which the matrix
     of  $Q $ is
\[Q = \begin{pmatrix} 0 & 1 \\ -1 & 0 \end{pmatrix}, \quad Q(u,v)
= {^tv}Q u,\quad u , v \in \Q^2 \] then Hodge decomposition
$H_{\C} = H^{1,0} \oplus H^{0,1}$ is defined by the one
dimensional subspace $H^{1,0}$ with  generator $v$ of coordinates
$(v_1,v_2)\in \C^2$. While $Q (v,v) = 0$,  the Hodge-Riemann
positivity condition is written as $i Q(v, \overline v) = -i (v_1
\overline v_2- \overline v_1 v_2 )> 0$, hence $v_2 \not= 0$, so
 we divide by $v_2$ to get a unique representative of $H^{1,0}$
 by a vector of the form $v = (\tau, 1)$ with Im($\tau) > 0$ . Hence the
Poincar\'e half plane \{ $z \in \C :$ Im$z > 0$\} is a classifying
space for polarized Hodge Structures of dimension $2$. This will apply to the
cohomology  $H:= H^1( T, \R)$ of a complex torus of dim.$1$. Note
that the torus  is a projective variety. Indeed,  a non degenerate
lattice in $\C$ defines a Weirstrass function $\PP(z)$ that
defines with its  derivative  an embedding of the torus onto a
smooth elliptic curve $ (\PP(z), \PP'(z)): T \rightarrow  \P_C^2
$.

\subsubsection{Moduli space of Hodge structures of weight 1 {\rm(see \cite{Grif})}} 
The first cohomology group $H_{\Z}$  of a smooth compact algebraic
curve $C$ of genus $g$ is of rank $2g$. The cohomology with
coefficients in $\C$  decomposes into the direct sum: 
$$H_{\C} =
H^{1,0} \oplus H^{0,1}$$ 
of two conjugate subspaces of dimension
$g$. We view here the differentiable structure as fixed, on which
the algebraic structure vary, hence the cohomology is a fixed
vector space on which the decomposition vary. The problem here, is
the classification of the different Hodge Structures on the cohomology.
Recall that $H_{\C}$ is endowed via the cup-product with a non
degenerate alternating form: 
$$Q: (\alpha, \beta) \to \int \alpha
\wedge \beta.$$ 
The subspace $H^{1,0}$ is an element of the
Grassmann variety  $ G(g,H_{\C})$ parametrizing complex vector
subspaces of dimension
$g$. 

The first Riemann bilinear relation states that
   $Q (\alpha, \beta) = 0, \, \forall \alpha, \beta \in H^{1,0}$,
 which translates into a set of algebraic equations satisfied by
 the point
 representing $H^{1,0}$ in $G(g,H_{\C})$, hence the family of such subspaces
 is parameterized by  an algebraic subvariety
 denoted \v{D} $\subset
G(g,H_{\C})$.
 Let $Sp (H_{\C}, Q)$ denotes the symplectic
 subgroup of complex linear automorphisms of $H_{\C}$ commuting with
 $Q$; it acts transitively on \v{D} (see \cite{carl}), hence such variety is
 smooth.
 
  The second positivity relation  is an open condition,
 hence it defines an open subset $D$ of elements representing
 $H^{1,0}$ in \v{D}. The subspace $D$ is called the
 classifying space of polarized Hodge Structures of type $1$
 (the cohomology here is  primitive, hence polarized).
 
 To describe $D$  with coordinates, we choose
a basis $\{ b_1, \ldots, b_g, b_{g+1}, \ldots, b_{2g} \}$ of
$H_{\C}$ in which the matrix of $Q$ is written as a matrix $J$
below,  then $H^{1,0}$ is represented by a choice of free $g$
generators which may be written as a matrix $M$ with $g$ columns
and $2g$ rows satisfying the condition:
\[J= \begin{pmatrix} 0 & I_n \\
-I_n & 0
\end{pmatrix}, 
M = \begin{pmatrix} M_1 \\
M_2
\end{pmatrix}\]

\[ \begin{pmatrix} {^tM}_1 ,\, {^tM}_2
\end{pmatrix}\cdot \begin{pmatrix} 0 & I_n \\
-I_n & 0
\end{pmatrix} \cdot \begin{pmatrix} M_1 \\
M_2
\end{pmatrix} = {^tM}_1 M_2- {^tM}_2 M_1  = 0 \]
 The positivity condition $i Q (\alpha, \overline {\alpha} ) > 0,
 \forall \alpha \not= 0 \in H^{1,0} $ translates into
\[i \cdot \begin{pmatrix} {^t\overline {M}}_1 \, {^t\overline {M}}_2
\end{pmatrix}\cdot \begin{pmatrix} 0 & I_n \\
- I_n & 0
\end{pmatrix} \cdot \begin{pmatrix} M_1 \\
M_2
\end{pmatrix} = i({^t\overline {M}}_1 M_2 - {^t\overline {M}}_2 M_1)  \]
is positive definite. We deduce in particular that the determinant
$\det (M_2) \not= 0$ and we may choose  $M_2 = I_g$, $M_1 = Z$ such that
\[ D \simeq \{ Z :  {^tZ} = Z, \rm{Im}(Z) > 0 \}.\]
that is the imaginary part of the matrix $Z$ defines  a positive
definite  form. Then $D$ is called the generalized Siegel upper
half-plane.
 If $g =1$, then $D$ is the upper half plane.
Let  Sp$(H_{\C}, Q)\simeq$ Sp$(g, \C) =  \{ X \in {\rm GL}(n, \C):
{^tX} \cdot J = J \cdot X \}$, then Sp$(V, Q):= $ Sp$(H_{\C},
Q)\cap $GL$(H_{\R})$ acts transitively on $D$ as follows:
\[ \begin{pmatrix} A & B \\
C & D
\end{pmatrix}\cdot Z = (A\cdot Z + B) \cdot (C\cdot Z + D ) ^{-1}.
\]

It  is a classical but not a trivial result
 that smooth compact curves are homeomorphic  as real surfaces
 to the ``connected'' sum of $g$ tori
 with a natural basis of paths  $\gamma_1, \ldots,\gamma_g,\gamma_{g+1},
 \ldots, \gamma_{2g}$, whose classes generates the homology and such that
 the matrix of the Intersection product is $J$.
 
 The classes obtained by Poincar\'e
 duality isomorphism $ \eta_1, \ldots, \eta_{2g} $ form a basis
 of $H_{\C}$ on which the matrix of the form $Q$
 defined by the product on cohomology is still $J$.
   Let $[\gamma^*_i]$ denotes the ordinary
  dual basis of cohomology. It is not difficult to check on the
  definitions that $[\gamma^*_i] = -\eta_{i+g}$ for $i \in [1,g]$ and
$[\gamma^*_{i+g}] = \eta_i$ for $i \in [g, 2g]$, from which we
deduce that the dual basis satisfy also the matrix $J$. Let
$\omega_i$ be a set of  differential forms defining a basis of
$H^{1,0}$, then  the period matrix of the basis $\omega_i$  with
respect to the basis $[\gamma^*_i]$ is defined by the coordinates
on that basis given by the entries $\int_{\gamma_i} \omega_j$ for
$i \in [1,2g], j \in [1,g]$. Hence such matrices represent
elements in $D$.

\subsubsection{Hodge structures of weight 1 and abelian varieties}
Given a Hodge Structure:
$$( H_{\Z},H^{1,0}\oplus H^{0,1}), $$
the projection on $H^{0,1} $
induces an isomorphism of real vector spaces:
\begin{equation*}
  H_{\R} \to H_{\C} = H^{1,0}\oplus H^{0,1} \to H^{0,1}
\end{equation*}
since $\overline{ H^{0,1}} = H^{1,0}$. Then we deduce that
 $H_{\Z}$ is a lattice in the complex space  $H^{0,1}$, and
 the quotient $T := H^{0,1}/H_{\Z}$ is a complex torus .
 When $H_{\Z}$ is identified with the image
  of cohomology spaces $Im (H^1 (X,\Z)\to H^1(X,\R))$ of a complex manifold
 $X$, resp. $ H^{0,1}$ with $H^1 (X,\OO)$, the torus $T$
 will be identified with the Picard variety
 $Pic^0(X)$ parameterizing the holomorphic line bundles on $X$
 with first Chern class equal to zero as follows. We consider the
exact sequence of sheaves defined by $f \mapsto e^{2i \pi f}$:
\begin{equation*}
 0 \to \Z \to \OO_X  \to \OO_X^* \to 1
\end{equation*}
where $1$ is the neutral element of the sheaf of multiplicative
groups $\OO_X^*$, and its associated long exact sequence of
derived functors of the global sections functor:
\begin{equation*}
  \to H^1 (X,\Z) \to H^1 (X,\OO_X)  \to H^1 (X,\OO_X^*) \to H^2 (X,\Z)
\end{equation*}
 where the morphisms can be interpreted geometrically; when the
space $H^1 (X,\OO_X^*)$ is identified with isomorphisms classes of
line bundles on $X$, the last morphism defines the Chern class of
the line bundle. Hence the torus $T$ is identified with the
isomorphism classes $\LL$ with $c_1(\LL) = 0$:
\begin{equation*}
 T := \frac{H^1(X,\OO_X)}{Im H^1(X,\Z)}\simeq
  Pic^0(X):= Ker(H^1(X,\OO_X^*) \xrightarrow {c_1} H^2 (X,\Z))
\end{equation*}
It is possible to show that the Picard variety of a smooth
projective variety is an abelian variety. (Define a K\"{a}hler
form with integral class on $Pic^0(X)$).

\subsubsection{Hodge structures of weight 2}
\begin{equation*}
 (H_{\Z}, \, H^{2,0}\oplus H^{1,1}\oplus H^{0,2};
\, H^{0,2}=\overline{H^{2,0}} ,\, H^{1,1}=\overline{H^{1,1}};\, Q)
\end{equation*}
 the intersection form $Q$ is symmetric and
 $\hbox{F}(\alpha,\beta)= Q( \alpha,\overline{\beta})$ is Hermitian.
 The decomposition is orthogonal for $\hbox{F}$ with
  $\hbox{F}$ positive definite on $H^{1,1}$, negative definite on
  $H^{2,0}$ and  $ H^{0,2}$.
  
\begin{lemma}
 The HS is completely determined by the subspace
  $ H^{2,0}\subset H_{\C}$ such that $H^{2,0}$ is totally isotropic for $Q$
 and the associated Hermitian product $H$ is negative definite
 on $ H^{2,0}$. The signature of $Q$ is ($h^{1,1}, 2 h^{2,0}$).
\end{lemma}

 In the lemma $H^{0,2}$ is determined by conjugation
 and $ H^{1,1} = (H^{2,0}\oplus H^{0,2})^{\bot}$.

\section{ Mixed Hodge Structures  (MHS)}

The theory of Mixed Hodge Structures (MHS) introduced  by
Deligne in \cite{HII} is a generalization of Hodge Structures that
can be defined on the cohomology of all  algebraic varieties.

Since we are essentially concerned by filtrations of vector
spaces, it is not more  difficult  to describe this notion in the
terminology of abelian categories. We start with a formal study of
Mixed Hodge Structures.  

Let $A
 = {\Z}, {\Q}$ or ${\R}$, and define $A \otimes  \Q$ as  $\Q$ if
 $A = \Z$ or $ \Q$ and $\R$ if
 $A = \R$.
 For an $A$-module $H_A$ of
finite type, we write $H_{A \otimes  \Q}$ for the  $(A \otimes
\Q)$-vector space
 $(H_A) \otimes_A (A \otimes  \Q)$. It is a
 rational  space if  $A = \Z $ or  $\Q$ and a real space  if $A = \R$.
 
\begin{definition} [Mixed Hodge structure (MHS)] An $A$-Mixed Hodge
Structure $H$ consists of:\\
 1) an $A$-module of finite type
$H_A$; \\
 2) a finite increasing  filtration  $W$ of the
$ A \otimes  \Q $-vector space $H_A \otimes \Q $
 called the weight filtration; \\
3) a finite decreasing  filtration  $F$ of the ${\C}$-vector space
$H_{\C} = H_A  \otimes_{A} \C$, called the Hodge filtration, such
that the systems:
\begin{equation*}
Gr^W_n H = (Gr^W_n (H_{A \otimes \Q}) ,(Gr^W_n H_{\C}, F))
\end{equation*}
consist of $A\otimes {\Q}$-Hodge Structure of weight $n$.
 \end{definition}

The Mixed Hodge Structure is called real, if $A = \R$, rational,
if $A = \Q$, and integral, if $A = \Z$. We need to define
clearly, how the filtrations are induced, for example:  \\
$ F^p Gr^W_n H_{\C}:= ((F^p \cap W_n\otimes \C) +W_{n-1}\otimes \C
)/ W_{n-1}\otimes \C \subset (W_n\otimes \C) / W_{n-1}\otimes
\C$.

\vskip.1in
\n {\it Morphism of MHS. }  A morphism  $f : H\rightarrow H'$ 
of Mixed Hodge Structures
 is a morphism $f : H_A\rightarrow H'_A $ whose extension to
 $H_{\Q}$ (resp. $H_{\C}$) is compatible with the
filtration $W$, i.e., $f(W_jH_A)\subset W_jH'_A $ (resp.  $F$, i.e.,
$f(F^jH_A)\subset F^jH'_A $), which implies that it is also compatible  with
$\overline F$. 
\vskip.1in

\n These definitions allow us to speak of the category of 
Mixed Hodge Structures.
\vskip.1in

\n The main result of this chapter is
\begin{theorem}[Deligne]\label{abelian} The category of Mixed Hodge 
Structures  is abelian.
\end{theorem}
The proof relies on the following:
\vskip.1in

\n  {\it  Canonical decomposition of the weight filtration.}
  While there is an equivalence between the Hodge filtration and
  the Hodge decomposition, there is no such result for the  weight
  filtration  of a Mixed Hodge Structure.  In the category of
   Mixed Hodge Structures,
 the short exact sequence  $0 \to Gr^W_{n-1} \to W_n/W_{n-2}
 \to Gr^W_n \to 0$ is a non-split extension
 of the two pure Hodge structures $Gr^W_n$ and $ Gr^W_{n-1} $.
  To  construct the Hodge decomposition,
  we introduce for each
  pair of integers
$(p,q)$,  the subspaces of $H = H_{\C}$:
\begin{equation*}
I^{p,q}=(F^p\cap W_{p+q})\cap(\overline {F^q}\cap W_{p+q}+
\overline {F^{q-1}}\cap W_{p+q-2} + \overline {F^{q-2}}\cap
W_{p+q-3}+ \cdots )
\end{equation*}
By construction they are related for $p+q = n$ to the components
$H^{p,q}$ of the Hodge decomposition: $Gr^W_{n}H \simeq
\oplus_{n=p+q}H^{p,q}$;

\begin{prop}  The projection  $ \varphi: W_{p+q} \rightarrow
Gr^W_{p+q}H\simeq \oplus_{p'+q'=p+q}H^{p',q'}$ induces an
isomorphism $ I^{p,q} \xrightarrow{\sim} H^{p,q}$.
 Moreover,
\begin{equation*} W_n = \oplus_{p+q\leq
n} I^{p,q}, \quad  F^p = \oplus_{p'\geq p} I^{p',q'}
\end{equation*}
\end{prop}
\begin{remark}
i) For three  decreasing opposed filtrations in an abelian
category, Deligne uses in the proof an inductive argument based on
the formula for $i > 0$
  $$F^{p_i} \oplus \overline {F^{q_i}}\simeq Gr^W_{n-i}H,
  \quad p_i +q_i = n-i+1,$$
to construct for $i > 0$ a decreasing family $(p_i, q_i)$ starting
with $(p_0 = p, q_0 = q): p + q = n $ and deduces the existence of a subspace
$A^{p,q} \subset W_m$ projecting isomorphically onto the subspace
of type $(p,q)$ of the Hodge Structure $Gr^W_n H$.
 To simplify, we choose here, for $W$ increasing,  a sequence
of the following type $(p_0,q_0) = (p,q)$, and for $ i > 0 $, $p_i
= p, q_i = q-i +1$ which explains theasymmetry in the formula
(see also \cite{St} Lemma-Definition 3.4),
including the fact that for $i = 1$,
 $\overline {F^{q}}\cap W_{p+q-1}  \subset \overline {F^{q}}\cap W_{p+q}$.\\
ii) In general $I^{p,q}\not= \overline{I^{q,p}}$, we have only
$I^{p,q} \equiv \overline{I^{q,p}}$ modulo $W_{p+q-2}$.\\
iii) A morphism of Mixed Hodge Structures   is necessarily
compatible with  the decomposition, which will be the main
ingredient to prove the strictness  with respect to $W$ and $F$
(see \ref{strict}).
\end{remark}

The proof of this proposition is given below. This proposition is
used to prove Theorem \ref{abelian}.

 The knowledge of the linear algebra underlying
 Mixed Hodge Structures is supposed to help the reader before the reader is confronted
 with their construction. The striking result to remember is that morphisms
 of Mixed Hodge Structures are {\it strict} (see \ref{strict})
 for both filtrations, the weight $W$ and Hodge $F$.
 The complete proofs  are purely algebraic \cite{HII}; the corresponding theory
  in an abelian category is developed for objects with opposite filtrations.
 
 \subsection{Filtrations } Given a morphism in an additive category,
   the isomorphism between the image and co-image is one of the
   conditions  to define an abelian category. In the abelian
   category of vector spaces endowed with finite  filtrations by
   subspaces, if we consider a morphism compatible with
   filtrations $f : (H, F) \to (H', F')$ such that $ f(F^j) \subset F^{' j}$,
   the induced filtration on the image by $F'$ does not coincide
   with the induced filtration by $F$ on the co-image. In this case we say that
   the morphism is strict if they coincide (see \ref{strict}).
This kind of problem will occur for various repeated
restrictions of filtrations and we need here to define with
precision the properties of induced filtrations, since this is at
the heart of the subject of Mixed Hodge Structures.

On a sub-quotient of a filtered vector
 space and in general of an object of an abelian category, there are two ways
 to restrict the filtration (first on the sub-object, then on the quotient and
 the reciprocal way). On an object $A$ with two  filtrations $W$ and $F$,
 we can repeat the restriction for each
 filtration and we get different objects if we start with $W$ then $F$
 or we inverse the order.
 
 We need  to know  precise relations between
 the induced filtrations in these various ways, and to know the precise behavior
 of a linear morphism with respect to such filtrations. For
 example, we will indicate in section $5$, three different ways to induce a
 filtration on the terms of a spectral sequence. A central result
 for Mixed Hodge Structures is to give conditions on the complex such that the three
 induced filtrations coincide.
 Hence, we recall here  preliminaries on filtrations after Deligne.
 \vskip.1in
 
\n Let $ {\A}$ denote an abelian category.
\vskip.1in

 \begin{definition}
 A decreasing (resp. increasing)  filtration $F$ of an object
$A$ of $ {\A}$ is a  family of  sub-objects of $ {\A}$, satisfying
\begin{equation*}
\forall n,m,\quad \quad n\leq m \Longrightarrow F^m(A)\subset
F^n(A)\quad(resp. \, n\leq m\Longrightarrow F_n(A)\subset F_m(A))
\end{equation*}
\end{definition}

\n The pair $(A,F)$ of an object of $\A$ and a decreasing 
(resp. increasing) filtration
will be called a filtered object.
\vskip.1in

\n  If $ F$ is a decreasing filtration  (resp. $W$ an
increasing filtration), a shifted filtration $F[n]$ by an integer
$ n$ is defined as
\begin{equation*}(F[n])^p(A)=F^{n+p}(A) \quad (W[n])_p(A) =
W_{n-p}(A)
 \end{equation*}
 
 Decreasing  filtrations $F$ are  considered for a general
study. Statements for increasing filtrations  $W$ are deduced by
the change of indices  $W_n(A) = F^{-n}(A)$. A filtration is
finite if there  exist integers $n$ and $ m$
 such that $F^n(A) = A$ and $F^m(A) = 0$.\\ 
 A morphism  of
 filtered objects
   $(A,F)\xrightarrow{f}(B,F)$ is a morphism  $A\xrightarrow{f} B$
  satisfying $f(F^n(A)) \subset F^n(B)$ for all $ n \in {\Z}$.
   Filtered objects (resp. of finite filtration) form an
   additive category
  with existence of  $kernel$ and $cokernel$  of a morphism
  with natural induced filtrations
  as well image and  coimage, however the image and  co-image
  will not be necessarily filtered-isomorphic, which is the main
  obstruction to obtain an abelian category.
  
  To lift this obstruction, we are lead to define the notion of strictness for 
  compatible morphims (see \ref{strict} below).
  
   \begin{definition} The graded object  associated to $(A,F)$
   is defined as:
\begin{equation*} Gr_F(A) = \,\, \oplus_{n \in{\Z}}
 Gr_F^n (A) \quad \mbox{where} \quad Gr_F^n (A) =
F^n (A)/F^{n+1}(A).
 \end{equation*}
 \end{definition}
 
 \subsubsection{Induced filtration}
  A filtered object  $(A,F)$  induces
 a filtration on a sub-object $i: B\rightarrow A$ of $A$,  defined by
  $F^n(B) = B\cap F^n(A)$. Dually, the quotient  filtration on $A/B$ is defined by:
\begin{equation*}F^n(A/B) = p(F^n(A)) = (B+F^n(A))/B
  \simeq F^n(A)/(B\cap F^n(A)),
\end{equation*}
  where $p: A \to A/B$ is the projection.
  
  \subsubsection{} The cohomology of a sequence of filtered morphisms:
$$ (A, F)\xrightarrow {f} (B, F) \xrightarrow {g} (C, F)$$
satisfying $g\circ f= 0$
  is defined as   $H = Ker\,g/Im\,f$; it is {\it filtered} and endowed
  with the   quotient filtration of the induced filtration on
   $Ker\,g$. It is
  equal to the  induced filtration on  $H$ by the
   quotient filtration on $B/Im\,f$, since $ H \subset (B/Im\,f)$.
   
\subsubsection{Strictness}\label{strict} For filtered modules over a ring, a morphism
of filtered objects $f:(A,F)\rightarrow (B,F)$ is called
strict if the relation:
   $$ f(F^n(A)) = f(A) \cap F^n(B)$$ 
is  satisfied; that is, any element $b\in F^n(B) \cap Im A$
is already in $Im \, F^n(A)$. 
 
 This  is not  the case for any morphism.
  
  A filtered morphism in an additive  category
  $f: (A,F)\rightarrow (B,F)$
  is called strict if it induces a filtered isomorphism
  $(Coim(f), F) \rightarrow (Im(f), F)$
  from the coimage to the image of  $f$ with their induced
  filtrations.
  
This concept is basic to the theory, so we mention the following
criteria:

\begin{prop} i) A filtered morphism $f: (A, F) \to (B, F)$ of objects
with finite filtrations is strict if and only if the  following
exact sequence of graded objects is exact:
\[ 0 \to Gr_F(Ker f) \to Gr_F(A)\to Gr_F(B)\to  Gr_F(Coker f)\to 0 \]
ii) Let $S: \,\, (A, F)\xrightarrow{f} (B, F) \xrightarrow{g}
(C,F) $ be a sequence  $S$ of strict morphisms such that $g \circ f =
0$, then the cohomology $H$ with induced filtration satisfies:
$$H(Gr_F(S)) \simeq Gr_F(H(S)).$$
\end{prop}

\subsubsection{Two filtrations}
   Let $A$ be an object of ${\A}$ with two
filtrations $F$ and $G$. By definition, $ Gr^n_F(A)$ is a quotient
of a sub-object of $A$, and as such, it is endowed with an induced
filtration by  $G$. Its associated graded object defines a
bigraded object $Gr^n_G Gr^m_F(A)_{n,m\in {\Z}}$. We refer to \cite{HII}
for:

\begin{lemma}[ Zassenhaus' lemma] The objects $Gr^n_G
Gr^m_F(A)$ and $Gr^m_F Gr^n_G(A)$ are isomorphic.
\end{lemma}

\begin{remark}Let $H$ be a third filtration on $A$. It induces a
filtration on $Gr_F(A)$, hence on $Gr_G Gr_F(A)$. It induces also
a filtration on $Gr_F Gr_G(A)$. These filtrations do not
correspond in general under the above isomorphism.
 In the formula $Gr_H Gr_G Gr_F(A)$, $G$ et $H$ have
symmetric role, but not  $F$ and $G$.
 \end{remark}
 
\subsubsection{Hom and tensor functors} If $A$ and $B$ are
two filtered objects of  ${\A}$, we define a  filtration on the
left exact functor $Hom$:
\begin{equation*}
F^k Hom(A,B) = \lbrace f: A\rightarrow B : \forall n,
f(F^n(A))\subset F^{n+k}(B) \rbrace
\end{equation*}
Hence:
 \begin{equation*}
 Hom ((A,F),(B,F))= F^0(Hom(A,B)).
 \end{equation*}
 If $A$ and $B$ are modules on some ring, we define:
\begin{equation*}
 F^k(A\otimes B) = \sum_m Im
(F^m(A)\otimes F^{k-m}(B) \rightarrow A\otimes B)
\end{equation*}

\subsubsection{Multifunctor} In general if $H: {\A}_1\times \ldots
\times {\A}_n\rightarrow \B$
 is a  right exact multi-additive functor, we define:
\begin{equation*} F^k(H(A_1,...,A_n))=  \sum
_{\sum{k_i=k}} Im(H(F^{k_1}A_1,...,F^{k_n}A_n)\rightarrow
H(A_1,...,A_n))
\end{equation*}
and dually if  $H$ is left exact:
\begin{equation*}
 F^k(H(A_1,...,A_n)=  \bigcap_{\sum k_i = k } Ker (H(A_1,...,A_n)
 \rightarrow  H(A_1/F^{k_1}A_1,...,A_n/F^{k_n}A_n))
\end{equation*}
 If $H$ is exact, both definitions are equivalent.
 
 \subsection{Opposite filtrations}
We review here the definition of Hodge Structure in terms of filtrations, as it
is adapted to its generalization to Mixed Hodge Structure.  For any $A$-module
$H_A$, the complex conjugation extends to a conjugation on the
space $H_{{\C}}=  H_A \otimes_A {\C}$. A filtration $ F$ on
$H_{{\C}}$ has a conjugate filtration $\overline F$ such that
$({\overline F})^jH_{\C}= \overline {F^jH_{\C}}$.

\begin{definition} ([HS1])  An A-Hodge structure $H$ of weight $n$
  consists of:\\
i) an $A$-module of finite type  $H_A$,\\
 ii) a finite filtration  $F$ on $H_{{\C}}$ (the Hodge filtration) such that
$F$ and its conjugate $\overline F$ satisfy the relation:
\begin{equation*}
Gr^p_F Gr^q_{\overline F} (H_{\C}) =
 0, \quad \mbox{for} \,\,  p+q \neq n
  \end{equation*}
\end{definition}

The module $H_A$, or its image in $H_{\Q}$ is called the lattice.
The Hodge Structure is real when $A = \R$, rational when $A = \Q$ and integral
when $A = \Z$.

 \subsubsection{Opposite filtrations} The linear
algebra of Mixed Hodge Structures applies to  an abelian category $\A$  if we use the
following definition where
no conjugation map appears. \\
Two finite filtrations $F$ and $G$
on an object $A$ of ${\A}$ are n-opposite if:
\begin{equation*}
Gr^p_F Gr^q_G(A)= 0\quad \mbox{for} \,\, p+q\neq n.
\end{equation*}
hence the Hodge  filtration $F$ on a Hodge Sructure of weight $n$ is
$n$-opposite to its conjugate
 $\overline F$. The following  constructions
will define an equivalence of categories between  objects of
${\A}$ with two  $n$-opposite filtrations and  bigraded objects of
${\A}$ of the following type.

\begin{example} Let $A^{p,q}$ be a  bigraded object of
${\A}$ such that
  $A^{p,q}=0$, for all but a finite number of
pairs $(p,q)$ and  $A^{p,q} =0$ for $p+q\neq n$; then we define
two  $n$-opposite filtrations on $A= {\sum_{p,q}} A^{p,q}$
$$F^p(A)= {\sum_{p'\geq p}} A^{p',q'} ,\quad
G^q(A)= {\sum_{q'\geq q}} A^{p',q'}$$ 
We have $ Gr^p_F Gr^q_G(A)
= A^{p,q}$.
 \end{example}
 
  We state reciprocally \cite{HII}:
  
 \begin{prop}
i)  Two finite filtrations  $F$ and $G$ on an object $A$
 are $n$-opposite, if and only if:
   $$\forall \, p,q,\quad p+q=n+1 \Rightarrow F^p(A)
 \oplus  G^q(A)\simeq A.$$
ii) If $F$ and $G$ are $n$-opposite, and if we put
$A^{p,q}=F^p(A)\cap G^q(A)$, for $p+q=n$, $A^{p,q}=0$ for $ p+q
\neq n$, then $A$ is a direct sum of $A^{p,q}$.
 \end{prop}
 
  Moreover, $F$
and $G$ can be deduced from the
 bigraded object $A^{p,q}$ of $A$ by the above  procedure.
 We recall the previous equivalent definition of Hodge Structure:
 
 \begin{definition}([HS2])
 An $A$-Hodge Structure on $H$ of weight $ n$ is a pair of
 a finite A-module $H_A$ and a  decomposition into a  direct sum on $
 H_{{\C}}= H_A \otimes_A \C$:
\begin{equation*} H_{{\C}} =
{ \oplus_{p+q=n}} H^{p,q}\quad \mbox{such that}\quad {\overline
H^{p,q}} = H^{q,p}
\end{equation*}
\end{definition}
 The relation with the previous definition is given by  $H^{p,q} =
F^p(H_{{\C}})\cap \overline F^q(H_{{\C}})$ for $ p+q=n$.

\subsubsection{Complex Hodge Structures} For some arguments in analysis, we
don't need to know that ${\overline F}$ is the conjugate of $F$.

\begin{definition}  A complex Hodge Structure  of weight $n$ on a complex vector space
  $H$ is given by a pair of $n-$opposite
  filtrations $F$ and ${\overline F}$, hence a
  decomposition into a  direct sum of subspaces:
\begin{equation*}
H = \oplus_{p+q =n} H^{p,q}, \quad \mbox{where}\,\, H^{p,q}= F^p
\cap {\overline F}^q
\end{equation*}
\end{definition}
The two $n-$opposite
  filtrations $F$ and ${\overline F}$ on a complex Hodge Structure of weight $n$
  can be recovered from the decomposition  by the formula:
  \begin{equation*}
F^p = \oplus_{i\geq p} H^{i, n-i} \quad {\overline F}^q =
\oplus_{i\leq n-q}H^{i, n-i}
\end{equation*}
Here we don't assume the existence of conjugation although we
  keep the notation ${\overline F}$. An $A$-Hodge 
  Structure of weight $n$ defines a complex Hodge Structure
   on  $H = H_{\C}$.
   
 To define polarization,  we recall that the conjugate space
${\overline H}$ of a complex vector space $H$, is the same group
$H$ with a different complex structure, such that the identity map
on the group $H$ defines a real linear map $
 \sigma :  H \rightarrow {\overline H}$ and the
product by scalars satisfy the relation $ \forall \lambda \in \C,
v \in H $, $\lambda \times_{{\overline H}}\sigma ( v ) := \sigma
({\overline \lambda}\times_H v )$, then the complex structure on
${\overline H}$ is unique. On the other hand a complex linear 
morphism $f: V \to V'$ defines a
complex linear conjugate morphism ${\overline f}: {\overline V} \to {\overline V'}$
satisfying ${\overline f}(\sigma ( v )) = \sigma (  f(v))$.

\begin{definition}
 A polarization  of a complex Hodge Structure of weight $n$ is a bilinear morphism
 $S:H\otimes{\overline H} \to \C$ such that:
\[   S(x,  \sigma
 (y) ) = (-1)^n\overline { S(y, \sigma(x) )}\, \hbox{ for }\,  x,y \in L 
 \hbox{ and } S( F^p,
\sigma({\overline F}^q)) = 0 \,\, {\rm for } \,\,  p+q > n. \]
 and moreover $S(C(H) u, \sigma(v))$ is a positive definite Hermitian
 form on $H$ where $C(H)$ denotes the Weil action on $H$.
\end{definition}

\begin{example} A complex Hodge Structure  of weight $0$ on a complex vector space
  $H$ is given by a decomposition into a  direct sum of subspaces
$ H = \oplus_{p\in \Z} H^p $, then $H^{i,j} = 0 $ for $i+j \not=
0$, $F^p = \oplus_{i \geq p} H^p $ and $ {\overline F}^q =
\oplus_{i\leq -q} H^i $.

 A polarization is an Hermitian form on $H$ for which the
decomposition is orthogonal and whose restriction to $H^p$ is
definite for $p$ even and negative definite for odd $p$.
\end{example}

\subsubsection{Examples of Mixed Hodge Structure}
1) A Hodge Structure  $H$ of weight $n$, is a Mixed Hodge Structure 
with weight  filtration:
\begin{equation*}
W_i(H_{{\Q}})=0 \quad \mbox {for} \quad  i<n \qquad \mbox {and}
\quad W_i( H_{\Q}) =
 H_{{\Q}}\quad  \mbox {for}
 \quad i\geq n.
 \end{equation*}
2) Let $(H^i,F_i)$ be a finite family of $A$-Hodge Structures of weight $i\in
{\Z}$; then $H =\oplus_{i} H^i$ is endowed with the following Mixed Hodge Structure: 
\begin{equation*}
H_A= \oplus_{i} H^i_A, \quad W_n = \oplus_{i\leq n} H^i_A \otimes
\Q,\quad
 F^p=\oplus_{i} F^p_i.
 \end{equation*}
 3) Let $H_\Z = i \Z^n \subset \C^n$, then we consider the isomorphism
   $H_\Z \otimes \C \simeq\C^n$ defined with respect to the
   canonical basis $e_j$ of $\C^n$ by:
   \[H_{\C} \xrightarrow{\sim} \C^n: i e_j \otimes (a_j+ i b_j)
   \mapsto i(a_j+ i b_j) e_j = (-b_j + i a_j)e_j \]
   hence the conjugation $\sigma (i e_j \otimes (a_j+ i b_j)) =
   i e_j \otimes (a_j - i b_j)$ on
$H_\C $, corresponds to  the following conjugation on $\C^n$:
$\sigma (-b_j+ i a_j) e_j = (b_j + i a_j) e_j$.\\
4) Let $ H = (H_{\Z}, F, W)$ be a Mixed Hodge Structure; 
its $m-$twist is a Mixed Hodge Structure 
denoted by  $ H(m)$ and defined by:
\begin{equation*}
  H(m)_{\Z}:=  H_{\Z}\otimes (2i\pi)^m \Z,\,
  W_r H(m):=  (W_{r+2m} H_{\Q}) \otimes (2i\pi)^m \Q, \,
  F_r H(m):=  F^{r+m}H_{\C}.
  \end{equation*}
  
 \subsubsection{Tensor product  and Hom }  Let  $H$  and $H'$ be
two Mixed Hodge Structures.\\
1) We define their Mixed Hodge Structure tensor product $H\otimes H'$,  by applying
the above general rules of filtrations, as follows:\\
i) $ (H \otimes H')_{\Z} = H_{\Z} \otimes H'_{\Z}$ \\
ii) $W^r (H\otimes H')_{{\Q}} := \sum_{ p+p' = r} W^p
H_{\Q}\otimes W^{p'} H'_{\Q}$\\
 iii) $F^r (H\otimes H')_{{\C}} :=
\sum_{ p+p' = r} F^p H_{\C}\otimes F^{p'} H'_{\C}$.\\
2) The Mixed Hodge Structure on  $Hom(H,H')$ is defined   as follows:\\
i) $ Hom(H, H')_{\Z} := Hom_{\Z} ( H_{\Z}, H'_{\Z})$ \\
 ii)  $ W_r Hom(H, H')_{\Q} :=  \{ f:
Hom_{\Q}(H_{\Q},H'_{\Q}): \forall n,  f (W_nH)\subset W_{n+r} H'\} $\\
iii)  $ F^r Hom(H, H')_{\C} :=  \{ f:
   Hom_{\C}(H_{\C},H'_{\C}): \forall n,  f (F^nH)\subset F^{n+r} H'\} $.\\
  In particular  the dual  $H^*$ to $H$
is a Mixed Hodge Structure.

 \subsection{Proof of the canonical decomposition of the weight filtration}
 {\it Injectivity of $\varphi$.} Let  $m=p+q$, then
we have the inclusion
 modulo $W_{m-1}$ of  the image
$\varphi(I^{p,q}) \subset H^{p,q} = (F^p \cap \overline
{F^q})(Gr^W_{m}H)$.
 Let  $v\in I^{p,q}$ such that  $\varphi (v) = 0$,
 then $v \in F^p\cap W_{m-1}$ and the class  $cl(v)$
 in  $ (F^p \cap \overline
{F^q})(Gr^W_{m-1}H) = 0$  since  $p+q > m-1$, hence $cl(v)$ must
vanish; so we deduce that $v \in F^p\cap W_{m-2}$. This is a step
in an inductive argument based on the formula $F^p \oplus
\overline {F^{q-r+1}}\simeq Gr^W_{m-r}H$. We want to prove  $v \in
F^p \cap W_{m-r}$ for all $r
> 0$. We just proved this for $r = 2$. Hence we write
\begin{equation*}
 v\in \overline {F^q}\cap W_{m}+ \sum_{r-1 \geq i
\geq 1}\overline {F^{q-i}}\cap W_{m-i-1}+ \sum_{i > r-1 \geq
1}\overline {F^{q-i}}\cap W_{m-i-1}
 \end{equation*}
 where the right term is in $W_{m-r-1}$, since $W_{m-i-1} \subset
 W_{m-r-1}$ for $i > r-1$, hence
   $v \in F^p\cap \overline {F^{q-r+1}}\cap
W_{m-r}$ modulo $W_{m-r-1}$ since $\overline F$ is decreasing.
 As $(F^p\cap \overline
{F^{q-r+1}})Gr^W_{m-r}H = 0$ for $r > 0$, the class  $cl (v) = 0
\in Gr^W_{m-r}H$, then
  $v \in F^p \cap W_{m-r-1}$.  Finally, as
$W_{m-r} = 0 $ for large $r$, we deduce  $v = 0$.\\
 {\it Surjectivity.} Let $\alpha\in H^{p,q}$; there exists
$v_0\in F^p\cap  W_m$ (resp. $u_0 \in  {F^q}\cap  W_m$) such that
$\varphi (v_0) = \alpha = \varphi ({\overline u_0})$, hence $v_0
={\overline u_0}+w_0$ with $w_0\in W_{m-1}$. Applying the formula
$F^p \oplus \overline {F^{q}}\simeq Gr^W_{m-1}$, the class of
$w_0$ decomposes as $cl (w_0) = cl (v') + cl ({\overline u'})$
with $v' \in F^p\cap W_{m-1}$ and $u' \in F^q \cap W_{m-1}$; hence
there exists $w_1 \in W_{m-2}$ such that
  $v_0 ={\overline u_0}+ v' + {\overline u'} + w_1$. Let
$v_1:= v_0 - v'$ and $ u_1 = u_0 + u'$, then
  \[ v_1 = {\overline u_1} + w_1, \quad {\rm where } \,\, u_1 \in F^q
   \cap W_m, F^p \cap W_m, w_1 \in W_{m-2}.\]
 By an increasing inductive argument on $k$, we apply the formula:
$$F^p \oplus \overline {F^{q-k+1}}\simeq Gr^W_{m-k}$$
 to find vectors $v_k, u_k ,w_k $ satisfying:
 \begin{equation*}
\begin{split}
&v_k\in F^p\cap W_m ,\,\, w_k\in W_{m-1-k},\,\, \varphi (v_k) =
\alpha,\,\,  v_k =
{\overline u_k}+ w_k\\
 &u_k\in F^q\cap
W_m + F^{q-1}\cap W_{m-2}+F^{q-2}\cap W_{n-3}+...+F^{q+1-k}\cap
W_{m-k}
\end{split}
\end{equation*}
  then we decompose the class of
  $w_k$ in $Gr^W_{m-k-1}H $  in the inductive step as above.
  For large  $ k$,  $W_{m-1-k}= 0$ and we
  represent $\alpha$ in $I^{p,q}$. \\
  Moreover $W_n=W_{n-1}\oplus (\oplus_{ p+q=n} I^{p,q})$,
  hence by induction we decompose $W_n$ as direct sum of
  $I^{p,q}$ for $p+q \geq n$. \\
  Next we suppose, by induction, the formula for $F^p$ satisfied
  for all $ v \in W_{n-1}\cap F^p$. The image of an element
  $v\in F^p\cap W_n$ in $Gr^W_n H$ decomposes into Hodge components
  of type $(i,n-i)$ with $i\geq p$ since $v\in F^p\cap W_n$.
  Hence the decomposition of $v$ may be writen as $v = v_1 + v_2$
  with $v_1 \in \oplus_{i< p}I^{i,n-i}$ and
  $v_2 \in \oplus_{i\geq p}I^{i,n-i}$ with $v_1\in W_{n-1}$
  since its image vanish in $Gr^W_n H$. Hence, the formula for $F^p$
  follows by the inductive step.
  
\subsubsection{Proof of the abelianness of the category of MHS and strictness}
The definition of Mixed Hodge Structure has a surprising strong property, since any
morphism of Mixed Hodge Structures is necessarily strict for each filtration $W$ and
$F$. In  consequence, the category is abelian.

\begin{lemma} The kernel (resp. cokernel) of a morphism $f$ of Mixed Hodge Structures:
 $H\rightarrow H'$ is a Mixed Hodge Structure  $K$ with underlying module $K_A$ equal
 to the kernel (resp. cokernel) of
 $f_A: H_A \rightarrow {H'}_A$; moreover $ K_A \otimes {\Q}$ and $K_A \otimes {\C}$
are endowed with induced  filtrations  (resp. quotient
filtrations) by $W$  on $H_{ A \otimes \Q}$ (resp. ${H'}_{A
\otimes \Q}$) and $F$ on $H_{{\C}}$ (resp. ${H'}_{{\C}}$).
\end{lemma}

\begin{proof}
A morphism compatible with
 the filtrations is necessarily compatible with the canonical
  decomposition of the Mixed Hodge Structure 
 into $\oplus I^{p,q}$.
  It is enough to check the statement on $K_{\C}$, hence we drop the
  index by $\C$. We consider on $K = Ker (f)$ the induced filtrations from
$H$. The morphism $Gr^WK\rightarrow Gr^W H$ is injective, since it
is injective on the corresponding terms $I^{p,q}$; moreover, the
filtration $F$ (resp. $\overline F$) of $K$ induces on  $Gr^W K$
the inverse image of the filtration $F$ (resp. $\overline F$) on
$Gr^W H$:
\begin{equation*}
Gr^W K = \oplus_{p,q} (Gr^WK)\cap H^{p,q}(Gr^WH)\,\,\mbox{and}\,\,
H^{p,q}(Gr^WK) = (Gr^WK)\cap H^{p,q}(Gr^WH)
\end{equation*}
 Hence the filtrations  $W, F $ on $K$  define a Mixed Hodge Structure on
  $K$ which is a kernel of $f$ in the category of Mixed Hodge Structures.
The statement on the cokernel follows by duality. \end{proof}

 We still need to prove that for   a morphism  $f$ of Mixed Hodge Structures, the canonical
morphism $Coim (f) \to Im (f)$ is an isomorphism of Mixed Hodge Structures. Since by
the above lemma $Coim (f)$ and $Im (f)$ are endowed
with natural Mixed Hodge Structure, the result follows from the statement:\\
{\it A morphism of Mixed Hodge Structures which induces an isomorphism on the
lattices, is an isomorphism of Mixed Hodge Structures.}

 \begin{corollary} i) Each morphism $f: H\rightarrow H'$ is strictly compatible with the
filtrations $W$ on ${H}_{A \otimes \Q}$ and ${H'}_{A \otimes \Q}$
as well the filtrations $F$ on $H_{{\C}}$ and ${H'}_{\C}$. It
induces morphisms  $ Gr^W_n(f) : Gr^W_n(H_{A \otimes
\Q})\rightarrow Gr^W_n({H'}_{A \otimes \Q})$ compatible with the
${A \otimes \Q}$-Hodge Structures, and morphisms $ Gr^p_F(f) :
Gr^p_F(H_{{\C}})\rightarrow Gr^p_F({H'}_{{\C}})$ strictly
compatible with the induced  filtrations by  $W_{{\C}}$.\\
ii)  The functor $Gr^W_n$ from the category of Mixed Hodge Structures to the category
$A\otimes {\Q}$-Hodge Structures of weight $n$ is exact and the functor $Gr^p_F$
is  also exact.
\end{corollary}
\begin{remark} The above result shows that any  exact sequence of
Mixed Hodge Structures  gives rise to various exact sequences which, in the case of
Mixed Hodge Structures on cohomology of algebraic varieties that we are going to
construct, have in general interesting geometrical interpretation,
since we deduce from any long exact sequence of Mixed Hodge Structures:
$$H'^i \to H^i \to H''^i \to H'^{i+1}$$ 
various exact sequences:
$$ Gr^W_nH'^i \to Gr^W_nH^i \to Gr^W_nH''^i \to Gr^W_nH'^{i+1} $$
for $\Q$ (resp. $\C$) coefficients, and: \\
$$Gr_F^nH'^i \to Gr_F^nH^i \to Gr_F^nH''^i \to Gr_F^nH'^{i+1}$$
$$Gr_F^m Gr^W_nH'^i \to Gr_F^m Gr^W_nH^i \to Gr_F^m
Gr^W_nH''^i \to Gr_F^m Gr^W_nH'^{i+1}.$$
  \end{remark}
  
\subsubsection{Hodge numbers} Let $H$ be a Mixed Hodge Structure and set: 
$$ H^{pq}=
Gr^p_F Gr^q_{\overline F} Gr^W_{p+q}H_{\C} = (Gr^W_{p+q}
H_{{\C}})^{p,q}.$$ 
The Hodge numbers of $H$ are the integers
$h^{pq} = \dim_{{\C}}H^{pq}$, that
 is the Hodge numbers $h^{pq}$ of the Hodge Structure $Gr^W_{p+q}H$.

 \subsubsection{Opposite filtrations} Most of the proofs on the
algebraic
 structure of Mixed Hodge Structure may be carried for three filtrations in an
 abelian category defined as follows \cite{HII}
  \begin {definition} [Opposite filtrations] Three
finite filtrations $W$ (increasing), $F$ and $G$ on an object $A$
of ${\A}$ are opposite if
$$Gr^p_F Gr^q_{G} Gr^W_n(A) = 0\quad {\mbox for} \,\, p+q\not=n.$$
 \end {definition}
  This condition is symmetric  in $F$ and $G$.
It means that  $F$ and $G$ induce on $W_n(A)/W_{n-1}(A)$ two
$n$-opposite filtrations, then $Gr^W_n(A)$ is bigraded
\begin{equation*} W_n(A)/W_{n-1}(A)= \oplus_{p+q=n}
A^{p,q}\quad  {\mbox where}\,\, A^{p,q} = Gr^p_F Gr^q_G
Gr^W_{p+q}(A)
\end{equation*}
\begin{example}
i) A bigraded object $A = \oplus A^{p,q}$ of finite bigrading has
the following three opposite filtrations
$$ W_n= \oplus_{p+q\leq n} A^{p,q}\quad  ,F^p=
\oplus_{p'\geq p} A^{p'q'}\quad  ,G^q= \oplus_{q'\geq q}
A^{p'q'}$$
 ii) In the definition of an $A$-Mixed Hodge Structure, the filtration $W_{\C}$ on $H_{\C}$
 deduced from $W$ by saclar extension, the filtration  F and its
  complex conjugate , form a
system $(W_{{\C}},F,\overline F)$ of three opposite filtrations.
 \end{example}
 
 \begin{theorem}[Deligne]  Let $\A$ be an abelian  category and
  ${\A}'$ the category of objects of ${\A}$
   endowed with three opposite filtrations  $W$ (increasing),
   $F$ and $G$. The  morphisms of ${\A}'$
  are morphisms in ${\A}$ compatible with the three  filtrations.\\
i) ${\A}'$ is an abelian category.\\
 ii) The kernel (resp. cokernel)
of a morphism $f: A\rightarrow B$ in ${\A}'$ is the kernel (resp.
 cokernel ) of $f$ in ${\A}$, endowed with the three induced filtrations
from $A$ (resp. quotient of the filtrations on $B$). \\
iii) Any morphism $f : A\rightarrow B$ in ${\A}'$ is  strictly
compatible with the filtrations $W,F$ and $G$; the morphism $Gr^W
(f)$ is compatible with the bigradings of $Gr^W(A)$ and $Gr^W(B)$;
the morphisms
 $Gr_F(f)$ and $Gr_{G}(f)$ are strictly compatible with the
 induced  filtration by $W$.\\
 iv) The forget  the filtrations functors, as well
  $Gr^W, Gr_F, Gr_{G}$, $Gr^W Gr_F \simeq Gr_F Gr^W$,
  $Gr_F Gr_{G} Gr^W$ and $Gr_{G} Gr^W \simeq
   Gr^W Gr_{G} $ of ${\A}'$ in ${\A}$ are exact.
\end{theorem}

\subsection{Complex Mixed Hodge Structures} Although the cohomology of algebraic varieties
carries a Mixed Hodge Structure defined over $\Z$, we may need to work in analysis
without such structure over $\Z$.
 \begin{definition}  A complex Mixed Hodge Structure of weight $n$ on a complex vector space
  $H$ is given by an increasing  filtration $W$ and two
   decreasing filtrations $F$ and $G$ such that
$( Gr^W_k H, F,  G)$,  with the induced filtrations, is a complex HS
of weight $n+k$.
\end{definition}
For $n = 0$, we say a complex Mixed Hodge Structure.  
The definition of complex Hodge Structure
of weight $n$ is obtained  in the particular case when $W_n = H$
and $W_{n-1} = 0$.

 \subsubsection{ Variation of complex Mixed Hodge Structures}
 The structure which appears in deformation theory
  on the cohomology of the fibers of a
  morphism of algebraic varieties leads one to introduce the concept
  of variation of Mixed Hodge Structure.
 \begin{definition} i) A variation (VHS) of complex Hodge Structures on a complex
 manifold $X$ of weight $n$
 is  given by a data
$( \HH, F,  {\overline F})$ where $\HH$ is a complex local system,
$F$ (resp.
   ${\overline F}$) is a
   decreasing filtration varying holomorphically (resp. anti-holomorphically)
    by sub-bundles of the vector bundle
   $\OO_X\otimes_{\C}\HH$ (resp. $\OO_{\overline X}\otimes_{\C}\HH$ on the
   conjugate variety ${\overline X}$ with anti-holomorphic structural sheaf)
   such that for each point $x \in X$, data
$(\HH(x), F(x),  {\overline F}(x))$ form  a Hodge Structure  of weight $n$.
Moreover, the connection $\nabla$ defined by the local system
satisfies Griffiths tranversality: for  tangent vectors $v$
holomorphic and ${\overline u}$ anti-holomorphic
\begin{equation*}
(\nabla_{v} F^p) \subset F^{p-1}, \quad (\nabla_{{\overline u}}
{\overline F}^p) \subset {\overline F}^{p-1}
\end{equation*}
ii)  A variation (VMHS) of complex Mixed Hodge Structures
  of weight $n$ on $X$
 is  given by the following data
\begin{equation*}
( \HH, W,  F,  {\overline F})
\end{equation*}
where $\HH$ is a complex local system, $W$
   an increasing  filtration by sub-local systems (resp.
   ${\overline F}$) is a
   decreasing filtration varying holomorphically (resp. anti-holomorphically)
    satisfying Griffiths tranversality
\begin{equation*}
(\nabla_{v} F^p) \subset F^{p-1}, \quad (\nabla_{{\overline u}}
{\overline F}^p) \subset {\overline F}^{p-1}
\end{equation*}
such that $( Gr^W_k \HH, F,  {\overline F})$, with the induced
filtrations, is a complex VHS  of weight $n+k$.
\end{definition}
For $n = 0$ we just  say a complex Variation of Mixed Hodge Structures.
 Let $\overline {\HH}$ be
the conjugate local system of $\LL$. A linear morphism $S:
\HH\otimes_{\C}\overline {\HH} \to \C_X$ defines  a polarization
of a VHS if it defines a polarization  at each point $x\in X$.
 A complex Mixed Hodge Structure of weight
$n$ is graded polarisable if $(Gr^W_k \HH, F,{\overline F})$ is a
polarized Variation Hodge Structure.

\section{Hypercohomology and spectral sequence of a Filtered Complex}

The abstract definition of Mixed Hodge Structures is  intended to
describe global topological and geometrical  properties of
algebraic varieties and will be established by constructing
special
 complexes of sheaves $K$ endowed with two filtrations $ W$ and $ F$.
 As in the
definition of Mixed Hodge Structures, the filtration $W$ must be
defined on the rational level, while $F$ exists only on the
complex level.

The topological techniques used to construct $W$ on the rational
level are different from the geometrical techniques represented by de Rham complex
used to
construct the filtration $F$ on the complex level.

 Comparison morphisms between the rational  and
complex levels must be added in order  to obtain a satisfactory
functorial theory of Mixed Hodge Structures with respect to
algebraic morphisms. This functoriality on the level of complexes
cannot be obtained in the actual abelian category of complexes of
sheaves but in a modified category called the derived category
\cite{V}, \cite{V2}, \cite{Iver}, that we describe in this
section. 

In this section we are going to introduce the proper terminology. 

In the next section we use  this
 language in order to state rigorously the hypotheses
 on the weight
filtration $W$ and Hodge filtration  $F$ on $K$ needed to induce a
Mixed Hodge Structure on the $i$th-cohomology $( H^i(K), W, F) $
of the complex $K$.

The proof of the existence of the Mixed Hodge Structure is based
on the use of the spectral sequence of the filtered complex $(K,
W)$. In this section we recall what a spectral sequence is and
discuss its behavior with respect to morphisms of filtered
complexes.

The filtration $F$ induces different filtrations on the spectral
sequence in various ways. A detailed study in the next section
will show that these filtrations coincide under adequate
hypotheses and define a Hodge Structure on the terms of the
spectral sequence.

\subsubsection{ Spectral sequence defined  by  a filtered complex $(K,F)$
in an abelian category} We consider  decreasing filtrations. A
change of the sign of the indices transforms an increasing
filtration into a decreasing one, hence any result has a meaning
for both filtrations.

\begin{definition} Let $K$ be a complex of objects of an abelian
category  $ \A$, with a  decreasing filtration  by subcomplexes
$F$. It induces
 a filtration $F$ on the cohomology $H^*(K)$, defined by:
\begin{equation*}
F^iH^j(K) = Im (H^j(F^iK) \rightarrow H^j(K)), \quad \forall i, j
\in \Z.
\end{equation*}
\end{definition}

\vskip.1in
 The spectral sequence defined  by  the filtered complex $(K,F)$
 is a method to compute the graded object $Gr^*_FH^*(K)$.
  It consists  of indexed  objects of $\A$ endowed with differentials
  (see \ref{SS}
 for explicit definitions):
\begin{enumerate}
\item  terms $E_r^{pq}$ for $r > 0, p, q \in \Z$,
 \item differentials
  $ d_r: E_r^{pq}\rightarrow E_r^{p+r,q-r+1} $ such that $\,d_r \circ
  d_r = 0$,
\item  isomorphisms:
 $$E_{r+1}^{pq} \simeq H(E_r^{p-r,q+r-1}
\xrightarrow{d_r} E_r^{pq} \xrightarrow{d_r} E_r^{p+r,q-r+1})$$
 \end{enumerate}
of the $(p,q)$-term of index $r+1$ with
the corresponding cohomology of the sequence with index $r$. To
avoid ambiguity we may write ${_FE_r^{pq}}$ or $E_r^{pq}(K,F)$.
The first term is defined as:
 \[ E_1^{pq} = H^{p+q}(Gr_F^p(K)).\]
 {\it The aim of the spectral sequence} is to compute the terms:
\begin{equation*}
 E_{\infty}^{pq} := Gr_F^p(H^{p+q}(K))
\end{equation*}
 The spectral sequence is said to
degenerate if:
\[ \forall \, p,q, \,\,  \exists \, \,r_0 (p, q) \,\, \hbox{ such that }
\,\, \forall r \geq r_0, \quad E_r^{pq} \simeq E_\infty^{pq}:=
Gr^F_pH^{p+q}(K). \]
 It degenerates at rank $r$ if the
 differentials  $d_i $ of $E_i^{pq}$ vanish for $i \geq r$.\\
 In Deligne-Hodge theory, the spectral sequences of the next section degenerates
 at rank less than $2$.  Most known
 applications are in the case of degenerate spectral sequences
 \cite{g-h}, for example we assume the filtration biregular, that
 is finite in each degree of $K$.  It is often convenient to
 locate the terms on the coordinates $(p,q)$ in some region
 of the plan $\R^2$; it degenerates for example when for $r$ increasing, the
 differential $d_r$ has source or target outside the region of non vanishing
 terms.

 \medskip
 Formulas for the terms  $r > 1$ are  mentioned later in this
section, but we
 will work hard to introduce sufficient conditions to imply specifically
 in our  case
  that $d_r = 0$ for $r > 1$, hence the terms $   E_r^{pq} $ are identical for all
 $r > 1$ and isomorphic to the
 cohomology of the sequence $ E_1^{p-1,q} \xrightarrow{d_1} E_1^{p,q}
 \xrightarrow{d_1} E_1^{p+1,q}$.

\medskip
\n {\it Morphisms of spectral sequences.} A morphism of filtered
complexes  $f:(K,F)\rightarrow  (K',F')$ compatible with the
filtration induces a morphism of the corresponding spectral
sequences.

\medskip
\n {\it  Filtered resolutions}. In presence of two filtrations by
subcomplexes $F$ and $W$ on a  complex  $K$ of
  objects of an abelian  category
${\A}$,  the filtration $F$  induces by restriction a new
filtration $F$ on the terms $W^iK$, which also induces a quotient
filtration $F$ on $Gr^i_WK$. We define in this way the graded
complexes $Gr_FK$, $Gr_WK$ and $Gr_FGr_WK$.

\begin{definition} i) A filtration $F$
on a complex $K$
is called biregular if it is finite in each degree of $K$.\\
ii)  A morphism $f: X \xrightarrow{\approx} Y$ of complexes of
objects of $\A$ is a quasi-isomorphism denoted by $\approx$  if
 the induced morphisms on cohomology $H^*(f): H^*(X) \xrightarrow {\sim}
 H^*(Y)$ are isomorphisms for all degrees.\\
 iii) A morphism $f:(K,F)\xrightarrow {\approx} (K',F)$ of
complexes with biregular filtrations is a filtered
quasi-isomorphism if it is compatible with the filtrations and
induces a quasi-isomorphism on the graded object $Gr^*_F(f):
 Gr^*_F(K)\xrightarrow {\approx}  Gr^*_F(K') $.\\
iv) A morphism $f:(K,F,W)\xrightarrow {\approx}(K',F,W)$ of
complexes
 with  biregular filtrations  $F$  and $W$  is a bi-filtered quasi-isomorphism
   if $Gr^*_F
Gr_W^*(f)$ is a quasi-isomorphism.
\end{definition}
In the case iii) we call $(K',F)$ a filtered resolution of $(K,F)$
and in iv) we say a bi-filtered resolution, while in i) it is just
a resolution.

\begin{prop}
   Let
$f : (K,F) \rightarrow (K',F')$ be a filtered morphism with
biregular filtrations, then the following assertions are
equivalent:\\
i)  $f$ is a filtered quasi-isomorphism.\\
ii)  $E_1^{pq}(f) : E_1^{pq}(K,F) \rightarrow  E_1^{pq}(K',F')$ is
an isomorphism for all $p,q$.\\
iii) $E_r^{pq}(f) : E_r^{pq}(K,F) \rightarrow E_r^{pq}(K',F')$ is
an isomorphism for all $r \geq 1$ and all $p,q$.
\end{prop}

By definition of the terms $E_1^{pq}$, (ii) is equivalent to (i).
We deduce iii) from ii) by induction. If we  suppose the
isomorphism in iii) satisfied for $r \leq r_0$, the isomorphism
for $r_0 + 1$ follows since $E_{r_0}^{pq}(f)$ is compatible with
$d_{r_0}$.

\begin{prop} (Prop. 1.3.2 \cite{HII})\label{delignestrict} Let $K$ be a complex with a biregular
filtration $F$. The following conditions are equivalent: \\
i)  The spectral sequence defined by $F$ degenerates at rank $1$
($E_1 = E_{\infty}$)\\
ii) The differentials $d: K^i \rightarrow K^{i+1}$ are strictly
compatible with the filtrations.
\end{prop}

\subsubsection{Diagrams of morphisms} In the next section, we call
morphism of filtered complexes  $f : (K_1,F_1)\rightarrow (K_2,
F_2) $, a class
   of  diagrams of
morphisms:
$$(K_1, F_1)  \xleftarrow  {g_1 \approx} (K'_1, F'_1)
\xrightarrow {f_1} (K_2, F_2) \quad , \quad (K_1, F_1)
\xrightarrow {f_2} (K'_2, , F'_2) \xleftarrow {g_2 \approx} (K_2,
, F_2)$$  where $g_1$ and $g_2$ are filtered
 quasi-isomorphisms. We think of $f$ as $f_1 \circ g_1^{-1}$ (resp.
 $f =  g_2^{-1} \circ f_2$) by adding the inverse of a quasi-isomorphism
 to the morphisms in  the category. This formal construction is explained
 below.

 It follows that a diagram of filtered  morphisms of complexes
 induces a morphism of the corresponding spectral sequences, but
 the reciprocal statement is not true: the existence of a
 quasi-isomorphism is stronger than the existence of an
 isomorphism of spectral sequences.

The idea to add  the inverse of quasi-isomorphisms to the
morphisms of complexes is due to Grothendieck   (inspired by the
construction of the inverse of a multiplicative system in a ring).
It is an advantage to use such
 diagrams on the complex level, rather than to work with morphisms
 on the cohomology level,  as the theory of derived category
 will show.
 The derived filtered categories described below have been  used
extensively in the theory of perverse sheaves and developments of
Hodge theory in an essential way.

\subsection{Derived filtered category}
 We will use in the next section the language of {\it filtered and
  bi-filtered derived categories},
 to define Mixed Hodge Structure with compatibility between the various data and
 functoriality compatible with the  filtrations.

\subsubsection{ Background on derived category}
The idea of Grothendieck is to construct a new category where the
class of all quasi-isomorphisms become isomorphisms in the new
category. We follow here the construction given by Verdier \cite{V} in two
steps.  In the first step we construct the homotopy category where
the morphisms are classes defined up to homotopy (\cite{V}, see
\cite{Iver},\cite{BBD}), and in the second step, a process of inverting all
quasi-isomorphisms called  localization is carried by a calculus
of fractions similar to the process of inverting a
 multiplicative system in a ring, although in this case the system
 of quasi-isomorphisms is not commutative so that a set of diagram
 relations must be carefully added in the definition of
 a multiplicative system (\cite{Iver},\cite{BBD}).

\subsubsection{The homotopy category $K(\A)$}
Let $\A$ be an abelian category and let $C(\A)$ (resp. $C^+(\A)$,
$C^-(\A)$, $C^b(\A)$) denotes the abelian category of complexes of
objects in $\A$ (resp. complexes $X^*$ satisfying $X^j = 0$,  for
$j<<0$ and for  $j >> 0$, i.e., for $j$ outside a finite interval). \\
An homotopy between two morphisms of complexes $f,g: X^*\to Y^*$
is a family of morphisms $h^j:X^j\to Y^{j-1}$ in $\A$ satisfying
$f^j-g^j=d_Y^{j-1}\circ h^j+h^{j+1}\circ d_X^j$. Homotopy defines
 an equivalence relation on
the additive group $Hom_{C(\A)} (X^*, Y^*)$.

\begin{definition} The category $K(\A)$ has the same object as the
category of complexes $C(\A)$, while the group of morphisms
$Hom_{K(\A)}(X^*, Y^*)$  is the group of morphisms of the two
complexes of $\A$ modulo the homotopy  equivalence relation.
\end{definition}

\subsubsection{ Injective resolutions}. An abelian  category $\A$ is said to
have enough injectives if each object $A \in \A$ is embedded in an
injective object of $\A$. In this case we can give another
description of $Hom_{K(\A)}(X^*, Y^*)$.

 Any complex $ X$ of $\A$
bounded below is quasi-isomorphic to a complex of injective
objects $I^*(X)$ called its injective resolution \cite{Iver}.

\begin{prop}  Given a morphism  $f: A_1 \to A_2 $ in $ C^+(\A)$
   and two injective resolutions $A_i \xrightarrow{\approx}
   I^*(A_i)$ of $A_i$,
   there exists an extension  of $f$ as a morphism of resolutions
   $I^*(f): I^*(A_1) \to I^*(A_2) $; moreover two extensions  of $f$ are
   homotopic. In particular $ Hom_{K^+(\A)}(A_1,A_2) \simeq
 Hom_{K^+(\A)}(I^*(A_1), I^*(A_2))$.
\end{prop}

See lemma $4.4$ in \cite{H1}. Hence,  an
injective resolution of an object in $\A$ becomes unique up to an
isomorphism and functorial in the category
$K^+(\A)$.\\
The category $K^+(\A)$ is only additive, even if  $\A$ is abelian.
Although we keep the same objects as in $C^+(\A)$, the
transformation on $Hom$ is an important  change in the category
since an homotopy equivalence between two complexes (i.e $f:  X
\to Y$ and $g:
  Y \to X$ such that $g\circ f$ (resp.$f\circ g$) is homotopic to the identity)
   becomes an isomorphism.
   
\begin{remark} The $i$th cohomology  of a sheaf $\FF$  on a
topological space $V$, is defined up to an isomorphism as the
cohomology of the global section of an injective resolution $H^i
(I^*(\FF)(V))$. The complex of global sections $I^*(\FF)(V)$ is
defined up to an homotopy in the  category of groups $C^+(\Z)$,
while it is defined up to an
 isomorphism in the homotopy  category of groups $K^+(\Z)$,
  and called the
higher direct image of $\FF$ by the global section functor
$\Gamma$ or the image $ R \Gamma (V, \FF)$ by the derived functor.
\end{remark}

\subsubsection{The derived category $D(\A)$}
 The resolutions  of a given
complex are quasi-isomorphic. If we want to consider all
resolutions as isomorphic, we must invert quasi-isomorphisms of
complexes. We construct  now a new category
 $D(\A)$ with the same objects as $K(\A)$  but with a different
  additive group of morphisms of two
objects $Hom_{D(\A)}(X,Y)$ that we describe.\\
 Let $I_Y$ denotes
the category whose objects are quasi-isomorphisms $s': Y
\xrightarrow{\approx} Y'$ in $K(\A)$. Let $s'': Y
\xrightarrow{\approx} Y''$ be another object. A  morphism
$Y'\xrightarrow{h} Y'' $ satisfying $h \circ s' = s''$
 defines  a morphism $h: s' \to s''$.
 The key property is that  we can take limits in $K(\A)$:
\[ Hom_{D(\A)} (X,Y) := \lim_{\quad \longrightarrow \,I_Y}
 Hom_{K(\A)}(X, Y') \]
By definition, $X \xrightarrow{f} Y'\xleftarrow{s'}Y $ is equal to
$X \xrightarrow{g} Y'' \xleftarrow{s''}Y $ in the inductive limit
if and only if there exists a diagram $Y' \xrightarrow{u} Y'''
\xleftarrow{v}Y''$ such that $u \circ s' = v \circ s''$ is a quasi
isomorphism $s''': Y \to Y'''$ and $u \circ f = v \circ g $ is the
same morphism $X \to Y'''$. In this case, an element of the group
at right may be represented by a symbol $ {s'}^{-1}\circ f$ and
this representation is not unique since in the above limit
${s'}^{-1}\circ f = {s''}^{-1}\circ g $.\\
We summarize what we need to know here by few remarks:\\
1) A morphism $f: X \to Y$ in $D(\A)$ is represented by a diagram
of morphisms:
 \[ X \xleftarrow{u \approx} Z \xrightarrow{g} Y \quad
{\rm or} \quad X \xleftarrow{g} Z \xrightarrow{u \approx} Y \]
 where $u$ is a quasi-isomorphism in $ K(\A)$ or by a sequence
  of such diagrams.\\
 2) When there are enough injectives, the $Hom$  of two objects $A_1$, $A_2$
 in
$D^+(\A)$ is defined by their injective resolutions (lemma 4.5,
Prop. 4.7 in \cite{H1}):
 \[ Hom_{D^+(\A)}(A_1,A_2) \simeq
Hom_{D^+(\A)}(I^*(A_1), I^*(A_2))\simeq Hom_{K^+(\A)}(I^*(A_1),
I^*(A_2))\]
 In particular,  all resolutions of
 a complex are isomorphic in the derived category.

 \subsubsection{ Triangles}
 We define the shift   of a complex $(K,d_K)$, denoted by $TK$ or $K[1]$,
  by shifting the degrees:
\begin{equation*}
 (TK)^i = K^{i+1} , \quad d_{TK} = -d_K
 \end{equation*}
 Let $u : K\rightarrow K'$ be a morphism of
$C^+{\A} $, the cone $C(u)$ is the complex $TK \oplus K'$ with the
differential $\left({-d_K \atop u}{0 \atop d_{K'}}\right)$ . The
exact sequence associated  to $C(u)$ is:
\begin{equation*}
0\rightarrow K'\rightarrow C(u)\rightarrow TK\rightarrow 0
\end{equation*}
Let $h$ denotes an homotopy  from a morphism $u: K \to K'$ to
$u'$, we define an isomorphism  $I_h: C(u) \xrightarrow {\sim}
C(u')$ by the matrix $\left ({Id \atop h} {0 \atop Id}\right )$
acting on $TK \oplus K'$, which commute with the injections of $K'$
in
$C(u)$ and $C(u')$, and with the projections on $TK$. \\
Let $h$ and $h'$ be two homotopies of $u$ to $u'$. A second
homotopy of  $h$ to $h'$, that is a family of morphisms $k^{j+2}:
K^{j+2}\rightarrow K^{'j}$ for $j\in {\Z}$, satisfying $h-h'
=d_{K'}\circ k - k\circ d_K$, defines an homotopy of $I_h$
to $I_{h'}$.\\
A distinguished (or exact) triangle   in $K({\A})$  is a sequence
of complexes isomorphic to the image of an exact sequence
associated to a cone in $C ({\A})$. We remark:\\
1) The  cone over the identity morphism of a complex $X$ is
homotopic to zero.\\
 2) Using the construction of the mapping cylinder complex over a
 morphism of complexes $u: X \to Y $, one can transform
 $u$, up to an homotopy equivalence  into an injective morphism of
 complexes \cite{Iver}).\\
 3) The derived category
is a triangulated  category (that is endowed with a class of
distinguished triangles).
 A  distinguished  triangle in $D({\A})\,$
is  a sequence of complexes isomorphic to the image of a
distinguished  triangle in $K({\A})\,$. Long exact sequences of
cohomologies are associated to triangles.

\subsubsection{Derived functor}  Let  $T : \A \rightarrow  \B$ be a functor
 of abelian categories. We denote also by $T: C^+{\A} \to C^+{\B}$
the corresponding functor on complexes, and by  $can: C^+{\A}
\to D^+{\A}$  the canonical functor, then we construct  a derived
functor:
 \[ R T: D^+{\A} \to D^+{\B} \]
satisfying $R T \circ can = can \circ T$ under the following
conditions:\\
1) the functor   $T : \A \rightarrow  \B$ is left exact; \\
2)  the category
 ${\A}$ has enough injective objects.
  
The construction is dual with the following condition: $T$ is
 right exact and there exists enough projective.
 
\n a) Given  a complex $K$ in $D^+({\A})$, we start by
choosing an injective resolution of $K$, that is a
quasi-isomorphism $i: K \stackrel{\approx}{\rightarrow} I(K)$
where the components of $I$ are injectives in each degree
(see \cite{H1} Lemma 4.6 p. 42).\\
b) We define $RT(K) = T(I(K))$. \\
c) A morphism  $f: K \to K'$ gives rise to a morphism
 $ R T(K) \to R T(K')$ functorially, since $f$
 can be extended
  to a morphism $I(f): T(I(K))\to T(I(K'))$, defined on
  the injective resolutions uniquely   up to homotopy.\\
  In particular, for a
different choice of an injective  resolution $J(K)$ of $K$, we
have an isomorphism $T(I(K))\simeq T(J(K))$ in $D^+ (\B)$.

\begin{definition}
i) The  cohomology $H^j(R T (K))$ is called the hypercohomology
$R^jT(K)$ of $T$ at $K$.\\
ii) An object $A \in \A$  is $T$-acyclic if
 $R^jT(A) = 0$ for $j>0$.
\end{definition}

\begin{remark}
i) If $K \xrightarrow {\approx } K'$ is a quasi-isomorphism of
complexes, $TK$ and $TK'$ are  not  quasi-isomorphisc in general,
while $RT(K)$ and $RT(K')$ must be quasi-isomorphic since the image of
an isomorphism in the derived category must
be an isomorphism.\\
ii) It is important to know that we can use acyclic objects to
compute $RT$: for any resolution $A(K)$ of a complex $K$ $:
K\xrightarrow{\approx} A(K)$), by acyclic objects, $TA(K)$ is
isomorphic to the complex $R T(K)$.
\end{remark}

\begin{example}
 1) The hypercohomology of the global section functor $\Gamma$
 in the case of
  sheaves on a topological space, is equal to the cohomology
 defined via  flasque resolutions or any ``acyclic '' resolution.
\smallskip

\n 2) {\it Extension groups.} The group of morphisms of two
complexes $Hom_{D(\A)}(X^*,Y^*)$
 obtained  in the new category $D(\A)$ called  derived category
 has a significant interpretation as an extension group:
 \[Hom_{D(\A)}(X^*,Y^*[n])= Ext^n(X^*, Y^*)\]
In presence of enough injectives, these groups are derived from
the $Hom$ functor. In general  the group $Hom_{D(\A)}(A,B[n])$ of
two objects in $\A$ may be still interpreted  as the Yoneda
$n-$extension group \cite{Iver}.
\end{example}

\subsubsection{Filtered homotopy categories $ K^+F(\A), K^+F_2(\A)$}
For an abelian category ${\A}$, let
 $ F{\A}\,$ (resp. $F_2 {\A}) $ denotes the category of
 filtered objects (resp. bi-filtered) of ${\A}$ with finite
 filtration(s),
 $C^+ F{\A}$ (resp. $ C^+F_2 {\A}$ ) the
category of complexes of $ F{\A}$ (resp. $ F_2{\A}$ )  bounded on
the left ( zero in degrees near  $-\infty$) with morphisms of
complexes respecting the filtration(s).\\
 Two morphisms $u, u': (K,F,W) \to (K',F,W)$ to
are homotopic if there exists an homotopy  from $u$ to $u'$
 compatible with  the filtrations, then it induces an homotopy on
each term $k^{i+1}: F^jK^{i+1}\rightarrow F'^jK^{'i}$ (resp. for
$W$) and in particular  $Gr_F(u-u')$ (resp. $Gr_F Gr_W (u-u')$) is
homotopic to $0$.\\
 The homotopy category  whose objects are bounded below complexes
   of filtered (resp.  bi-filtered) objects
of ${\A}$, and whose morphisms are equivalence classes modulo
homotopy compatible with  the filtration(s) is denoted by $K^+
F{\A}$ (resp. $ K^+F_2{\A}$).

\subsubsection{ Derived filtered  categories $ D^+F(\A), D^+F_2(\A)$}
They are deduced from $K^+F{\A}$ (resp. $ K^+F_2{\A}$) by
inverting the filtered quasi-isomorphisms (resp. bi-filtered
quasi-isomorphisms).
 The objects of $D^+F{\A}$ (resp.$ D^+F_2{\A}$ ) are complexes of
 filtered objects of $\A$
as  of $K^+F{\A}$ (resp. $ K^+F_2{\A}$). Hence, the morphisms are
represented by diagrams with filtered (resp. bi-filtered)
quasi-isomorphisms.

 \subsubsection{Triangles}
 The  complex  $T(K,F,W)$ and the cone $C(u)$ of a morphism
 $u : (K, F, W)\rightarrow (K', F, W)$ are  endowed naturally with
 filtrations $F$ and $W$. The   exact sequence associated  to $C(u)$ is
compatible with the filtrations. A filtered homotopy $h$ of
morphisms $u$ and $u'$ defines a filtered isomorphism of cones
$I_h: C(u) \xrightarrow {\sim} C(u')$.\\
Distinguished (or exact) triangles   are defined similarly in $K^+
F {\A}$ and  $K^+ F_2{\A}$ as well in $ D^+F{\A}$ and $ D^+F_2
{\A}$. Long filtered (resp. bi- filtered) exact sequences of
cohomologies are associated to triangles.

\subsection{Derived functor on a filtered complex} Let  $T: \A \to \B$
 be a left exact functor of abelian
 categories with enough injectives in $\A$. We want to construct a
derived functor $RT: D^+F{\A} \to D^+F{\B}$ (resp. $RT: D^+F_2
{\A} \to D^+F_2 {\B}$). Given a filtered complex with biregular
 filtration(s) we  define first the image of the filtrations
 via acyclic
  filtered resolutions.
 Then, we remark that the image of a filtered quasi-isomorphism
  is a filtered quasi-isomorphism, hence the construction factors by $RT$
 through the derived filtered category.
 
 We need  to introduce the concept of $T$-acyclic filtered
resolutions.

\subsubsection{Image of a filtration by a left exact functor}
 Let   $(A,F)$ be a filtered object in $\A$, with a finite
filtration. Since $T$ is left exact,  a filtration  $TF$ of $TA$
is defined by the sub-objects $TF^p(A)$.

If $Gr_F(A)$ is
$T$-acyclic, then the objects $F^p(A)$ are $T$-acyclic as
successive extensions of $T$-acyclic objects. Hence, the image by
$T$ of the sequence of acyclic objects:
$$ 0\rightarrow
F^{p+1}(A)\rightarrow F^p(A) \rightarrow Gr^p_F(A) \rightarrow 0$$
is exact; then:

\begin{lemma} If $Gr_FA$ is a $T-$acyclic object, we have
$ Gr_{TF}TA \simeq T Gr_FA$.
\end{lemma}

\subsubsection{} Let $A$ be an object with two finite filtrations
$F$ and $W$ such that $Gr_F Gr_WA$ is {\it $T$-acyclic}, then the
objects $Gr_FA $ and $ Gr_WA$ are $T$-acyclic, as well $F^q(A)\cap
W^p(A)$. As a consequence of acyclicity, the sequences: 
$$
0\rightarrow T(F^q\cap W^{p+1})\rightarrow T(F^q\cap
W^p)\rightarrow T((F^q\cap W^p)/(F^q\cap W^{p+1}))\rightarrow 0$$
are exact, and $T(F^q(Gr^p_WA))$ is the image in $T(Gr^p_W(A)$ of
$T(F^q\cap W^p)$. Moreover,  the isomorphism $Gr_{TW}TA \simeq T
Gr_W A$ transforms the filtration $Gr_{TW}(TF)$ on $Gr_{TW}TA$
into the filtration $T(Gr_W(F))$ on  $TGr_W A$.

\subsubsection{  $ R T :
D^+F({\A}) \rightarrow D^+F({\B})$} 

\

Let $F$ be a  biregular
filtration of $K$.

  A filtered   $T$-acyclic  resolution of $K$ is given by a
  filtered  quasi-isomorphism $i \colon (K, F)\rightarrow (K', F') $
  to a   complex with a  biregular filtration such that $Gr^p_F(K^{n})$
   are acyclic for $T$ for all $p$ and $n$. 
  
\begin{lemma}[Filtered derived functor of
 a left exact functor $T:\A \to \B$]
Suppose we are given
 functorially for each filtered
 complex $(K,F)$ a filtered $T$-acyclic resolution   $i: (K,F)\rightarrow
(K',F')$, we define   $  T' : C^+F({\A}) \rightarrow D^+F({\B}) $
  by the formula $ T'(K,F) = (TK', TF')$.
A filtered quasi-isomorphism $f:(K_1,F_1)\rightarrow (K_2,F_2)$ 
induces an isomorphism $ T'(f) :  T(K_1,F_1)\simeq T(K_2,F_2)$ in
$D^+F({\B})$, hence $T'$ factors through a derived functor
 $ R T : D^+F({\A}) \rightarrow D^+F({\B})$ such that
 $RT (K, F) = (TK',TF')$,  and  we have $Gr_{F}RT(K) \simeq R
T(Gr_FK)$.
\end{lemma}

 In particular for a different choice $(K'',
F'')$ of $(K',F')$ we have an isomorphism $(T K'', TF'')\simeq (T
K', TF')$ in $D^+F({\B})$ and
$$ R T(Gr_FK)\simeq Gr_{TF'}T(K')\simeq Gr_{TF''}T(K'').$$

\begin{example}
In the particular case of interest, where  ${\A}$ is the category
of sheaves of $A$-modules on a topological space $X$, and where
$T$ is the global section functor $\Gamma$ of $ {\A}$ to the
category of modules over the ring $A$, an example of  filtered
$T$-acyclic resolution of $K$ is the simple complex ${\mathcal
G}^*(K)$, associated to the double complex defined by Godement
resolution (\cite{Iver} Chap.II, \S3.6 p.95 or 
\cite{G} Chap.II, \S 4.3 p.167) ${\mathcal G}^*$ in each degree of $K$,
filtered by ${\mathcal  G} ^*(F^p K)$.

This example will apply to the next result for bi-filtered
complexes $(K, W, F)$ and the resolution  $({\mathcal G}^* K,
{\mathcal G}^*W, {\mathcal G}^* F)$ satisfying
 $$ Gr_{{\mathcal G}^*F} Gr_{{\mathcal G}^*W}({\mathcal G}^* K)
 \simeq {\mathcal  G}^*(Gr_F Gr_W K)$$
 \end{example}
 
 \subsubsection{$R T :
D^+F_2({\A}) \rightarrow D^+F_2({\B})$}

\

Let $F$ and $W$ be two
biregular filtrations of $K$.

 A  bi-filtered  $T$-acyclic resolution  of $K$ is a  bi-filtered quasi-
 isomorphism $i: (K , W, F) \rightarrow (K', W', F')$ of $(K,W, F)$ to a
  bi-filtered complex biregular for each  filtration
 such that  $Gr^p_F Gr^q_W(K^{'n})$ are acyclic for  $T$ for all $p$, $q$ and $n$.

\begin{lemma}
Suppose we are given
 functorially for each bi-filtered
 complex $(K,F,W)$ a bi-filtered $T$-acyclic resolution   $i: (K,F,W)
 \rightarrow
(K',F',W')$, we define   $  T' : C^+F({\A}) \rightarrow
D^+F({\B})$ by the formula $ T'(K,F,W) = (TK', TF', T W')$.\\
A bi-filtered quasi-isomorphism  $f:(K_1,F_1,W_1)\rightarrow
(K_2,F_2,W_2)$ induces an isomorphism $ T'(f) :
T'(K_1,F_1,W_1)\simeq T'(K_2,F_2,W_2)$ in $D^+F_2({\B})$, hence
$T'$ factors through a derived functor
 $ R T : D^+F_2({\A}) \rightarrow D^+F_2({\B})$ such that
 $RT (K,F,W) = (TK',TF',TW')$, and we have
  $Gr_{F}Gr_{W}RT(K) \simeq R T(Gr_FGr_W
 K)$.
\end{lemma}

 In particular for a different choice $(K'',
F'',W'')$ of $(K',F',W')$ we have an isomorphism $(T K'',
TF'',TW'')\simeq (TK', TF',TW'')$ in $D^+F_2({\B})$ and
$$ R T(Gr_FGr_{W}K)\simeq Gr_{TF'}Gr_{TW'}T(K')\simeq Gr_{TF''}T(K'')$$

\subsection{The spectral sequence defined by a  filtered complex}\label{SS}
A decreasing filtration  $F$ of a complex $K$ by sub-complexes
induces a filtration still denoted $F$ on its cohomology
$H^*(K)$. 

The aim of a spectral sequence is to compute the
associated graded cohomology $Gr_F^* H^*(K)$ of the filtered group
($ H^*(K), F $), out of the cohomology  $H^*(F^iK/F^jK)$ of the
various indices of the filtration. The spectral sequence
$E^{p,q}_r(K,F)$ associated to $F$ (\cite{2 C}, \cite{HII}) leads
for large $r$ and under mild conditions, to such graded cohomology
defined by the filtration. 

A morphism in the derived filtered
category defines a natural morphism of associated spectral
sequences; in particular a quasi-isomorphism defines  an
isomorphism. 

Later we shall study Mixed Hodge Complex where the 
weight spectral sequence
becomes interesting and meaningful since it contains geometrical
information and degenerates at rank $2$. Now, we give the definition of the terms of the 
spectral sequence and some examples.

\n To define the spectral
terms $E_r^{pq}(K,F)$ or ${_FE}_r^{pq}$ or simply $E_r^{pq}$ with
respect to $F$, we put for   $r > 0$ and $p,q \in \Z$:
\begin{equation*}
\begin{split}& Z_r^{pq} = Ker(d: F^p K^{p+q} \rightarrow
K^{p+q+1}/ F^{p+r}K^{p+q+1} )\\
&B_r^{pq} = F^{p+1}K^{p+q} + d(F^{p-r+1} K^{p+q-1})
\end{split}
\end{equation*}
 Such formula still makes sense for $r = \infty$ if we set, for a filtered
  object
$(A,F)$, $F^{- \infty}(A)=A$ and $F^{\infty}(A) = 0$:
\begin{equation*}
\begin{split}& Z_{\infty}^{pq}=Ker(d:F^p K^{p+q} \rightarrow
K^{p+q+1})\\
 &B_{\infty}^{pq} = F^{p+1} K^{p+q}+ d (K^{p+q+1})
  \end{split}
\end{equation*}
 We set by definition:
\[E_r^{pq} = Z_r^{pq}/(B_r^{pq} \cap Z_r^{pq})\quad E_{\infty}^{pq}
= Z_{\infty}^{pq}/B_{\infty}^{pq} \cap Z_{\infty}^{pq}\]
 The
notations  are similar to \cite{HII} but different from \cite{G}.

\begin{remark}
In our case, in order to obtain some arguments by duality, we note
the following equivalent  dual definitions to $Z_r^{pq}$ and
$Z_{\infty}^{pq}$:
\begin{equation*}
\begin{split}& K^{p+q}/B_r^{pq} = coker(d: F^{p-r+1} K^{p+q-1}
\rightarrow K^{p+q}/F^{p+1}(K^{p+q} ))\\
& K^{p+q}/B_{\infty}^{pq} = coker(d:K^{p+q+1}
  \rightarrow K^{p+q}/F^{p+1} K^{p+q})\\
&E_r^{pq} =Im(Z_r^{pq} \rightarrow
 K^{p+q}/B_r^{pq}) = Ker(K^{p+q}/B_r^{pq} \rightarrow
 K^{p+q}/(Z_r^{pq}+B_r^{pq})).
\end{split}
\end{equation*}
\end{remark}
  The term
$Gr_F^p(H^{p+q}(K))$ is said to be the limit of the spectral
sequence.  If the filtration is biregular
 the terms $E_r^{pq}$  compute this limit after a
finite number of steps; for each $p,q$ there exists $r_0$ such
that:
\[ Z_r^{pq} = Z_{\infty}^{pq}, \quad
 B_r^{pq} = B_{\infty}^{pq}, \quad  E_r^{pq} = E_{\infty}^{pq}, \quad
 \forall r > r_0 \]
 Note that in some cases, it is not satisfactory to get only the
 graded cohomology and
this is one motivation to be not happy with spectral sequences and
prefer to keep the complex as in the derived category. \\

\begin{lemma}
For each $r$, there exists a differential $d_r$ on the terms
$E_r^{pq}$ with the property that its cohomology is exactly
$E_{r+1}^{pq}$:
\begin{equation*}
E_{r+1}^{pq} = H(E_r^{p-r,q+r-1} \xrightarrow{d_r} E_r^{pq}
\xrightarrow{d_r} E_r^{p+r,q-r+1}).
\end{equation*}
 where  $d_r$ is induced by the differential $d: Z_r^{pq} \rightarrow
 Z_r^{p+r,q-r+1}$.
 \end{lemma}

  We define $d_r$ later with new notations  for an increasing filtration. \\
   For $r< \infty$, the terms  $E_r$ form a complex of
  objects, graded by the degree $p-r(p+q)$ (or a
 direct sum of complexes with index $p-r(p+q))$.
The  first term may be written as:
\begin{equation*}
E_1^{pq} = H^{p+q}(Gr_F^p(K))
\end{equation*}
so that the differentials $d_1$ are obtained as connecting
morphisms defined by the short exact sequences of complexes
\begin{equation*}
0 \rightarrow Gr_F^{p+1}K \rightarrow F^pK/F^{p+2}K \rightarrow
Gr_F^pK \rightarrow 0.
\end{equation*}
It will be convenient to set  for $r = 0$, $ E_0^{pq} =
Gr_F^p(K^{p+q})$.\\
 The spectral sequence  $E_i^{p,q}(K,F)$ is
said to degenerate at rank $r$ if the differentials $d_i$ are zero
for $i \geq r$ independently of $p,q$, then we have in this case
\begin{equation*}
E_r^{pq} = E_i^{pq} = E_{\infty }^{pq}\quad for\ i\geq r .
\end{equation*}
There is no easy general construction for ($E_r,d_r$) for $r> 0$;
however we will see that in the case of Mixed Hodge Structures of
interest to us the terms with respect to the weight filtration $W$
have  a geometric meaning and degenerate at rank $2$:
${_WE}_2^{pq} = {_WE}_\infty^{pq}$.

  \subsubsection{Equivalent notations for increasing filtrations} For an increasing filtration
$W$ on $K$, the  precedent formulas are transformed by the usual
change of indices to pass from $W$ to a decreasing filtration $F$,
that is $F^i = W_{-i}$.

Let $W$ be an increasing filtration on $K$. We set for all $j$, $n
\leq m$ and $n \leq i \leq m$:
\begin{equation*}
W_i H^j(W_mK/W_nK) = Im(H^j(W_iK/W_nK) \rightarrow H^j(W_mK/W_nK)
\end{equation*}
then we adopt the following  new notations for the terms, for all
$r \geq 1, p$ and $q$:

\begin{lemma} The terms of the spectral sequence for ($K, W$)
are equal to:
\begin{equation*}
E_r^{pq}(K, W)) = Gr_{-p}^W H^{p+q}(W_{-p+r-1}K/W_{-p-r}K).
\end{equation*}
\end{lemma}

\begin{proof}
 Let ($K_r^p,W$) denotes the quotient complex $K^p_r := W_{-p+r-1}K/W_{-p-r}K $
 with   the induced filtration by subcomplexes $W$; we put:
\begin{equation*}
\begin{split}
&Z_{\infty}^{pq}(K^p_r, W) := Ker\,(d:
(W_{-p}K^{p+q}/W_{-p-r}K^{p+q})\rightarrow
(W_{-p+r-1}K^{p+q+1}/W_{-p-r}K^{p+q+1}))\\
&B_{\infty}^{pq}(K^p_r, W):=
(W_{-p-1}K^{p+q}+dW_{-p+r-1}K^{p+q-1})/W_{-p-r}K^{p+q}
\end{split}
\end{equation*}
which coincide, {\it up to the quotient by $W_{-p-r}K^{p+q}$},
with $Z_r^{pq}(K, W)$ (resp. $B_r^{pq}(K, W)$) with:
\begin{equation*}
\begin{split}
&Z^{p,q}_r:= Ker\,(d:W_{-p}K^{p+q}\rightarrow
K^{p+q+1}/W_{-p-r}K^{p+q+1})\\
 & B^{p,q}_r:= W_{-p-1}K^{p+q} + d W_{-p+r-1}K^{p+q-1}
\end{split}
\end{equation*}
then, we define:
\[E_{\infty}^{pq}(K^p_r, W) = {Z_{\infty}^{pq}(K^p_r, W) \over
B_{\infty}^{pq}(K^p_r, W)\cap Z_{\infty}^{pq}(K^p_r, W)} =
Gr_{-p}^WH^{pq}(W_{-p+r-1}K/W_{-p-r}K)\]
 and find:
\begin{equation*}
\begin{split}&E_r^{pq} (K,W)= Z_r^{pq}/ (B_r^{pq} \cap Z_r^{pq})
= Z_{\infty}^{pq}(K^p_r, W)  / (B_{\infty}^{pq}(K^p_r, W) \cap
Z_{\infty}^{pq}(K^p_r, W))\\
&  = E_{\infty}^{pq}(K^p_r, W)  =
Gr_{-p}^WH^{pq}(W_{-p+r-1}K/W_{-p-r}K)
\end{split}
\end{equation*}
\end{proof}

  To define the differential $d_r$,
 we consider the exact sequence:
\begin{equation*}
0 \rightarrow W_{-p-r}K/W_{-p-2r}K\rightarrow  W_{-p+r-1}
K/W_{-p-2r}K\rightarrow W _{-p+r-1}K/W_{-p-r}K\rightarrow 0
\end{equation*}
and  the connecting morphism:
\begin{equation*}
H^{p+q}(W_{-p+r-1}K/W_{-p-r}K) {\buildrel \partial \over
\rightarrow } H^{p+q+1}(W_{-p-r}K/W_{-p-2r}K)
\end{equation*}
 the injection $W_{-p-r}K \rightarrow W_{-p-1}K$ induces
 a morphism: 
 \[ \varphi: H^{p+q+1} (W_{-p-r} K/W _{-p-2r}K) \rightarrow
 W_{-p-r} H^{p+q+1}(W_{-p-1} K/W_{-p-2r} K)\].
 Let $\pi$ denote the projection on the right term below, equal to
 $E_r^{p+r,q-r+1}$:
  \[ W_{-p-r} H^{p+q+1}(W_{-p-1} K/W_{-p-2r} K)
\xrightarrow{\pi} Gr_{-p-r}^W H^{p+q+1}(W_{-p-1} K/W_{-p-2r} K)
\] the composition of morphisms $\pi \circ \varphi \circ
\partial$ restricted to $W_{-p}H^{p+q}(W_{-p+r-1}K/W_{-p-r}K)$
induces the differential:
\begin{equation*}
d_r : E_r^{pq} \rightarrow E_r^{p+r,q-r+1}
\end{equation*}
while the injection $W_{-p+r-1} \rightarrow W_{-p+r}K$ induces the
isomorphism:
  \begin{equation*}
  H(E_r^{pq} , d_r) \xrightarrow {\sim } E_{r+1}^{pq} =
   Gr_{-p}^W H^{p+q}(W_{-p+r}K/W_{-p-r-1}K).
   \end{equation*}

\subsubsection{Hypercohomology spectral sequence} Let $T : {\A}
\rightarrow {\B}$ be a left exact functor of abelian categories,
and $(K,F)$ an object of $D^+F{\A}$ and $R T(K,F):D^+F{\A}\to
D^+F{\B} $ its derived functor. The spectral sequence defined by
 the complex $R T(K,F)$ is written as:
\begin{equation*}
 {_FE}_1^{p,q} = R^{p+q}T(Gr_F^p) \Rightarrow Gr_F^p R^{p+q}T(K)
 \end{equation*}
This is the hypercohomology spectral sequence  of the filtered
complex  $K$. For an increasing filtration $W$ on $K$, we have:
$$ {_WE}_1^{p,q} = R^{p+q}T(Gr_{-p}^W) \Rightarrow Gr_{-p}^W R^{p+q}T(K)$$
 It depends functorially on $K$ and a filtered quasi-isomorphism
  induces an  isomorphism of spectral sequences.
 The differentials $d_1$ of this spectral sequence are the image
 by $T$ of the
 connecting morphisms defined by the short  exact sequences:
  $$0 \rightarrow Gr_F^{p+1}K \rightarrow F^pK/F^{p+2}K
  \rightarrow Gr_F^pK \rightarrow 0.$$

\subsubsection{Examples} 1) Let $K$ be a complex, the canonical
filtration $\tau$ is the increasing filtration by sub-complexes:
$$\tau_{\leq p}\  = (\cdots \to K^{p-1} \to Ker d \to 0 \cdots \to
0 )$$
  then: 
  $$ Gr^{\tau}_{\leq p} K \xrightarrow{\approx} H^p ( K)[-p],
  \quad
   H^i(\tau_{\leq p}(K)) = H^i ( K) \,\,{\rm if }\,\ i \leq
p,\,\, {\, \rm and \,\,} 0  \,\, {\rm if \, }\, i > p.$$
 2) The sub-complexes  of $K $:
 $$ \sigma_{\geq p} K := K^{*\geq p} = (0\to \cdots \to 0 \to K^p \to K^{p+1}\to \cdots) $$
 define a decreasing  biregular filtration, the trivial
filtration of $K$ such that $Gr^p_\sigma K = K^p [-p]$,
i.e., it coincides with the Hodge filtration on de Rham complex. 

  A quasi-isomorphism
$ f : K \rightarrow K'$ is necessarily a filtered
quasi-isomorphism for both, the trivial and the canonical
filtrations. The hypercohomology spectral sequences of a left
exact functor attached to the trivial and canonical filtrations of
$K$ are the two natural hypercohomology spectral sequences of
$K$.\\
 3)  Let $f : X \rightarrow Y$ be a continued map
of  topological spaces. Let ${\mathcal   F}$ be an abelian sheaf
on $X$ and ${\mathcal   F}^*$ a resolution of ${\mathcal   F}$ by
 $f_*-$acyclic sheaves, then $R^if_* {\mathcal   F} \cong
H^i(f_* {\mathcal   F}^*)$.  The hypercohomology spectral sequence
of the global section functor $\Gamma (Y,* )$ of the complex $ R
f_* {\mathcal F}^*$ with its canonical filtration,  is:
$${_{\tau}E}_1^{pq} = \H^{p+q}(Y,R^{-p}f_*{\mathcal   F}[p]) \simeq
H^{2p+q}(Y,R^{-p}f_*{\mathcal   F})  \Rightarrow Gr^{ \tau}_{-p}
H^{p+q}(X,{\mathcal   F}).$$
 This formula illustrates basic difference in Deligne's notation:
 the sheaf $R^{-p}f_*{\mathcal   F}[p]$ is in degree $-p$.\\
 In classical notations, Leray's spectral
sequence for $f$ and ${\mathcal   F}$ starts at:
 \[E_2^{pq} = H^p(Y,R^qf_* {\mathcal  F})\]
 To relate both notations  we need to  renumber  the classical term
$  E_{r+1}^{2p+q,-p}$ into the new term  $E_r^{pq}$.

\section{Mixed Hodge Complex (MHC)}

In this section, we give the definition  of a Mixed Hodge Complex
(MHC) and  prove  Deligne's fundamental theorem that the
cohomology of a Mixed Hodge Complex is endowed with a Mixed Hodge
Structure.

  First, on an algebraic variety $V$, we define a cohomological version of
  a  Mixed Hodge Complex, that we call Cohomological 
  Mixed Hodge Complex (CMHC), which is defined essentially
 by a bi-filtered complex of sheaves $(K_{\C},F,W) $ where the filtration $W$ is
 rationally defined and satisfies precise conditions sufficient
 to define a Mixed Hodge Complex structure on
 the global section functor $R\Gamma (V, K)$. The results proved by Deligne
  are technically difficult and so strong that
 the theory is reduced to constructing such a Cohomological  Mixed Hodge
 Complex  for all algebraic
varieties in the remaining sections; hence the theoretical path to
 construct a Mixed Hodge Structure on a variety follows the pattern:
 \[ \hbox{CMHC} \Rightarrow
\hbox{ MHC} \Rightarrow \hbox{MHS}\]
It is true that a direct
study of the logarithmic complex
   by Griffiths and Schmid \cite{G-S} is very attractive, but the initial work of
    Deligne
   is easy to apply, flexible and helps to go beyond this case
  towards  a general theory.
  
 The de Rham complex of a smooth compact complex
variety is a special case of a Mixed Hodge Complex, called a Hodge
complex (HC) with the characteristic property that it induces a
Hodge Structure on its hypercohomology. We start by rewriting the
Hodge theory that we know, with terminology that is fitted to
our generalization.

 Let $A$ denote $\Z$, $\Q$ or $\R$ and  $A \otimes \Q$ the field
 $\Q$ or $\R$ accordingly as in section $3$, $D^+ (\Z) )$ (resp.
 $D^+ (\C) )$, $D^+ (V,\Z) )$, $D^+ (V,\C) )$)
denotes the derived category of $\Z-$modules (resp. $\C-$vector
spaces and corresponding sheaves on $V$).

\begin{definition}[Hodge Complex (HC)] A Hodge
$A$-complex $ K$ of weight $n$ consists of:\\
i)  A complex $K_A$ of $A$-modules, such that $H^k(K_A)$
is an $A$-module of  finite  type for all $k$;\\
  ii) A filtered complex  $(K_{\C},F) $ of $\C-$vector spaces;\\
   iii) An isomorphism
$\alpha : K_A \otimes \C \simeq K_{\C}$ in $D^+(\C)$.\\
The following axioms  must be satisfied:\\
(HC 1) the differential $d$ of $K_{\C}$ is strictly compatible
with the  filtration $F$, i.e., $d^i: (K_A^i, F) \to (K_A^{i+1}, F)$ is
strict, for all $i$; \\
(HC 2) for all $k$,
 the filtration $F$ on $H^k(K_{\C}) \simeq H^k(K_A) \otimes \C$
 defines an $A$-Hodge Structure of weight $n+k$ on $H^k(K_A)$.
 \end{definition}

  Equivalently, in (HC1)  the
spectral sequence defined by $(K_{\C},F)$ degenerates at
$E_1(E_1=E_{\infty})$ (see \ref{delignestrict} or \cite{HII}), and in (HC2) the
filtration $F$ is
 ($n+k$)-opposed to its complex conjugate  (conjugation makes sense since
  $A \subset \R)$.
  
\begin{definition} Let $X$ be a topological space. An
$A$-Cohomological Hodge Complex (CHC) $K$ of weight  $n$ on $X$,
consists of:  \\
i) A complex of sheaves $K_A $ of $A-$modules on  $X$;\\
  ii) A filtered complex of sheaves $(K_{\C},F) $ of $\C$-vector spaces
  on $X$;\\
 iii) An isomorphism $\alpha:   K_A\otimes \C  \xrightarrow {\approx}
 K_{\C}$ in  $D^+(X,\C)$ of $\C-$sheaves on $X$.\\
Moreover, the following axiom must be satisfied:\\
(CHC) The  triple  $(R \Gamma(K_A), R \Gamma (K_{\C},F), R \Gamma
(\alpha))$ is a Hodge Complex of weight $n$.
\end{definition}

If $(K,F)$ is a Hodge Complex (resp. Cohomological Hodge 
Complex) of weight  $n$, then $(K[m], F[p])$
is a Hodge Complex (resp. Cohomological Hodge 
Complex) of weight $n+m-2p$.

The following statement is a new version of Hodge decomposition Theorem:

\begin{theorem} Let $X$ be a compact complex
algebraic manifold and consider:\\
i) $K_{\Z}$ the  complex reduced  to a  constant sheaf $\Z$ on $X$
in degree zero; \\
 ii)  $K_{\C}$ the analytic de Rham
complex $\Omega ^*_X$ with its  trivial filtration  $F^p = \Omega
^{*\geq p}_X$ by subcomplexes:
\begin{equation*} F^p \Omega ^*_X:= 0 \rightarrow 0 \cdots 0 \rightarrow
\Omega ^p_X \rightarrow \Omega ^{p+1}_X \rightarrow \cdots
\rightarrow \Omega ^n_X \rightarrow 0;
\end{equation*}
iii) The quasi-isomorphism $\alpha \colon K_{\Z} \otimes \C
\xrightarrow{\approx}  \Omega
^*_X$  (Poincar\' e lemma).\\
 Then $(K_{\Z}, (K_{\C},F), \alpha)$ is a Cohomological Hodge Complex
 of weight $0$; its
 hypercohomology on $X$, isomorphic to the cohomology of $X$, carries a
  functorial  Hodge Structure for morphisms of algebraic  varieties.
\end{theorem}

 The
new idea here is to observe the degeneracy of the spectral sequence
of $(\Omega^*_X,F)$ and deduce the definition of the Hodge filtration
from the trivial filtration $F$ on the de Rham complex without any
reference to harmonic forms, although the proof of the
decomposition is given  via a reduction  to the case of a
projective variety, hence a compact K\"{a}hler manifold:

\begin{definition} The Hodge filtration $F$ is defined on
the cohomology as follows:
\begin{equation*}
F^p H^i(X, \C):= F^p \H^i(X, \Omega^*_X):= Im (\H^i(X, F^p
\Omega^*_X) \rightarrow \H^i(X, \Omega^*_X)),
\end{equation*}
where the first equality follows from holomorphic Poincar\'e
lemma on the resolution of the constant sheaf $\C$ by the analytic
de Rham complex $\Omega^*_X$.
\end{definition}

\begin{prop}[Deligne \cite{DI}] Let $X$ be a smooth compact complex algebraic
variety, then the filtration $F$ induced on cohomology by the
 filtration $F$ on the de Rham complex is a Hodge filtration.
 \end{prop}
 
  The proof is based on the  degeneration at rank one of the  spectral
sequence with respect to $F$ defined as follows:
\begin{equation*}
{_FE}_1^{p,q}:= \H^{p+q}(X, Gr^p_F \Omega^*_X) \simeq
H^q(X,\Omega^p_X) \Rightarrow Gr^p_F \H^{p+q}(X,  \Omega^*_X).
\end{equation*}
The degeneration at rank one may be deduced from classical  Hodge
decomposition, but it has  been also obtained by direct algebraic
methods by Deligne and Illusie in \cite{Del-Ill} for the algebraic
de Rham complex with respect to Zariski topology on which the
 filtration $F$ is also defined.
 
 The  isomorphism on complex smooth algebraic
varieties between analytic and algebraic de Rham hypercohomology
defined respectively with analytic and algebraic de Rham complexes
is given by Grothendieck's comparison theorem (see \cite{Gr2}).

 Now, we define the structure including two filtrations by weight $W$
 and  $F$ needed on a complex, in order to
 define a Mixed Hodge Structure on its cohomology.
 
\begin{definition}[MHC] An $A$-Mixed Hodge Complex (MHC) $K$ consists
 of:\\
 i) A complex $K_A $ of $A-$modules such that $H^k(K_A)$ is an
$A$-module of finite type for all~ $k$;\\
 ii) A filtered  complex
 $(K_{A \otimes \Q}, W) $ of
$(A \otimes \Q)-$vector spaces with an increasing  filtration $W$;\\
 iii) An isomorphism $ K_A \otimes \Q \xrightarrow {\sim}
   K_{A \otimes\Q}$ in $D^+(A \otimes \Q )$;\\
iv) A bi-filtered
 complex ($K_{\C},W,F)$ of $\C-$vector spaces  with an increasing
 (resp. decreasing) filtration $W$ (resp. $F$) and an
isomorphism: 
$$\alpha:
(K_{A\otimes\Q},W)\otimes \C  \xrightarrow {\sim} (K_{\C},W)$$ 
in $ D^+F(\C)$.\\
Moreover, the following axiom is satisfied: \\
(MHC) For all $n$, the system consisting of\\
- the complex  $Gr_n^W(K_{A \otimes \Q}) $ of $(A \otimes
\Q)-$vector spaces,\\
- the complex $Gr_n^W(K_{\C}, F)$ of $\C-$vector spaces with
induced $F$ and\\
- the isomorphism
 $Gr_n^W(\alpha):Gr_n^W(K_{A \otimes \Q}) \otimes \C \xrightarrow {\sim}
  Gr_n^W(K_{\C})$,\\
 is an $A \otimes \Q$-Hodge Complex of weight $n$.
 \end{definition}
 
The above  structure has
a corresponding structure on a complex of sheaves on $X$ called a
Cohomological Mixed Hodge Complex:

 \begin{definition}[CMHC] An $A$-Cohomological Mixed Hodge Complex $K$
(CMHC) on a topological space  $X$ consists of:\\
i) A complex  of sheaves $K_A$ of sheaves of $A$-modules on $X$
such that $H^k(X,K_A)$ are $A$-modules of finite type;\\
 ii) A filtered  complex $(K_{A \otimes \Q}, W)$ of sheaves of $(A \otimes
\Q)-$vector spaces on $X$
   with an increasing  filtration $W$ and an
  isomorphism
 $K_A \otimes \Q  \simeq K_{A \otimes \Q}$ in $D^+(X,A \otimes
 \Q)$;\\
iii) A bi-filitered  complex of sheaves $(K_{\C}, W, F)$ of
$\C-$vector spaces on $X$ with an increasing
 (resp. decreasing) filtration $W$ (resp. $F$ ) and an
 isomorphism:
  $$\alpha : (K_{A \otimes \Q}, W) \otimes \C \xrightarrow {\sim} (K_{\C}, W)  $$
   in $D^+F(X, \C)$.\\
Moreover, the following axiom is satisfied: \\
(CMHC) For all $n$, the system consisting of: \\
- the complex $Gr_n^W(K_{A \otimes \Q}) $ of sheaves of  $(A
\otimes \Q)$-vector spaces on $X$; \\
- the complex $Gr_n^W(K_{\C}, F)$ of sheaves of  $ \C$-vector
spaces on $X$ with induced  $F$;\\
- the isomorphism
 $Gr_n^W(\alpha):Gr_n^W(K_{A \otimes \Q}) \otimes \C \xrightarrow {\sim}
  Gr_n^W(K_{\C})$,\\
 is an $A \otimes \Q$-Cohomological Hodge Complex of weight $n$.
 \end{definition}

 If $(K,W,F)$ is a Mixed Hodge Complex (resp. Cohomological Mixed Hodge Complex),
then for all $m$ and $n \in \Z, (K[m],W[m-2n],F[n])$ is a Mixed
Hodge Complex (resp. Cohomological Mixed Hodge Complex).

 In fact, any Hodge Complex or Mixed Hodge Complex
 described here is obtained from de Rham complexes with
modifications (at infinity) as the logarithmic complex described
later in the next section. A new construction of Hodge Complex has
been later introduced with the theory of differential modules and
perverse sheaves \cite{BBD} and \cite{Se1} but will not be covered
in these
lectures. 

Now we describe how  we deduce first the Mixed Hodge Complex from
a Cohomological Mixed Hodge Complex, then a Mixed Hodge 
Structure from a Mixed Hodge Complex.

\begin{prop}  If $K=(K_A,(K_{A \otimes
\Q},W), (K_{\C},W,F))$ and the isomorphism $\alpha$ is an
$A$-Cohomological Mixed Hodge Complex
then:
 $$ R \Gamma K = (R \Gamma K_A, R \Gamma (K_{A \otimes
\Q},W), R \Gamma(K_{\C},W,F))$$ 
with $R\Gamma (\alpha)$ is an
$A$-Mixed Hodge Complex.
\end{prop}

The main result of Deligne in \cite{HII} and \cite{HIII}
 is algebraic and states in short: 

\begin{theorem}[Deligne] The cohomology of a Mixed Hodge 
Complex carries a Mixed Hodge Structure.
\end{theorem}

The proof of this result requires a detailed study of spectral
sequences and will occupy the rest of this section. We give
here, before the proof,  the
properties of the various spectral sequences which
may be of interest as independent results.

 Precisely, the weight
spectral sequence of a Mixed Hodge complex is in the category of
Hodge Structure.

So, the Mixed Hodge Structure on cohomology is approached step by step by Hodge
Structures on the terms of the weight spectral sequence
${_WE_r^{pq}}$ of $(K_{\C}, W)$. However, the big surprise is that
the spectral sequence degenerates quickly, at rank two for $W$ and
at rank one for $F$: this is all which is needed.
\vskip.1in

A careful study shows  that there are various  induced filtrations
by $F$ on the
  weight spectral sequence ${_WE_r^{pq}}$ of $(K_{\C}, W)$.
 Two direct filtrations
$F_d $ and $F_{d^*}$ and one recurrent $F_{rec}$ are induced by
$F$ on the terms ${_WE_r^{pq}}$. Under the conditions imposed on
the two filtrations $W$ and $F$  in the definition of a Mixed Hodge Complex, the
three filtrations coincide, the spectral sequence with respect to
$W$ degenerates at rank $2$ and the induced filtration by $F$ on
the first terms $E^{p,q}_1$ define a Hodge Structure of weight $q$
  (lemma on two filtrations). The filtration
$d_1$ is a morphism of $HS$, hence the terms $E^{p,q}_2$ carry a
Hodge Structures of weight $q$. The proof consists to show that $d_r$ is
compatible with the induced Hodge Structure, but for $r > 1$ it is a morphism
between two Hodge Structures of different weight, hence it must
vanish.

\begin{prop}[MHS on the  cohomology of a MHC] Let $K$ be an  $A$-MHC.\\
i)The filtration $W[n]$ of $  H^n(K_A) \otimes \Q \simeq H^n(K_{A
\otimes \Q})$:
\[ (W[n])_q(H^n( K_{A \otimes \Q}) := \, Im \,
(H^n(W_{q-n} K_{A \otimes \Q}) \to H^n( K_{A \otimes \Q}))\]
 and the  filtration $F$ on  $H^n(K_{\C}) \simeq H^n(K_A) \otimes_A \C$:
\[ F^p(H^n( K_{\C}) := \, Im \,
(H^n(F^p K_{\C}) \to H^n( K_{\C}))\]
 define on $ H^n(K)$ an $A$-Mixed Hodge Structure, i.e.:\\
$$(H^n(K_A),(H^n(K_{A \otimes \Q}), W),(H^n(K_{\C}), W, F))$$ 
is an $A$-Mixed Hodge Structure.\\
ii) On the terms $_WE_r^{pq}(K_{\C}, W)$, the recurrent filtration
and the  two direct filtrations
  coincide $F_d = F_{rec} =
F_{d^*}$ and define the Hodge filtration $F$ of a Hodge Sructure of weight
$q$ and $d_r$ is compatible with $F$.\\
 iii) The morphisms $d_1:\,{_WE}_1^{p,q}\,
\rightarrow \,{_WE}_1^{p+1,q}$ are strictly compatible with $F$.\\
  iv) The   spectral sequence of $(K_{A \otimes \Q}$,
W) degenerates  at rank $2 \, (_WE_2 = {_WE}_{\infty})$.\\
 v) The   spectral sequence of Hodge Structures of $(K_{\C}, F)$ degenerates at
rank $1 \, (_FE_1 = {_FE}_{\infty})$.\\
 vi) The   spectral sequence  of the complex
$Gr_F^p(K_{\C})$,
 with the induced  filtration $W$, degenerates at rank $2$.
 \end{prop}

 Notice that the index of the weight is not given by the index of
 the induced filtration $W$ on cohomology, but shifted by $[n]$.
 One should recall that the weight of the Hodge Structure on the terms ${_WE}_r^{p,q}$
 is always $q$, hence the weight of $Gr^W_{-p}H^{p+q}(K)$ is $q$, that is
  $(W[p+q])_q = W_{-p}$.
  \vskip.1in

  Now, we are going to give the proofs of the statements of this introduction.

 \subsection{HS on the cohomology of a smooth compact algebraic variety}
 
 \

\n {\it Proof of the proposition (5.5)}. i) For $X$ projective, we
use the Dolbeault resolution of
 the sub-complex $F^p\Omega ^*_X$ by the  sub-complex $F^p \EE^*_X$
 defined as a simple complex associated to a  double complex:
 \begin{equation*}
F^p \EE^*_X = s((\EE_X^{r,*}, \overline{\partial}), \partial)_{
r\geq p }
\end{equation*}
This  is  a fine resolution (see \ref{dol}), which compute  de
Rham hypercohomology:  
$$\H^n (X, F^p \Omega ^*_X) \simeq H^n(X,
F^p\EE^*). $$ 
Then we can  identify the  terms $E^{pq}_1 =
 H^q(X, \Omega^p_X)$ with the  harmonic forms
 of type $(p,q)$ to deduce  the
 degeneration of the spectral sequence $({_FE}^{p,q}_r, d_r)$
 of the filtered de Rham complex at rank $1$ from classical Hodge theory,
  and moreover
 we obtain  the Hodge decomposition on $X$.

 \vskip 0.1 cm
  ii) If $X$ is not projective, by
 Chow's lemma (see \cite{Sha} p.69), there exists a projective variety and a projective
 birational morphism $f : X' \to X$. By Hironaka's
 desingularization (\cite{Hi}) we can suppose $X'$ smooth, hence $X'$
 is a K\"{a}hler manifold. Then, by Grothendieck's duality theory
 (\cite{DI} {\S} 4 or \cite{H1}), there
 exists for all integers $p$ a trace map $Tr(f): Rf_* \Omega^p_{X'} \to
 \Omega^p_{X}$  inducing a map on cohomology 
 $Tr(f): H^q(X', \Omega^p_{X'}) \to H^q(X,\Omega^p_{X})$, 
 because  $\H^q(X, Rf_* \Omega^p_{X'}) \simeq H^q (X',\Omega^p_{X'}) $,  
 such that the composition with
 the canonical reciprocal morphism 
 $f^*: H^q(X, \Omega^p_{X}) \to H^q(X',\Omega^p_{X'})$  is the identity:
 \[Tr(f) \circ f^* = Id:  H^q(X,
 \Omega^p_{X}) \to H^q(X',\Omega^p_{X'}) \to H^q(X,\Omega^p_{X})\]
  In particular we deduce that $f^*$ is injective. Since the map:
  $$f^*: f^*(\Omega^p_{X}, F) \to (\Omega^p_{X'}, F)$$ 
  is compatible
  with filtrations, we deduce a map of spectral sequences:
 \[ E_1^{pq} = H^q(X, \Omega^p_X) \xrightarrow{f^*}E_1'^{pq} =
H^q(X', \Omega^p_{X'}),\quad  f^*: E_r^{pq}(X) \hookrightarrow
E_r^{pq}(X')\]
 which is injective on all terms and commute with the differentials $d_r$.
  The proof is by induction on $r$ as follows.
 It is true for $r = 1$, as we have just noticed. The differential $d_1$ vanishes on $X'$;
 it must vanish on $X$, then the terms for $r=2 $ coincide with the terms
 for $r = 1$ and we can repeat the argument for all $r$.
 
 The degeneration of the
Hodge spectral sequence on $X$ at rank $1$ follows, and it  is
equivalent to the isomorphism:

\begin{equation*}
 \H^n (X, F^p \Omega ^*_X) \xrightarrow {\sim} F^p \H^n
(X, \Omega ^*_X). \end{equation*}

 Equivalently, the dimension of
the hypercohomology of the de Rham complex $\H^j(X, \Omega^*_X)$
is equal to the dimension of Hodge cohomology
$\oplus_{p+q=j}H^q(X,\Omega^p_X)$. 

  However, Hodge theory tells more. From the Hodge filtration, we  deduce
  the definition of the subspaces:
 \[H^{p,q}(X) := F^p H^i (X, \C ) \cap  \overline {F^q }H^i (X, \C ),\quad
 {\rm for} \, p+q = i\]
satisfying $H^{p,q}(X) = \overline {H^{q,p}(X)}$. We must check
the decomposition:
\begin{equation*}
H^i (X, \C ) = \oplus_{p+q=i} H^{p,q}(X),
 \end{equation*}
 and moreover, we will deduce  $H^{p,q}(X)\simeq H^q(X,\Omega^p_X)$.
We have:
$$F^p H^{n}(X) \cap \overline {F^{n-p+1} }H^{n}(X) \subset F^p
H^{n}(X') \cap \overline {F^{n-p+1}} H^{n}(X') = 0 .$$
This shows that $F^p H^{n}(X) + \overline {F^{n-p+1} }H^{n}(X)$ is a direct sum.
We want to prove that this sum equals $H^n(X)$.

Let $h^{p,q} = {\dim} H^{p,q}(X)$, since  the spectral sequence
degenerates at rank $1$, we have:
\[ {\dim} F^p H^{n}(X) = \sum_{i \geq p} h^{i,n-i}, \, {\dim}
\overline {F^{n-p+1} }H^{n}(X) = \sum_{i \geq n-p+1} h^{i,n-i},\]
  then:
\[ \sum_{i \geq p} h^{i,n-i}+\sum_{i \geq n-p+1} h^{i,n-i} \leq
{\dim}H^n(X) = \sum_i h^{i,n-i}\]
  from which we deduce $\sum_{i \geq p}
h^{i,n-i}\leq \sum_{i \leq n-p} h^{i,n-i}$.

By Serre duality on $X$ of dimension $N$, we have  $ h^{i,j}=
h^{N-i,N-j}$,  which transforms the inequality into:
 $\sum_{N-i \leq N-p} h^{N-i,N-n+i}\leq \sum_{N-i \geq N-n+p}
 h^{N-i,N-n+i}$,
from which we deduce on $ H^m(X)$ for $ m = 2N-n $, the opposite
inequality by setting $j = N-i, q = N-n+p$ shows that, for all $q$ and $ m$:
$$\sum_{j \geq q} h^{j,m-j}\geq \sum_{j \leq m-q} h^{j,m-j}, \quad
$$
In particular:
$$\sum_{i \geq
p} h^{i,n-i}\geq \sum_{i \leq n-p} h^{i,n-i} \quad \hbox{hence}
\quad \sum_{i \geq p} h^{i,n-i} = \sum_{i \leq n-p} h^{i,n-i}.$$
This implies ${\dim} F^p +  {\dim} \overline {F^{n-p+1} } = {\dim}
H^{n}(X)$. Hence:
 \[H^{n}(X) = F^pH^n(X) \oplus \overline {F^{n-p+1} }H^n(X)\]
  which,
  in particular, induces a decomposition: 
 $$F^{n-1}H^{n}(X) = F^p H^{n}(X)
\oplus H^{n-1,n-p+1}(X).$$

\begin{remark}
 If we use distinct notations for $X$ with Zariski topology and
 $X^{an}$ for the analytic associated manifold, then the filtration
$F$ is defined on the algebraic de Rham hypercohomology groups and
the comparison theorem (see \cite{Gr2}) is compatible with the
filtrations: $\H^i(X,F^p \Omega^*_X) \simeq \H^i(X^{an}, F^p
\Omega^*_{X^{an}})$.
\end{remark}

\subsubsection{} Let $\LL$ be a rational local system with a
 polarization, rationally defined, on the associated local system
$\LL_{\C} = \LL \otimes_{\Q}\C$, then the spectral sequence
defined by the Hodge filtration on de Rham complex with
coefficients $\Omega ^*_X\otimes_{\C}\LL$ degenerates at rank $1$:
\begin{equation*} E^{pq}_1 = H^q( X, \Omega ^p_X(\LL))\Rightarrow
H^{p+q}(X,\LL_{\C}), \quad E^{pq}_1 = E^{pq}_{\infty}
\end{equation*}
and the induced filtration by $F$ on cohomology defines a Hodge Structure. The
proof by Deligne proceeds as in Hodge theory.

\subsection{MHS on the cohomology of a Mixed Hodge Complex }
 The proof is based on a delicate study of the induced filtration
 $W$ and $F$ on the cohomology.
 To explain the difficulty, imagine for a moment that we want to give
 a proof by induction on the length of $W$. Suppose that the weights of a
 Mixed Hodge Complex: $(K,W,F)$ vary from
 $ W_0 = 0$ to $W_l = K$ and suppose we did construct the Mixed Hodge Structure on
 cohomology for $l-1$, then we consider the long exact sequence
 of cohomology:
\[H^{i-1}(Gr^W_l K )\to H^i(W_{l-1} K )\to H^i( K) \to H^i(Gr^W_l K )
\to H^{i+1}(W_{l-1} K )\] 
If the result was established then the
morphisms of the sequence are strict, hence the difficulty
 is a question of relative
 positions of the subspaces $W_p$ and $F^q$ on $H^i( K )$ with respect to
 Im$H^i(W_{l-1} K )$ and the projection on $H^i(Gr^W_l K )$. This
 study is known as the two filtrations lemma.

\subsubsection{Two filtrations} This section relates results on
various induced filtrations on terms of a spectral sequence,
contained in \cite{HII} and \cite{HIII} (lemma on two
filtrations). A filtration $W$ of a complex by subcomplexes define
a spectral sequence $E_r(K,W)$. The filtration $F$ induces in
various ways filtrations on $E_r(K,W)$, different
 in general. We discuss here axioms on $W$
and $F$, at the level of complexes, in order to get compatibility
of the various induced filtrations by  $F$ on the spectral
sequence of $(K,W)$. What we have in mind is to find axioms
leading to define the Mixed Hodge Structure with induced
filtrations $W$ and $F$ on cohomology. The proof is technical,
hence  we emphasize here  the main ideas as a guide to Deligne's
proof.

 \subsubsection{} Let $(K,F,W)$ be a   bi-filtered
complex  of objects of an abelian category, bounded below. The
filtration $F$,
 assumed to be biregular, induces on the terms $E_r^{pq}$ of
the spectral sequence $E(K,W)$ various filtrations as follows:

\begin{definition}[$F_d, F_{d^*}$] The first direct filtration
on $E_r(K,W)$ is the filtration $F_d$ defined  for $r$ finite or
$r = \infty$, by  the image:
$$ F_d^p(E_r(K,W)) = Im(E_r(F^pK,W) \rightarrow E_r(K,W)).$$
Dually, the second direct filtration $F_{d^*}$ on  $E_r(K,W)$ is
defined by the kernel:
$$F_{d^*}^p(E_r(K,W)) = Ker(E_r(K,W) \rightarrow E_r(K/F^pK,W)).$$
\end{definition}

 This  definition is temporarily convenient for the next computation,
 since the filtrations $F_d, F_{d^*}$ are naturally induced by $F$,
 hence compatible with the differentials $d_r$.

    Since for $r = 0,\, 1$, $B^{pq}_r \subset Z^{pq}_r$,

  \begin{lemma}
   We have $F_d = F_{d^*}$ on $E_0^{pq}= Gr_F^p(K^{p+q})$
   and $E_1^{pq}=H^{p+q}(Gr^W_p
    K)$.
  \end{lemma}

\begin{definition}[$F _{rec}$]  The recurrent filtration  $F_{rec}$
on $E_r^{pq}$ is defined as follows \\
i)  On $E_0^{pq} , F_{rec} = F_d =
F_{d^*}$. \\
ii) The recurrent filtration  $F_{rec}$ on $E_r^{pq}$  induces a
filtration on ker $d_r$, which induces on its turn the recurrent
filtration  $F_{rec}$ on $E_{r+1}^{pq}$ as a quotient of ker
$d_r$.
\end{definition}

\subsubsection{} The precedent definitions of direct filtrations
apply to $E_{\infty}^{pq}$ as well and they are compatible with
the isomorphism $E_r^{pq} \simeq E_{\infty}^{pq}$ for large $r$.
We deduce, via this  isomorphism a recurrent filtration $F_{rec}$
on $E_{\infty}^{pq}$. The filtrations $F$ and $W$ induce each a
filtration on $H^{p+q}(K)$. We want to prove that the isomorphism
$E_{\infty}^{pq} \simeq Gr_W^{-p}H^{p+q}(K)$ is compatible with
 $F$  at right and $F_{rec}$ on $E_{\infty}^{pq}$.

\subsubsection{Comparison of $F_d, F_{rec}, F_{d^*}$}
 In general we have only the inclusions
\begin{prop} i) On $E_r^{pq}$, we have the inclusions
 \[F_d(E_r^{pq}) \subset F_{rec}(E_r^{pq}) \subset
 F_{d^*}(E_r^{pq})\]
 ii) On $E_{\infty}^{pq}$, the filtration induced by the
filtration $F$ on $H^*(K)$ satisfy
 \[F_d(E_\infty^{pq}) \subset
F(E_\infty^{pq}) \subset F_{d^*}(E_\infty^{pq}). \]
  iii) The differential $d_r$ is compatible with $F_d$ and
  $F_{d^*}$.
 \end{prop}
 We want  to show that these three filtrations coincide under the
conditions on the Mixed Hodge Complex, for this we need to
 know the compatibility of $d_r$ with $F_{rec}$ which will be  the
 Hodge filtration of a Hodge Structure.
 In fact, Deligne proves an intermediary statement, that will apply
 inductively for Mixed Hodge Complexes.

\begin{theorem}[Deligne]\label{del} Let  $K$ be a  complex with
 two filtrations $W$ and $F$, where  $W$ is biregular and suppose
 that for each non negative integer $r < r_0$, the differentials
 $d_r$ of the graded complex $E_r(K,W)$ are strictly compatible
  with  $F_{rec}$, then\\
i) For $r \leq r_0$ the sequence
$$0 \rightarrow E_r(F^p K,W) \rightarrow E_r(K,W)
\rightarrow E_r (K/F^p K,W) \rightarrow 0$$ is exact, and for $r =
r_0 +1$, the sequence
$$ E_r(F^pK,W) \rightarrow E_r(K,W) \rightarrow
 E_r(K/F^pK,W)$$
is exact. In particular for $r \leq r_0 + 1$, the two direct and
the recurrent  filtration on $E_r(K,W)$ coincide $F_d = F_{rec} =
F_{d^*}$.\\
ii) For $ a < b$ and $r < r_0$, the differentials
 $d_r$ of the graded complex $E_r(F^aK/F^bK,W)$ are strictly compatible
  with  $F_{rec}$.\\
   iii) If the above condition is satisfied for all $r$, then
the spectral sequence $E(K,F)$ degenerates at $E_1$
  $(E_1 = E_{\infty})$ and the  filtrations $F_d, F_{rec}, F_{d^*}$
 coincide into a unique filtration called $F$. Moreover, the isomorphism
  $E_{\infty}^{pq} \simeq Gr_W^{-p}H^{p+q}(K)$ is compatible
    with the  filtrations $F$ and we
 have an isomorphism of spectral sequences
 \[Gr^p_F E_r(K,W) \simeq E_r(Gr^p_F K,W)\]
\end{theorem}

\begin{proof} This surprising  statement  looks natural only
if we have in mind the degeneration of $E(K,F)$ at rank $1$ and
the strictness in the category of Mixed Hodge Structures.

For fixed $p$, we consider the following property:
$$(P_r) \, E_i(F^pK,W)
\hbox{ injects into }E_i(K,W)\hbox{ for }i \leq r\hbox{ and its
image is }F^p_{rec}\hbox{ for }i \leq r+1.$$

We already noted that $(P_0)$ is satisfied. The proof by induction
on $r$ will apply as long as $r$ remains $ \leq r_0$. Suppose $r <
r_0$ and $(P_s)$ true for all
 $s \leq r$, we prove $(P_{r+1})$.
 The sequence:
$$E_r(F^pK,W)\xrightarrow{d_r}E_r(F^pK,W)\xrightarrow{d_r}E_r(F^pK,W)$$
injects into:
$$E_r(K,W)\xrightarrow{d_r}E_r(K,W)\xrightarrow{d_r}E_r(K,W)$$
with image $F_d = F_{rec}$, then, the image of $F^p_{rec}$ in
$E_{r+1}$:
$$F^p_{rec}E_{r+1} = Im[Ker
(F_{rec}^pE_r(K,W)\xrightarrow{d_r}E_r(K,W)) \to E_{r+1}(K,W)]$$
coincides with the image of $F^p_{d}$ which is by definition $
Im[E_{r+1}(F^p K,W) \to E_{r+1}(K,W)]$.

Since $d_r$ is strictly compatible with $F_{rec}$, we have:
$$d_r E_r(K,W)
\cap E_r(F^pK,W)= d_r E_r(F^pK,W)$$ which means that $E_{r+1}(F^p
K,W)$ injects into $E_{r+1}(K,W)$, hence we deduce the injectivity
for $r+1$.
 Since ker $d_r$ on $F^p_{rec}$ is equal to ker $d_r$ on $E_{r+1}(F^p
 K,W)$, we deduce $F^p_{rec} = F^p_d$ on $E_{r+2}(K,W)$, which proves
 $(P_{r+1})$.

 Then, i) follows from a dual statement applied to
 $F_{d^*}$ and ii) follows, because $E_r(F^bK,W) \hookrightarrow E_r(F^aK,W)$
   is injective as they are both in $E_r(K,W)$.

iii) We deduce from i) the first exact sequence followed by its
dual:
\begin{equation*}
\begin{split}
& 0\rightarrow E_r(F^{p+1}K,W) \rightarrow E_r(F^pK,W) \rightarrow
 E_r(Gr_F^pK,W)\rightarrow 0\\
& 0\leftarrow E_r(K/F^pK,W) \leftarrow E_r(F^{p+1}K,W)\leftarrow
E_r(Gr_F^pK,W) \leftarrow 0.
\end{split}
\end{equation*}
In view of the injections in i) and the coincidence of $F_d =
F_{rec} = F_{d^*}$ we have a unique filtration $F$, the quotient
of the first two terms in the first exact sequence above is
isomorphic to $Gr_F^pE_r(K,W)$, hence we deduce an isomorphism
\[Gr_F^pE_r(K,W) \simeq E_r(Gr_F^pK,W)\]
compatible with $d_r$ and autodual. If the hypothesis is now true
for all $r$, we deduce an exact sequence:
 $$0 \rightarrow
E_\infty(F^pK,W) \rightarrow E_\infty(K,W) \rightarrow
 E_\infty(K/F^pK,W)\rightarrow 0$$
 which is identical to:
$$0 \rightarrow
Gr_W H^*(F^pK) \rightarrow Gr_W H^*(K) \rightarrow
 Gr_W H^*(K/F^pK)\rightarrow 0$$
 from which we deduce, for all $i$:
 $$0 \rightarrow
H^i(F^pK) \rightarrow  H^i(K) \rightarrow
  H^i(K/F^pK)\rightarrow 0$$
and that the  filtrations $W$  induced  on $H^i(F^pK)$ from
$(F^pK, W)$ and from $(H^i(K), W)$ coincide.
\end{proof}

\subsubsection{Proof of the existence of a MHS on cohomology of a MHC}
The above theorem \ref{del} applies to Mixed Hodge Complexes, since the
hypothesis of the induction on $r$ in the theorem will be
satisfied as follows. If we assume that the  filtrations $F_d =
F_{rec} = F_{d^*}$ coincide for $r < r_0$ and moreover define the
same Hodge filtration $F$ of a Hodge Structure of weight $q$ on
$E^{p,q}_r(K,W)$ and $d_r: E^{p,q}_r \to E^{p+r,q-r+1}_r$ is
compatible with such Hodge Structure, then in particular $d_r$ is
strictly compatible with $F$, hence the induction apply.

\begin{lemma}
  For $r \geq 1$, the differentials
$d_r$ of the spectral sequence  $_WE_r$ are strictly compatible
to the recurrent filtration  $F = F_{rec}$. For $r \geq 2$, they
vanish.
\end{lemma}

The initial statement applies for $r = 1$ by definition of a Mixed
Hodge Complex  since the  complex $Gr_{-p}^W K$ is a Hodge Complex
 of weight  $-p$. Hence,
the two direct filtrations and the recurrent filtration  $F_{rec}$
coincide with the Hodge filtration $F$ on $_WE_1^{pq} = H^{p+q}(Gr_{-p}^W K)$.
The differential $d_1$ is compatible with the direct filtrations,
 hence with  $F_{rec}$, and  commutes
 with  complex conjugation since it is defined on
 $A \otimes \Q$, hence it is  compatible with
 $\overline F_{rec}$. Then it is  strictly compatible
 with the Hodge  filtration  $F = F_{rec}$.

 The filtration $F_{rec}$  defined in this way is $q$-opposed
 to its  complex conjugate and  defines a Hodge Structure of weight $q$ on
 $_WE_2^{pq}$.

  We suppose by induction that the two direct filtrations and the
recurrent filtration coincide on $_WE_s(s \leq r+1) : F_d =
F_{rec} = F_{d^*}$ and  $_WE_r = \, {_WE_2}$.  On $_WE_2^{p,q} =
{_WE}_r^{p,q}$, the filtration $F_{rec} :=F$ is compatible with
$d_r$ and $q$-opposed to its complex conjugate. Hence the morphism
 $ d_r : \, {_WE}_r^{p,q} \rightarrow \, {_WE}_r^{p+r,q-r+1}$
  is a morphism of a Hodge Structure of weight $q$ to a Hodge Structure
   of weight $q-r+1$
  and must vanish  for $r>1 $.
In particular, we deduce that the weight spectral sequence
degenerates at rank $2$.

The filtration on $_WE_{\infty}^{p,q}$ induced by the filtration
$F$ on $H^{p+q}(K)$ coincides with the  filtration $F_{rec}$ on
$_WE_2^{p,q}$.
 
\section{Logarithmic complex, normal crossing divisor and the mixed
cone} 

In the next two sections, we show how the previous abstract
algebraic study  applies to algebraic varieties by constructing
explicitly the Cohomological Mixed Hodge Complex. First, on a smooth complex
variety $X$ containing a Normal Crossing Divisor (NCD) $Y$ with
smooth irreducible components, we introduce the  complex of
sheaves of differential forms  with logarithmic singularities
along $Y$ denoted $\Omega^*_X(Log Y)$ and sometimes
 $\Omega^*_X<Y>$. Its hypercohomology is isomorphic to the cohomology of $X-Y$
with coefficients in $\C$ and it is naturally endowed with two
filtrations $W$ and $F$. It is only when $X$ is compact that the
bi-filtered complex  $(\Omega^*_X(Log Y), W, F)$ underlies a
Cohomological Mixed Hodge complex which defines a Mixed Hodge Structure
on the cohomology of $X-Y$.

 The Mixed Hodge Structure of a smooth variety $V$ depends on the properties at
infinity of $V$, i.e., we have to consider a
compactification of $V$ by a compact algebraic variety $X$, which
is always possible by a result of Nagata \cite{NaG}. Moreover, by
Hironaka's desingularization theorem \cite{Hi}, we can suppose $X$ is
smooth and the complement $Y = X - V $ is a Normal Crossing Divisor
(NCD) with smooth irreducible components. Hence we can use
$(\Omega^*_X(Log Y), W, F)$ to define a Mixed Hodge Structure 
on the cohomology of $X-Y$ and then carry 
this Mixed Hodge Structure onto the cohomology of $V$. It is not
difficult to show that such a Mixed Hodge Structure does not depend on the
compactification $X$ and will be referred to as the 
Mixed Hodge Structure on $V$. We
stress that the weight $W$ of the cohomology $H^i (V)$ is $ \geq
i$, to be precise $W_{i-1} = 0, W_{2i} = H^i (V)$.

 The dual of the logarithmic complex \cite{E} is more natural to construct
 and is an example of the general construction in the next section.
 Similar to the basic Mayer-Vietoris construction in cohomology,
in the case of the normal crossing divisor $Y$, we  define a
resolution of the complex $\Z_Y $ by a complex
  $\Z_{Y_{{\underline *}}}$ defined in  degree $j-1$ by the constant sheaf $\Z$ on the
disjoint sum $Y_{\underline j}$ of the intersections of $j$
distinct components of $Y$. The interest in such resolution is its
construction from the spaces $Y_{\underline j}$ forming what we
call a strict simplicial resolution of $Y$. In such case we will
develop a natural procedure to deduce a canonical Mixed 
Hodge Complex on $Y$. This construction is relatively 
easy to understand in this case, hence it is the best illustration
of the construction in the following section. The weight $W$ of the
cohomology $H^j (Y)$ is $ \leq j$, to be precise $W_{-1} = 0, W_{j} =
H^j(V)$. 

We associate to the natural morphism   $\Z_X \to \Z_{Y_{\underline*}}$  
a complex called the cone defining  the relative cohomology
isomorphic to the cohomology with compact supports of the
complement $X-Y$, Poincar\'e dual to the cohomology of $X-Y$.

  The combination of the two cases, when we consider  a
sub-normal Crossing Divisor $Z$, a union of some components of $Y$,
and consider the  complement $Y - Z$, we obtain an open 
Normal Crossing Divisor which
will be a model for the general case with the  most general type
of Mixed Hodge Structure on the  cohomology 
$H^j(Y-Z)$ of weights varying from $0$
to $2j$.

 Finally we discuss the technique of the mixed cone which
associates  a new Mixed Hodge Complex to a morphism
of Mixed Hodge Complexes at the level of the
homotopy category, which leads to a Mixed Hodge Structure on the relative cohomology
and a long exact sequence of Mixed Hodge Structures. However there exists an
ambiguity on the Mixed Hodge Structure obtained since it depends on an homotopy
between various resolutions \cite{E}. In this case the ambiguity is
related to the embedding of the rational cohomology into the
complex cohomology and equivalently related to the problem of
extension of Mixed Hodge Structure in general.

\subsection{MHS on the cohomology of smooth varieties}
  Now we construct the Mixed Hodge Structure on the 
  cohomology of a smooth complex
  algebraic variety $V$.
 The algebraicity of $V$ allows us to use  a  result of  Nagata \cite{NaG}
 to embed  $V$ as an open Zariski subset of a
 complete variety $ Z$. Then the singularities of $Z$ are included
 in $D:= Z-V$. Since Hironaka's desingularization process
 in characteristic zero (see \cite{Hi}) is carried out by blowing up smooth centers above
  $D$,
there exists a variety $X \to Z$ above $Z$ such that the inverse image
of $D$ is a normal crossing divisor  $Y$  with
 smooth components in $X$ such that $X-Y \simeq Z-D$.
 
 Hence we may start with the hypothesis that $V = X^*:= X-Y$
 is the complement of a normal crossing divisor $Y$  in  a
 smooth proper variety $X$,
 so  the construction  of the Mixed Hodge Structure is  reduced to this situation,
under the condition that we prove its independence  of the choice of $X$. 

 We introduce now
the logarithmic complex and prove that it underlies a 
Cohomological Mixed Hodge Complex
on $X$ and computes the cohomology of $V$.

\subsubsection{Logarithmic complex}
Let $X$ be a complex manifold and $Y$ be 
a Normal Crossing Divisor in $X$. We denote
by $ j: X^* \to X$  the embedding of $X^*:= X-Y$ into $X$.
 A meromorphic form $\omega$ has a pole of order at most $1$ along $Y$
  if at each point $y$, $f \omega$ is holomorphic for some
  local equation  $f$ of $Y$ at $y$. Let $\Omega^*_{X}(*Y)$ denote the
  sub-complex of $j_*\Omega_{X^*}^*$ defined by
   meromorphic forms along  $Y$, holomorphic on $X^*$.
   
  \begin{definition}
 The logarithmic de Rham complex  of $X$
 along a Normal Crossing Divisor $Y$ is the subcomplex
 $\Omega_X^*(Log \, Y)$ of the complex
 $\Omega^*_{X}(*Y)$ defined by the sections $\omega$  such that
   $\omega$ and $d\omega$ both have a pole of order
 at most $1$ along  $Y$.
   \end{definition}
   
    By definition, at each point $y \in Y$, there exist local coordinates
$(z_i)_{i \in [1,n]}$ on $X$ and $ I \subset [1,n]$ such that $Y$ is
defined at $y$ by the equation $\Pi_ {i \in I }  z_i = 0$. Then
$\omega$ has logarithmic poles along  $Y$ if and only if it can be
written locally as:
  \begin{equation*}
\omega = \sum_{i_1,\cdots,i_r} \omega_{i_1,\cdots,i_r}
  \frac{dz_{i_1}}{z_{i_1}} \wedge \cdots  \wedge
  \frac{dz_{i_r}}{z_{i_r}} \quad \hbox { where } \, \; \omega_{i_1,\cdots,i_r} \;
   \hbox{ is holomorphic.}
   \end{equation*}
 This formula may be used as a definition but then we have to prove
 its independence of the choice of coordinates, that is $\omega$ may be
  written in this form with respect to any set of local
  coordinates at $y$.
  
  The ${\OO}_X$-module $\Omega_X^1(Log \,Y)$ is locally free with
 basis  $(dz_i/z_i)_{i \in I}$ and
 $(dz_j)_{ j \in [1,n]-I}$ and
 $\Omega_X^p(Log \,Y) \simeq \wedge^p \Omega_X^1(Log \,Y)$.
 
Let $f: X_1 \rightarrow X_2$ be a morphism of complex manifolds,
with normal crossing divisors  $Y_i$ in $ X_i$ for $i=1,2$, such that
$f^{-1}(Y_2) = Y_1$. Then, the reciprocal morphism $f^*:
f^*(j_{2*} \Omega_{X_2^*}^*) \rightarrow j_{1*}
 \Omega_{X_1^*}^*$ induces a morphism on logarithmic
 complexes:
 $$ f^* : f^* \Omega_{X_2}^* (Log\, Y_2) \rightarrow
\Omega_{X_1}^* (Log \,Y_1).$$

\subsubsection{Weight filtration $W$} Let $Y =  \cup_{ i \in I}
Y_i$ be the union of smooth irreducible divisors. Let $S^q$
denotes the set of  strictly increasing sequences $\sigma =
(\sigma_1,...,\sigma_q)$ in the set of indices $I, \, Y_\sigma =
Y_{\sigma_1...\sigma_q } = Y_{\sigma_1} \cap ... \cap
Y_{\sigma_q},\, Y^q = \coprod_{\sigma \in S^q} Y_{\sigma}$ the
disjoint union of $Y_{\sigma}$. Set $Y^0 = X$ and $\Pi : Y^q
\rightarrow Y$ the canonical projection. An increasing filtration
$W$, called the weight, is defined as follows:
 \begin{equation*}
 W_m(\Omega_X^p(Log \,Y)) = \sum_{\sigma \in S^m}\Omega_X^{p-m}
\wedge dz_{\sigma_1}/z_{\sigma_1} \wedge... \wedge
dz_{\sigma_m}/z_{\sigma_m}
\end{equation*}
The sub-${\OO}_X$-module $W_m(\Omega_X^p(Log \, Y)) \subset
\Omega_X^p(Log \, Y)$ is the smallest sub-module stable by exterior
multiplication with local  sections of $\Omega_X^*$ and containing
the products $dz_{i_1}/z_{i_1} \wedge...\wedge dz_{i_k}/z_{i_k}$
for $k \leq m$ for local equations $z_j$ of the components of $Y$.

\subsubsection{The Residue isomorphism} To define the Poincar\'e residue
along a component of an intersection of $Y_i$, we need to fix and
order a set of hypersurfaces $Y_{i_1}, \ldots,Y_{i_m}$ to
intersect. The following composed Poincar\'e residue is defined on
$Gr^W_m \Omega_X^* (Log \, Y)$:
\begin{equation*} Res : Gr^W_m (\Omega_X^p (Log \, Y))
\rightarrow \Pi_* \Omega_{Y^m}^{p-m} 
\end{equation*}
given by $Res ([\alpha \wedge (dz_{i_1}/ z_{i_1} \wedge ...
\wedge (dz_{i_m}/z_{i_m})]) = \alpha /Y_{i_1, \ldots,i_m}$. 

In the  case $m = 1$, it defines:
  \begin{equation*}
    Res : \Omega_X^* (Log \, Y) \rightarrow \Pi_*
    \Omega^*_{Y^1}[-1]
  \end{equation*}
  where $Y^1$ is the disjoint union of the irreducible components 
  of $Y$.
 In general it is a composition map of such residue
    in the one codimensional case up to a sign.
     We need to prove
    it is well defined, independent of the
  coordinates, compatible with the differentials and
 induces an isomorphism:
\begin{equation*} Res : Gr_m^W (\Omega_X^* (Log \,Y))
\xrightarrow{\sim}\Pi_* \Omega_{Y^m}^* [-m]
\end{equation*}
   We construct its inverse. Consider, for each sequence of indices
  $\sigma = ( i_1, \ldots, i_m)$
in $S^m$, the morphism  ${\rho}_{\sigma}: \Omega_X^p \rightarrow
Gr_m^W(\Omega_X^{p+m} (Log \, Y))$,
 defined locally as:
\begin{equation*}  \rho_{\sigma}(\alpha) =  \alpha \wedge dz_{\sigma_1} /
z_{\sigma_{1}} \wedge  ... \wedge  dz_{i_m} / z_{i_{m}}
\end{equation*}
It  does not depend on the choice of  $z_i$, since for another
choice of coordinates $z'_i$,  $z_i/z'_i$ are holomorphic and the
difference $(dz_i/z_i)-(dz'_i/z'_i)
= d(z_i/z'_i)/(z_i/z'_i)$ is holomorphic. 

Then $\rho_{\sigma} (\alpha) - \alpha \wedge dz'_{i_1} / z'_{i_1}
\wedge  ...\wedge  dz'_{i_m} / z'_{i_m} \in
W_{m-1}\Omega_X^{p+m}(Log Y)$, and successively $\rho_\sigma
(\alpha) - \rho'_\sigma (\alpha) \in W_{m-1}\Omega_X^{p+m}(Log \,
Y)$. We have $\rho_\sigma (z_{i_j} \cdot \beta) = 0$ and
$\rho_\sigma(dz_{i_j} \wedge  \beta') =0$
 for sections $\beta$ of $\Omega_X^p$ and $\beta'$  of
$\Omega_X^{p-1}$; hence $\rho_\sigma$  factors by
 $\overline \rho_\sigma $ on $ \Pi_* \Omega _{Y_\sigma}^p $
  defined locally and  glue globally into
a morphism of complexes on $X$:
\begin{equation*} \overline \rho_\sigma : \Pi_* \Omega _{Y_\sigma}^p \rightarrow
 Gr_m^W (\Omega_X^{p+m}(Log \, Y)), \quad 
 \overline \rho : \Pi_* \Omega_{Y^m}^*[-m] \rightarrow Gr_m^W
\Omega_X^*(Log \, Y).
\end{equation*}

 \begin{lemma} We have the following isomorphisms of sheaves:
 \item  i)  $ H^i(Gr_m^W
\Omega_X^* (Log \, Y)) \simeq \Pi_* \hbox{\bf C}_{Y^m}$ for $i = m$
and $0$ for $i\neq m$,
 \item ii) $  H^i (W_r \Omega_X^* (Log \, Y)) \simeq \Pi_* \hbox{\bf C}_{Y^i}$
for $i \leq r$ and $  H^i (W_r \Omega_X^* (Log \, Y))= 0 $ for $i
> r $, \\
and in particular  $ H^i (\Omega_X^*(Log \, Y)) \simeq \Pi_*
{\C}_{Y^i}$.
  \end{lemma}

 \begin{proof} The statement in i) follows from the residue
 isomorphism.
 
  The  statement in ii) follows easily by induction
 on  $r$, from i) and the long exact sequence associated to the short exact
 sequence $0 \rightarrow W_r \rightarrow W_{r+1} \rightarrow
Gr_{r+1}^W \rightarrow 0$, written as
 $$H^i(W_r) \rightarrow H^i( W_{r+1}) \rightarrow
H^i (Gr_{r+1}^W) \rightarrow H^{i+1}(W_r)$$
  \end{proof}

   \begin{prop}[Weight  filtration $W$] The  morphisms of filtered  complexes:
\begin{equation*}
(\Omega_X^*(Log \, Y),W) \buildrel \alpha \over \leftarrow
(\Omega_X^*(Log \, Y),\tau) \buildrel \beta \over \rightarrow (j_*
\Omega_{X^*}^*,\tau)
 \end{equation*}
 are filtered quasi-isomorphisms.
 \end{prop}

 \begin{proof} The quasi-isomorphism $\alpha$ follows from the
 lemma.  
 
 The morphism $j$ is Stein, since for each polydisc $U(y)$ in
 $X$ centered at a point $y \in Y$, the inverse image $X^* \cap
 U(y)$ is Stein as the complement of an hypersurface, hence $j$
 is acyclic for coherent sheaves, that is $Rj_*\Omega_{X^*}^*
 \simeq j_* \Omega_{X^*}^*$. By Poincar\'e lemma $\underline{\C}_{X^*} \simeq
 \Omega_{X^*}^*$, so that $Rj_*\underline{\C}_{X^*} \simeq j_* \Omega_{X^*}^*$,
 hence it is enough to prove $Gr^{\tau}_i Rj_*{\C}_{X^*} \simeq \Pi_*
 {\C}_{Y^i}$, which is a local statement. 
 
 For each polydisc $U(y) =
 U$ with $U^* = U - U \cap Y = (D^*)^m \times D^{n-m}$,
the cohomology $\H^i(U,  Rj_*\underline{\C}_{X^*}) = H^i(U^*,  \C)$ can be
computed by
 K\"unneth formula and is equal to $\wedge^i H^1(U^*,  \C)\simeq
 \wedge^i \Gamma(U, \Omega_U^1(Log \, Y)))$ where $dz_i/z_i, i\in
 [1,m]$ form a basis dual to the homology basis since $(D^*)^m \times D^{n-m}$
 is homotopic to an $i-$dimensional torus $(S^1)^m$.
\end{proof}

\begin{corollary}
The weight filtration is rationally defined.
 \end{corollary}
 
The main point here is that the $\tau$ filtration is defined with
rational coefficients as $(Rj_* \underline{\Q}_{X^*},\tau) \otimes \C$, which
gives the rational definition for $W$.

 \subsubsection{Hodge filtration $F$}  It is defined
  by the  formula
  $F^p =  \Omega_X^{* \geq p}(Log \, Y)$, which
 includes all forms of type $(p',q')$ with  $p' \geq p$. We have:
 \begin{equation*}
 Res : F^p(Gr_m^W \Omega_X^* (Log \, Y)) \simeq \Pi_* F^{p-m}
\Omega_{Y^m}^* [-m]
\end{equation*}
hence a filtered isomorphism:
\begin{equation*} Res : (Gr_m^W \Omega_X^*(Log \, Y),F) \simeq
(\Pi_* \Omega_{Y^m}^*[-m],F [-m]).
\end{equation*}
\begin{corollary}
  The system $\hbox{\bf K}$:
\begin{enumerate}
\item $(\hbox{\bf K}^{\Q},W) =
   (Rj_* \underline{\Q}_{X^*},\tau) \in Ob D^+F(X,\Q)$
\item $(\hbox{\bf K}^{\C},W,F) =
    (\Omega_{X}^*(Log \, Y),W,F) \in ObD^+F_2(X,\C)$
\item The isomorphism $(\hbox{\bf K}^{\Q},W) \otimes \C
    \simeq (\hbox{\bf K}^{\C},W) \, \, {\rm in}\,\, D^+F(X,\C )$
    \end{enumerate}
 is a Cohomological Mixed Hodge Complex on $X$.
 \end{corollary}
 
\begin{theorem}[Deligne]
 The system $K = R \Gamma(X,\hbox{\bf K})$ is a
Mixed Hodge Complex. It endows the cohomology 
of $X^* = X-Y$ with a canonical Mixed Hodge Structure.
\end{theorem}

\begin{proof} The  result follows directly from the general theory of
Cohomological Mixed Hodge Complex.

Nevertheless, it is interesting to understand what is needed for a
direct proof and compute the  weight spectral sequence at rank
$1$:
 \begin{equation*}
 \begin{split}
& _WE_1^{pq} (R \Gamma (X, \Omega_{X}^*(Log \, Y)) =
\H^{p+q}(X,Gr_{-p}^W \Omega_{X}^*(Log Y))
 \simeq \H^{p+q}(X,\Pi_* \Omega_{Y^{-p}}^*[p])\\
 & \simeq
H^{2p+q}(Y^{-p},\C)  \Rightarrow Gr_q^W H^{p+q}(X^*,\C).
 \end{split}
\end{equation*}
where the double arrow means that the spectral sequence
degenerates to the cohomology graded with respect to the filtration $W$ induced
by the weight on the complex level.
 In fact, we recall the proof up to  rank $2$.
 The  differential $d_1$:
 \begin{equation*}
 d_1 = \sum_{j=1}^{-p} (-1)^{j+1}
G(\lambda_{j,-p}) = G: H^{2p+q}(Y^{-p},\C)  \longrightarrow
H^{2p+q+2}(Y^{-p-1},\C)
 \end{equation*}
 is equal to an alternate
Gysin  morphism,
 Poincar\'e  dual to the  alternate
restriction morphism:
 \begin{equation*} \rho = \sum_{j=1}^{-p}(-1)^{j+1} \lambda_{j,-p}^*
:H^{2n-q}(Y^{-p-1},\C) \rightarrow H^{2n-q}(Y^{-p},\C)
\end{equation*}
hence the first term:
\begin{equation*}
(_WE_1^{pq},d_1)_{p \in \Z} = (H^{2p+q}(Y^{-p},\C) ,d_1)_{p \in
\Z}
 \end{equation*}
 is viewed as a complex in the category of Hodge Structures of weight $q$. It follows that
the terms:
 $$_WE_2^{pq} =H^p(_WE_1^{*,q},d_1)$$ 
 are endowed with a
 Hodge Structure of weight $q$. 
 
  We need to prove that the  differential
  $d_2$   is compatible with the induced Hodge filtration.
   For this we introduced the
   direct filtrations compatible with $d_2$ and proved that they coincide
   with the induced Hodge filtration. The differential $d_2$ is necessarily
    zero since it is a morphism of Hodge Structure of different weights:
  the Hodge Structure of weight $q$ on $E_2^{pq}$    and
 the Hodge Structure of weight $q-1$ on $E_2^{p+2,q-1}$.
 The proof is the same as any Mixed Hodge Complex and  consists 
 of a recurrent argument to show in this way that the
   differentials $d_i$ for $i \geq 2$ are zero.
   \end{proof}

\subsubsection{Independence of the compactification and functoriality}
Let $U$ be a complex smooth variety, $X$ (resp. $X'$) a
compactification of $U$ by a Normal Crossing Divisor $Y$ (resp. $Y'$) at infinity, $j:
U \to X$ (resp. $j': U \to X'$) the open embedding; then $j\times
j': U \to X\times X'$ is a locally closed embedding, with closure
$V $. By desingularizing $V$ outside the image of $U$, we are
reduced to the case where we have a smooth variety $X''
\xrightarrow{f} X$ such that $Y'' := f^{-1}(Y)$ is a Normal Crossing Divisor and $ U \simeq
X'' - Y''$, then we have an induced morphism $f^*$ on the
corresponding logarithmic complexes, compatible with the structure
of Mixed Hodge Complex. It follows that the induced morphism $f^*$ on
hypercohomology is compatible with the Mixed Hodge Structure 
and is an isomorphism
on the hypercohomology
groups, hence it is an isomorphism of Mixed Hodge Structure.

\n {\it Functoriality.} Let $f: U \to V$ be a morphism of smooth
varieties, $X$ (resp. $Z$) smooth compactifications of $U$ (resp.
$V$) by  Normal Crossing Divisor  at infinity, then taking the closure of the graph of
$f$ in $X\times Z$ and  desingularizing, we are reduced to the
case where there exists a compactification $X$ with an extension
$\overline f : X \to Z$ inducing $f$ on $U$. The induced morphism
 $\overline f^*$ on   the corresponding logarithmic
complexes is compatible with the filtrations $W$ and $F$ and with
the structure of Mixed Hodge Complex, hence it is compatible with 
the Mixed Hodge Structure on hypercohomology.

\begin{prop} Let $U$ be a smooth complex algebraic variety.\\
i)  The  Hodge numbers $h^{p,q}:=$ dim.$H^{p,q}(Gr^W_{p+q}H^i(U))$
vanish for $p , q \notin [0,i]$. In particular, the weight of the
cohomology $H^i(U)$ vary from $i$ to $2i$.\\
ii) let $X$ be a smooth compactification of $U$, then:
$$W_iH^i(U) =\hbox{Im } (H^i(X) \to H^i(U)).$$
\end{prop}

\begin{proof} i) The space $Gr^W_{r}H^i(U)$ is
 isomorphic to the term
$E_2^{i-r,r}$ of the spectral sequence with a Hodge Structure of
weight $r$, hence it is a sub-quotient of   $E_1^{i-r,r}=
H^{2i-r}(Y^{r-i})$ twisted by $\Q(i-r)$. Hence, we have
$h^{p,q}(H^{2i-r}(Y^{r-i})) = 0$, for $p+q = r$, unless $r-i \geq 0,
2i-r \geq 0$ ( i.e $i \leq r \leq 2i$ ) and $p \in [0, 2r-i]$,
$h^{p,q}(E_1^{i-r,r}) = 0$ unless $ p \in [i-r, r]$.

\medskip
 \n ii) Suppose $U$ is the complement of a Normal Crossing
Divisor. Denote $j:U\to X$ the inclusion. By definition 
$W_iH^i(U) =$ Im ($ \H^i(X, \tau_{
\leq 0} Rj_* \Q_U) \to H^i(U,\Q)$), hence it is equal to the image
of $H^i(X,\Q)$ since $\Q $ is quasi-isomorphic to $\tau_{ \leq 0}
Rj_* \Q_U$. If $ X - U$ is not a Normal Crossing Divisor, there
exists a desingularization $\pi: X' \rightarrow X$ with an
embedding of $U$ in $X'$ as the complement of a Normal Crossing
Divisor, then we use the trace map $Tr\, \pi: H^i(X',\Q)\to
H^i(X,\Q) $ satisfying $(Tr \pi) \circ \pi^*= Id$ and compatible
with Hodge Structures. In fact the trace map is defined as a morphism of sheaves
$R \pi_* \Q_{X'} \to \Q_X$ \cite{V2}, hence commutes with the
restriction to $U$. In particular, the images of both cohomology
groups coincide in $H^i(U)$.
\end{proof}

   \begin{xca}[Riemann Surface]
Let $\overline{C}$ be a connected compact Riemann surface of genus
$g$, $Y = \{ x_1, \ldots, x_m\}$ a subset of $m$ points, and  $C =
\overline{C} - Y $  the open surface with $m$ points in
$\overline{C}$ deleted.
 Consider the long exact sequence:
$$0 \rightarrow H^1(\overline{C}, \Z) \rightarrow H^1({C}, \Z)
\rightarrow H^2_Y(\overline{C}, \Z)= \oplus_{i = 1}^{i=m} \Z
\rightarrow H^2(\overline{C}, \Z)\simeq \Z \rightarrow H^2(C, \Z)
= 0$$
 then:
 $$0 \rightarrow H^1(\overline{C}, \Z) \rightarrow H^1({C}, \Z)
\rightarrow \Z^{m-1} \simeq Ker (\oplus_{i = 1}^{i=m} \Z \to
\Z)\to 0$$ 
is a short exact sequence of Mixed Hodge Structures  where $H^1({C}, \Z)=
W_2H^1({C}, \Z)$ is an extension of two different weights:
$W_1H^1({C}, \Z)= H^1(\overline{C}, \Z)$ of rank $2g$ and $Gr^W_2
H^1({C}, \Z)
\simeq \Z^{m-1}$.

The Hodge filtration is given by $F^0H^1(C, \C) = H^1(C,\C)$, $F^2
H^1(C, \C) = 0$, while:
 \[ F^1 H^1(C, \C) \simeq
\H^1(\overline{C},( 0 \to \Omega^1_{\overline{C}}(Log \{x_1,
\ldots, x_m \})\simeq H^0 (\overline{C},
\Omega^1_{\overline{C}}(Log \{x_1, \ldots, x_m \})\] is of rank $g
+ m-1$  and fits into  the exact sequence determined by the
residue morphism:
$$0 \to \Omega^1_{\overline{C}} \to
\Omega^1_{\overline{C}}(Log \{x_1, \ldots, x_m \}) \to \OO_{\{x_1,
\ldots, x_m\}} \to 0.$$
\end{xca}

   \begin{xca}[Hypersurfaces] Let $i: Y \hookrightarrow P$ be a smooth
hypersurface in a projective variety $P$.\\
 1)  To describe the cohomology
of the affine open set $U = P - Y$ we may use, by Grothendieck's
result on algebraic de Rham cohomology, algebraic forms on $P$
  regular on $U$ denoted $\Omega^*(U) = \Omega_P(*Y)$, or
meromorphic forms along $Y$, holomorphic on $U$ denoted
$\Omega_{P^{an}}(*Y) $, where the Hodge filtration is described by
the order of the pole (\cite{HII} Prop. 3.1.11) (but not the
trivial filtration $F$ on $\Omega^*(U)$ and  by Deligne's resulton
the
  logarithmic complex  with its trivial filtration $F$.\\
2) For example, in the  case of a curve $Y$ in a plane $\P^2$,
global holomorphic forms are all rational by Serre's result on
cohomology of coherent sheaves. We have an exact sequence of
sheaves:
$$ 0 \rightarrow \Omega^2_P \rightarrow i_* \Omega^2_P (Log Y
)\rightarrow \Omega^1_Y \rightarrow 0. $$
  Since $h^{2,0} = h^{2,1} = 0$, $H^0(P,\Omega^2_P) =
   H^1(P,\Omega^2_P) = 0$, hence  we deduce from the associated long
    exact sequence, the isomorphism:
  $$Res: H^0(P,\Omega^2_P(Log Y)) \xrightarrow{\sim} H^0(Y,\Omega^1_Y),$$
  that is  $1$-forms on $Y$ are residues
 of rational $2-$forms on $P$ with simple pole along the curve.\\
  In homogeneous coordinates, let $F = 0$ be the homogeneous equation of
  $Y$. We take the residue along $Y$
  of the rational form:
  $$\frac{A(z_0dz_1\wedge dz_2 - z_1dz_0\wedge dz_2 + z_2dz_0\wedge
  dz_1)}{F}$$ where $A$ is homogeneous of degree $d-3$ if $F$ has
  degree $d$ (\cite {carl} example 3.2.8).\\
3) For $P$ projective again, we consider the exact sequence for
  relative  cohomology (or cohomology with support in $Y$):
  $$ H^{k-1}(U) \xrightarrow{\partial} H_Y^{k}(P)
  \rightarrow H^{k}(P)\xrightarrow{j^*} H^{k}(U)$$
which reduces via Thom's isomorphism, to:
$$ H^{k-1}(U) \xrightarrow{r} H^{k-2}(Y)
  \xrightarrow{i_*} H^{k}(P)\xrightarrow{j^*} H^{k}(U)$$
where $r$ is the topological Leray's residue map dual to the tube
over cycle map $\tau:H_{k-2}(Y)
  \rightarrow H_{k-1}(U)$ associating to a cycle $c$
  the boundary in $U$ of a tube over $c$,
   and $i_*$ is Gysin map, Poincar\'e dual to
the map $i^*$ in cohomology.

 For $P = \P^{n+1}$ and $n$ odd, the map $r$ is an isomorphism:
$$ H^{n-1}(Y) \simeq H^{n+1}(P) \to H^{n+1}(U) \xrightarrow{r} H^{n}(Y)
  \xrightarrow{i_*} H^{n+2}(P)= 0\xrightarrow{j^*} H^{n+2}(U)$$
  and for $n$ even the map $r$ is injective:
$$  H^{n+1}(P)=0 \to H^{n+1}(U) \xrightarrow{r} H^{n}(Y)
  \xrightarrow{i_*} H^{n+2}(P)= \Q \xrightarrow{j^*} H^{n+2}(U)$$
then $r$ is surjective onto the primitive cohomology:
$$ r:H^{n+1}(U) \xrightarrow{\im} H_{\rm prim}^{n}(X)$$
\end{xca}

\subsection{MHS of a normal crossing divisor (NCD)}
Let $Y $ be a Normal Crossing Divisor in a proper complex smooth algebraic variety.
 We suppose the  irreducible components
  $(Y_i)_{i \in I}$ of $Y$  smooth and ordered.
  
\subsubsection{ Mayer-Vietoris  resolution} Let $S_q$
denotes the set of strictly increasing sequences $\sigma =
(\sigma_0,...,\sigma_q)$ on the ordered set of indices $I$,
$Y_{\sigma} = Y_{\sigma_0} \cap ... \cap Y_{\sigma_q} ,
Y_{\underline q} =  \coprod_{\sigma \in S_q} Y_{\sigma}$ is the
disjoint union, and for all $j \in [0,q], q \geq 1$ let
$\lambda_{j,\underline q} : Y_{\underline q} \rightarrow
Y_{\underline q-1}$ denotes a
 map inducing for each $\sigma$ the embedding
$\lambda_{j,\sigma}:Y_{\sigma} \rightarrow Y_{\sigma(\widehat j)}$
where $\sigma (\widehat j)= (\sigma_0,...,\widehat
{\sigma_j},..., \sigma_q)$ is obtained by deleting  $\sigma_j$.
Let $\Pi_q : Y_{\underline q} \rightarrow Y$ (or simply $\Pi$)
denotes the canonical projection and $ \lambda^*_{j,\underline q}
: \Pi_* \hbox{\bf Z}_{Y_{\underline {q-1}}}
 \rightarrow \Pi_* \hbox{\bf Z}_{Y_{\underline q}}$ the
restriction map defined by $\lambda_{j,\underline q}$
  for $j \in [0,q]$.
  
 \begin{definition}[Mayer-Vietoris resolution of
 $\Z_Y$] It is defined by the following complex of sheaves
 $\Pi_* {\Z}_{Y{\underline .}}$:
\begin{equation*}
0 \rightarrow \Pi_* \Z_{Y_{\underline 0}} \rightarrow
  \Pi_*\Z_{Y_{\underline 1}} \rightarrow \cdots \rightarrow
  \Pi_* \Z_{Y_{\underline {q-1}}}
{\buildrel {\delta_{q-1}} \over \rightarrow} \Pi_*
\Z_{Y_{\underline q}}\rightarrow \cdots
\end{equation*}
 where $ \delta_{q-1} =  \sum_{j \in [0,q]} (-1)^j \lambda^*_{j,{\underline q}}$.
 \end{definition}
 
 This resolution is associated to  an hypercovering of
$Y$ by topological  spaces in the following sense. Consider the
diagram of spaces over $Y$:
\begin{equation*}
Y_{\underline .} =  ( Y_{\underline 0}\quad \begin{matrix} \leftarrow \\
\leftarrow
\end{matrix}\quad  Y_{\underline 1}\quad
\begin{matrix} \leftarrow \\ \leftarrow \\ \leftarrow \end{matrix}
\quad  \cdots \quad Y_{\underline {q-1}}\quad  \begin{matrix}
\xleftarrow {\lambda_{j,\underline q}}\\ \vdots \\ \longleftarrow
\end{matrix}\quad
 Y_{\underline q}
 \cdots  )\  \xrightarrow{\Pi} Y
\end{equation*}
 This diagram is  the strict simplicial scheme associated in \cite{HII} to  the
 normal crossing divisor  $Y$, called here after
Mayer-Vietoris. The Mayer-Vietoris  complex is canonically
associated  as direct image by $\Pi$ of the sheaf
${\Z}_{Y_{\underline *}}$ equal to ${\Z}_{Y_{\underline  i}}$ on
$Y_{\underline  i}$. The generalization of such resolution is the
basis of the later general construction of Mixed Hodge Structure using  simplicial
covering of an algebraic variety.

\subsubsection{ The cohomological mixed Hodge complex of a NCD} The weight
filtration  $W$ on $\Pi_* {\Q}_{Y \underline .}$  is defined by:
\begin{equation*}
W_{-q} (\Pi_*\Q_{Y_{\underline *}})
 = {\sigma}_{\cdot \geq q} \Pi_* \Q_{ Y_{\underline *}}
=  \Pi_* {\sigma}_{\cdot \geq q}\Q_{ Y_{\underline *}}, \quad
Gr_{-q}^W (\Pi_* \Q_{Y_{\underline *}}) \simeq \Pi _*
\Q_{Y_{\underline q}} [-q]
\end{equation*}
 We introduce the
complexes $\Omega ^*_{Y_{\underline i}}$ of differential  forms on
$Y_{\underline i}$. The simple  complex $s(\Omega^*_{Y_{\underline
*} })$ is associated to the double complex $\Pi_* \Omega
^*_{Y_{\underline *}}$ with the exterior differential $d$ of forms
and the differential $\delta_*$ defined by $\delta_{q-1} = \sum_{j
\in [0,q]} (-1)^j \lambda^*_{j,{\underline q}}$ on $\Pi_*
\Omega^*_{Y_{\underline {q-1}}}$. The weight  $W$, and   Hodge $F$
filtrations are defined as:
\begin{equation*}
  W_{-q} = s(\sigma_{\cdot \geq q}
\Omega^*_{Y_{\underline *}}) = s(0 \rightarrow \cdots 0
\rightarrow \Pi_* \Omega^*_{Y_{\underline q}} \rightarrow \Pi_*
\Omega^*_{Y_{\underline {q+1}}} \rightarrow \cdots )
\end{equation*}

\begin{equation*}
  F^p = s(\sigma_{*\geq p} \Omega_{Y_{\underline *}}^*) = s(0
\rightarrow \cdots  0 \rightarrow \Pi_* \Omega^p_{Y_{\underline
*}} \rightarrow \Pi_* \Omega_{Y_{\underline *}}^{p+1}
\rightarrow\cdots)
\end{equation*}
 We have a filtered isomorphism
in the filtered derived category of sheaves of abelian groups
$D^+F(Y,\C)$ on $Y$:
 \begin{equation*}
  (Gr_{-q}^W s(\Omega_{Y_{\underline *}}^*),F) \simeq (\Pi_*
\Omega_{Y_{\underline q}}^*[-q],F) \quad
 \hbox{in} \quad D^+F(Y,\C).
 \end{equation*}
inducing isomorphisms in $D^+(Y,\C)$:
\begin{equation*}
({\Pi}_* \Q_{Y_{\underline *}},W) \otimes \C = (\C_{Y\underline
.},W) \xrightarrow{ \alpha \, \sim } (s(\Omega_{Y_{\underline
*} }^*),W)
\end{equation*}
\begin{equation*}
 Gr_{-q}^W(\Pi_*\C_{Y_{\underline *}}) \simeq \Pi_*
\C_{Y_{\underline q}}[-q] {\buildrel \sim \over \rightarrow}
{\Pi}_* \Omega_{Y_{\underline q}}^* [-q] \simeq Gr_{-q}^W
s(\Omega_{Y_{\underline *}}^*)
 \end{equation*}
 Let  $\hbox{\bf K}$ be the system consisting of:
\begin{equation*}
 (\Pi_* \Q_{Y_{\underline *}} ,W), \Q_Y \xrightarrow{\sim}
\Pi_*\Q_{Y_{\underline *}}, (s(\Omega_{Y_{\underline *}
}^*),W,F),\quad (\Pi_* \Q_{Y_{\underline *}} ,W) \otimes \C
\xrightarrow{\sim} (s(\Omega_{Y_{\underline *}}^*),W)
 \end{equation*}
 
\begin{prop}
The system $\hbox{\bf K}$
 associated to  a normal crossing divisor $Y$ with smooth  proper irreducible
 components, is a Cohomological Mixed Hodge Complex on $Y$. 
 It defines a functorial Mixed Hodge Structure on the
cohomology $H^i(Y,\Q)$,  with  weights  varying between $0$ and
$i$.
 \end{prop}
 
In terms of  Dolbeault  resolutions : $(s(\EE^{*,*}_{Y\underline
.}), W, F )$, the statement means that
 the complex of global  sections  $\Gamma (
Y,s(\EE^{*,*}_{Y_{\underline *}}), W, F ):= (\R \Gamma (Y, \C), W,
F)$ is a Mixed Hodge Complex in the following sense:
 \begin{equation*}
 \begin{split}
 & (Gr^W_{-i}(\R \Gamma (Y,\C),F):=
(\Gamma ( Y, W_{-i} s({\EE}^{*,*}_{Y_{\underline *}})/\Gamma (
Y,W_{-i-1}s({\EE}^{*,*}_{Y_{\underline *}}), F)\\
& \simeq (\Gamma ( Y, Gr^W_{-i} s({\EE }^{*,*}_{Y_{\underline
*}}), F) \simeq (\R \Gamma (Y_i,\Omega_{Y_{\underline
i}}^*[-i]),F)
\end{split}
\end{equation*}
is a Hodge Complex of weight $-i$ in the sense that:
\begin{equation*}
 (H^n (Gr^W_{-i}\R \Gamma (Y,\C)),F) \simeq  (H^{n -i}
(Y_{\underline i}, \C), F)
\end{equation*}
is a Hodge Structure of weight $n-i$. \\
The terms  of the  spectral sequence  $E_1(K,W)$ of  $(K,W)$ are
written as:
 \begin{equation*}
_WE_1^{pq} = \H^{p+q}(Y ,Gr_{-p}^W(s \Omega_{Y_{\underline *}}^*))
\simeq
 \H^{p+q}(Y,\Pi_* \Omega^*_{Y_{\underline p}}[-p]) \simeq
 H^q(Y_{\underline p},\C)
  \end{equation*}
  They carry the Hodge Structure of the space $Y_{\underline p} $.
  The  differential is a combinatorial restriction map 
  inducing  a morphism of Hodge Structures:
$$d_1 = \sum_{j \leq p+1} (-1)^j
   \lambda_{j,p+1}^*
   : H^q(Y_{\underline p},\C) \rightarrow H^q(Y_{\underline
    {p+1}},\C).$$
   The spectral sequence
 degenerates at $E_2\,(E_2 =
 E_{\infty})$.
 
\begin{corollary} The Hodge Structure on $Gr^W_q H^{p+q}(Y,\C)$ is the cohomology
of the complex of Hodge Structure defined by $(H^q (Y_{\underline *},\C), d_1)$
equal to $H^q (Y_{\underline p},\C)$ in degree $p \geq 0$:
 \begin{equation*}
 (Gr^W_q H^{p+q}(Y,\C), F)\simeq ((H^p(H^q
(Y_{\underline *},\C), d_1), F).
 \end{equation*}
 In particular, the weight of $ H^i(Y,\C)$ vary in the interval
 [$0,i$] ($Gr^W_q  H^i(Y,\C) = 0$ for $q \notin $ [$0,i$]).
 \end{corollary}
 We will see that the last condition on the weight is true for all
complete varieties.

 \subsection{Relative cohomology and the mixed cone }
  The notion of  morphism of Mixed Hodge Complex  involves compatibility
   between the rational level and complex level in the derived category,
   hence up to quasi-isomorphism. 
   
   To define a Mixed Hodge Structure on
   the relative cohomology, we define
   the notion of mixed cone with respect to a representative of the morphism
    on the level of complexes, hence depending on the representative. 
    
    The
    isomorphism between two structures
  obtained for two representatives depends on the choice of an homotopy, hence
it is not naturally defined. Nevertheless this notion is
interesting in  applications.  Later, one solution is to consider
Mixed Hodge Complex on a category of diagrams, then the diagonal filtration is a
natural example of such construction applied for example on
simplicial varieties.

  \subsubsection{ } \label{6.3.1} A morphism $u : K \rightarrow K'$ of Mixed Hodge Complex (resp.
CMHC) consists of morphisms:
\begin{equation*}
\begin{split}&u_A : K_A \rightarrow K'_A \,\, {\rm in}\,\, D^+A ({\rm
resp.}) D^+(X,A)),\\& u_{A \otimes \Q} : (K_{A \otimes \Q},W)
\rightarrow ({K'}_{A \otimes \Q},W) \,\,{\rm in }\,\, D^+F(A
\otimes \Q)
(\,\,{\rm resp.}\,\, D^+F(X,A \otimes \Q)), \\
&u_{\C} : (K_{\C},W,F) \rightarrow (K'_{\C},W,F) \,\,{\rm in}\,\,
D^+F_2 \C (\,\,{\rm resp.}\,\, D^+F_2(X,\C).)
\end{split}
\end{equation*}
and  commutative  diagrams:
$$\begin{matrix} K_{A\otimes \Q}&\xrightarrow {u_{A\otimes \Q}}&
K'_{A\otimes \Q}\qquad &K_{A\otimes \Q}\otimes \C&\xrightarrow
{u_{A\otimes \Q}\otimes \C}&K'_{A\otimes \Q}\otimes \C\\
\wr\downarrow{\alpha}&&\wr\downarrow{\alpha'}\qquad
&\wr\downarrow{\beta}&& \wr\downarrow{\beta'}\\
 K_A\otimes\Q&\xrightarrow{u_A\otimes \Q}&K'_A\otimes \Q \qquad &
  K_{\C}&\xrightarrow{u_{\C}}& K'_{\C}
  \end{matrix}$$
  in $D^+(A \otimes \Q)$ (resp. $D^+(X, A \otimes \Q)$ and in $D^+F(\C)$
(resp. $D^+F(X,\C))$ with respect to $W$.

\subsubsection{} Let $(K,W)$ be a complex of objects of an abelian category
  $\A$ with an increasing filtration  $W$. We denote by
   $(T_M K, W)$ or $(K[1], W[1])$ the mixed shifted complex
with a translation on degrees
    of $K$ and  $W$: $W_n(T_MK^i) = W_{n-1}K^{i+1}$.
    
\begin{definition}[Mixed cone]
  Let $u : K \rightarrow K'$ be a
morphism  of complexes in  $C^+F (\A)$ with an increasing
filtration. The mixed cone $C_M(u)$ is defined by the complex
$T_MK \oplus K'$ with the differential of  the cone $C(u)$.
\end{definition}

\subsubsection{}
 Let $u: K \rightarrow K'$ be a morphism of a Mixed Hodge Complex. There exists a
 quasi-isomorphism
  $v= (v_A, v_{A \otimes\Q},v_{\C}):\tilde K \xrightarrow{\approx} K$
and  a morphism $\tilde u  = (\tilde u_A, \tilde u_{A \otimes \Q},
\tilde u_{\C}): \tilde K \rightarrow K'$ of Mixed Hodge Complexes such that $ v$ and
 $\tilde u$ are defined successively in  $C^+A$, $C^+F(A
\otimes \Q)$ and $C^+F_2 \C$, i.e., we can find, by definition,
diagrams:
$$K_A  \xleftarrow{ \approx } \tilde K_A \rightarrow {K'}_A  ,
\quad K_{A \otimes \Q} \xleftarrow{ \approx }  \tilde K_{A \otimes
\Q} \rightarrow {K'}_{A \otimes \Q}  , \quad K_{\C}\xleftarrow{
\approx }  \tilde K_{\C} \rightarrow {K'}_{\C},$$
  or in short
$K \xleftarrow {\approx v} \tilde K \xrightarrow{\tilde u}
 K'$ (or equivalently
$K   \xrightarrow {\tilde u} \tilde K' \xleftarrow {\approx v}
K'$) representing $u$.

 \subsubsection{Dependence on homotopy} Consider  a  morphism $u
\colon K \rightarrow K'$ of
  Mixed Hodge Complexes, represented  by a morphism of complexes $ {\tilde u} :
\, \tilde K \rightarrow K'$.\\
To define the mixed cone $C_M( {\tilde u})$ out of: \\
(i) the cones $C({\tilde u_A}) \in C^+(A),\,  C_M(\tilde u_{A
\otimes \Q}) \in C^+F(A \otimes \Q), \, C_M( \tilde u_{\C}) \in
C^+F_2(\C)$,  \\
 we still need to define compatibility
isomorphisms:
\[ \gamma_1:   C_M(\tilde u_{A \otimes \Q}) \simeq
 C(\tilde u_A) \otimes \Q, \quad    \gamma_2:
 (C_M(\tilde u_{\C},W) \simeq (
  C_M(\tilde u_{A \otimes \Q}),W) \otimes \C \]
\n  successively    in $D^+(A \otimes \Q)$ and  $D^+F(\C)$. With the notations of \ref{6.3.1}
the choice of isomorphisms $C_M(\tilde\alpha, \tilde
\alpha')$
 and  $C_M(\tilde\beta, \tilde\beta')$ from the compatibility isomorphisms
 in $K$ and $K'$ does not define compatibility isomorphisms for the cone
 since the diagrams of compatibility are commutative only up to
 homotopy, that is there exists  homotopies $h_1$ and $h_2$
 ~{such that}:
 $$\tilde \alpha' \circ (\tilde u_{A \otimes \Q}) - (\tilde u_A \otimes
 \Q) \circ \tilde \alpha = h_1 \circ d + d \circ h_1, \,\,$$
 and:
 $$\tilde \beta' \circ \tilde u_{\C} - (\tilde u_{A \otimes \Q} \otimes
 \C ) \circ \tilde \beta = h_2 \circ d + d \circ h_2.$$
 ii) Then we can define the compatibility isomorphism as:
 \[C_M(\tilde\alpha, \tilde\alpha', h_1):=  \left({\tilde \alpha
\atop {h_1}}{0 \atop  \tilde \alpha'}\right):  C_M(\tilde u_{A
\otimes \Q}) \xrightarrow{\sim}
 C(\tilde u_A) \otimes \Q \]
 and a similar formula for $C_M(\tilde\beta, \tilde\beta', h_2)$.
 
\begin{definition} Let $u \colon K \rightarrow K'$ be a  morphism
of Mixed Hodge Complexes. The mixed cone $C_M(\tilde u, h_1,h_2)$ constructed above
depends on the choices of the homotopies  ($ h_1,h_2$) and  the
choice of a representative $\tilde u$ of $u$, such that:
\begin{equation*} Gr_n^W(C_M(\tilde u),F)
\simeq (Gr_{n-1}^W(T\tilde K),F) \oplus (Gr_n^WK',F)
\end{equation*}
 is a HC  of weight
$n$; hence  $C_M (\tilde u, h_1, h_2)$ is a Mixed Hodge Complex.
 \end{definition}
 
\section{MHS on the cohomology of a complex  algebraic
variety}

The aim of this section is to prove:

\begin{theorem}[Deligne]
The cohomology of a complex algebraic variety carries a natural
Mixed Hodge Structure.
 \end{theorem}

 The uniqueness and the functoriality  will follow easily, once we have fixed
 the case of Normal Crossing Divisors and the smooth case.

We need a construction based on a diagram of algebraic varieties:
\begin{equation*}
X_* =  ( X_0\quad \begin{matrix} \leftarrow \\
\leftarrow
\end{matrix}\quad  X_1\quad
\begin{matrix} \leftarrow \\ \leftarrow \\ \leftarrow \end{matrix}
\quad  \cdots \quad X_{q-1}\quad  \begin{matrix} \xleftarrow
{X_*(\delta_i)} \\ \vdots \\ \longleftarrow
\end{matrix}\quad
 X_q
 \cdots  )
\end{equation*}
similar to  the model in the case of a Normal Crossing Divisor. Here
the $X_*(\delta_i)$ are called the face maps 
(see below \ref{face} for the definition),  one for each $i \in [0,q]$. 

Moreover, we still  need   to state commutativity
relations when we compose the face maps. When we consider diagrams
of complexes of sheaves, we need resolutions of such sheaves, with
compatibility with respect to the maps $X_*(\delta_i)$ so as to avoid
the dependence on homotopy that we met
in the mixed cone construction.

The language of the simplicial category gives a rigorous setting to
state the compatibility relations needed, and  leads to  the
construction in \cite{HIII} of a simplicial hypercovering of an
algebraic variety $X$,
 using a general simplicial technique  combined with
 desingularization  at the various steps.
 
 We will take this topological construction here as granted and concentrate
 on the construction of a
 cohomological mixed Hodge complex on the variety $X$.  This construction is
 based on a diagonal process,
  out of various logarithmic complexes on the terms
 of the simplicial hypercovering, as in the previous case
  of normal crossing divisor,
   without the ambiguity of the choice
   of homotopy, because we carry resolutions of simplicial
   complexes of sheaves, hence functorial in the simplicial
   derived category.

In particular, we should view the simplicial
   category as a set of diagrams and the construction is carried out
   in a ``category of diagrams". In fact, there exists another construction
   based on  the
   ``category of diagrams" of cubical schemes \cite{St}, and
  an alternative construction with diagrams with only four edges
  \cite{E} for embedded varieties.

  In all cases, the Mixed Hodge Structure is constructed first
  for smooth varieties and normal
crossing divisors, then it is deduced for general varieties. The
uniqueness follows from the compatibility of the Mixed Hodge
Structure with Poincar\'e duality and classical exact sequences on
cohomology.

\subsection{MHS on  cohomology of simplicial varieties} To construct
a natural Mixed Hodge Structure on the cohomology of an algebraic
variety $S$, not necessarily smooth or compact, Deligne considers a
simplicial smooth variety $\pi: U_* \to S$ which is a
cohomological resolution  of the original variety in the sense
that the direct image  $\pi_* \Z_{U_*}$ is a resolution of $\Z_S$
(descent theorem, see Theorem \ref{descent}).

On each term of the simplicial resolution, which consists of the
complement of a normal crossing divisor in a smooth compact
complex variety, the various logarithmic complexes are connected
by functorial relations and form a simplicial Cohomological Mixed
Hodge Complex giving rise to the Cohomological Mixed Hodge Complex
defining the Mixed Hodge Structure we are looking for on the
cohomology of $S$. Although such a construction is technically
elaborate, the above  abstract development of Mixed Hodge Complexes
 leads easily to the result without further difficulty.
 
 \subsubsection{Simplicial category}
The simplicial category $\Delta$
 is defined by its objects, its morphisms and the composition of morphisms.\\
 i) The objects of $\Delta$ are the subsets of integers ${\Delta}_n := \lbrace 0,1,
\ldots ,n \rbrace$ for $ n \in {\hbox{\bf N}}$,\\
 ii) The set of morphisms of $\Delta$ are the sets $H_{p,q}$ of increasing mappings from
${\Delta}_p$ to
 ${\Delta}_q$ for integers $p,q \ge 0$, with the natural
composition of mappings : $H_{pq} \times H_{qr} \rightarrow
H_{pr}$. \\
Notice that $f: {\Delta}_p \to {\Delta}_q$ is increasing in the non-strict sense
 $ \forall i < j, \, f(i) \leq f(j)$.

\begin{definition} \label{face} We define for $  0 \leq i \leq n+1$ the
$i-$th face map as the unique strictly increasing mapping such
that $i \not\in {\delta}_i ( {\Delta}_n ) $: ${\delta}_i \colon
{\Delta}_n \rightarrow {\Delta}_{n+1}, \quad i \not\in {\delta}_i
( {\Delta}_n ):= Im \,{\delta}_i$.
\end{definition}

 The semi-simplicial category
$\Delta_{>}$ is obtained when we consider only  the strictly
increasing morphisms in $\Delta$. In what follows we could
restrict the constructions to
 semi-simplicial spaces which underly the simplicial spaces and
  work only with semi-simplicial spaces, since we use only the face
maps.

\begin{definition} A simplicial (resp.
co-simplicial) object $X_*:= (X_n)_{n\in \N}$ of a category
${\mathcal C}$ is a contravariant (resp. covariant) functor from
$\Delta$ to ${\mathcal C}$.\\
A morphism  $a: X_* \to Y_*$ of simplicial (resp. co-simplicial)
objects is defined by its components $a_n: X_n \to Y_n$ compatible
with the various maps image by the functor of simplicial morphisms
in $H_{pq}$ for all $p,q\in\N$.
\end{definition}

The functor $\Gamma:{\Delta} \to {\mathcal C}$  is defined by
$\Gamma (\Delta_n):= X_n$ and for each $f: {\Delta}_p \to
{\Delta}_q$, by $\Gamma (f): X_q \to X_p $ (resp.
 $\Gamma (f): X_p \to X_q $); $\Gamma (f)$  will be denoted by $X_*(f)$.

\subsubsection{Sheaves on a simplicial space}
If ${\mathcal C}$ is the category of topological spaces, we can define 
simplicial topological spaces.
 A sheaf $F^*$ on a simplicial topological space
$X_*$ is defined by: \\
1) A family of sheaves $F^n$ on $X_n$,\\
2) For each $f: \Delta_n \to \Delta_m$ with $X_*(f): X_m \to X_n$,
an $X_*(f)-$morphism $F_*(f)$ from $F^n$ to $F^m$, that is
$X_*(f)^* F^n \to  F^m$ on $ X_m $ satisfying for all $g: \Delta_r
\to \Delta_n$, $F_*(f\circ g) = F_*(f)\circ F_*(g)$.\\ A morphism
$u: F^* \to G^*$ is a family of morphisms $u^n:F^n \to G^n$ such that
for all $f: \Delta_n \to \Delta_m$, $ u^m F^*(f) = G^*(f) u^n$
where the left term is: 
$$X_*(f)^* F^n \rightarrow F^m
\xrightarrow{u^m} G^m$$ 
and the right term is: 
$$X_*(f)^* F^n
\xrightarrow{X_*(f)^*(u_n)}
X_*(f)^* G^n \to G^m .$$ \\
 The  image of the
$i-$th face map  by a functor is also denoted abusively by the
same symbol
 $\delta_i \colon X_{n+1} \rightarrow X_n$.\\
For a ring $A$, we will consider the derived category
 of cosimplicial sheaves of $A-$modules.

\subsubsection{Derived category on a simplicial space}
 The definition of a complex of sheaves $K$ on a simplicial topological space
$X_*$ follows from the definition of  sheaves, it has two degrees
$K := K^{p,q}$ where $p$ is the degree of the complex and $q$ is
the simplicial degree, hence for each $p$, $K^{p,*}$ is a
simplicial sheaf and for each $q$, $K^{*,q}$ is a complex on
$X_q$.

A quasi-isomorphism (resp. filtered, bi-filtered) $\gamma: K \to
K'$ (resp. with filtrations) of simplicial complexes on $X_*$, is
a morphism of simplicial complexes inducing a quasi-isomorphism
$\gamma^{*,q}: K^{*,q} \to {K'}^{*,q}$ (resp. filtered,
bi-filtered) for each space $ X_q$.

The definition of the derived category  (resp. filtered,
bi-filtered) of the abelian category of abelian sheaves of groups
(resp. vector spaces) on a simplicial space is obtained by
inverting the quasi-isomorphisms ( resp. filtered, bi-filtered).

 \subsubsection{}  A topological space $S$ defines a simplicial
constant space $S_*$ such that $S_n = S$ for all $n$ and $S_*(f) =
Id $ for all $f\in H^{p,q}$.

 An augmented simplicial space $\pi: X_* \to S$
 is defined by a family of maps $\pi_n: X_n \to S_n
= S$ defining a morphism of simplicial spaces.

\subsubsection{} If the target category ${\mathcal C}$ is the category
 of complex analytic spaces, we define a simplicial complex analytic space.
The structural sheaves $\OO_{X_n}$ of a simplicial
complex analytic space form a simplicial sheaf of rings.
 Let $\pi: X_* \to S$ be an augmentation to a complex analytic space $S$,
the de Rham complex of sheaves $\Omega^*_{X_n/S}$ for various $n$
form
a complex of sheaves on $X_*$ denoted $\Omega^*_{X_*/S}$\,.

 A simplicial sheaf $F^*$ on a constant simplicial space $S_*$
defined by $S$ corresponds to  a co-simplicial sheaf on $S$; hence
if $F^*$ is abelian, it defines a complex  via the face maps,
with:
$$d = \sum_i (-1)^i
\delta_i: F^n \to F^{n+1}.$$ A complex of abelian sheaves $K$ on
$S_*$, denoted by $K^{n,m}$ with $m$ the cosimplicial degree,
defines a simple complex $sK$:
\[  (sK)^n:= \oplus _{p+q=n} K^{pq};  \quad d(x^{pq}) =
d_K(x^{pq}) + \sum_i(-1)^i \delta_i x^{pq}. \] The filtration $L$
with respect to the second degree will be useful:
\[ L^r(sK) = s (K^{pq})_{q\geq r}. \]

\subsubsection{Direct image in the derived category of abelian sheaves
(resp. filtered, bi-filtered)} For  an augmented simplicial space
$a: X_* \to S $, we define a functor denoted $ R a_{*}$ on
complexes $K$ (resp. filtered $(K,F)$, bi-filtered $(K, F, W)$))
of abelian sheaves on $X_*$. We may view $S$ as a constant
simplicial scheme $S_*$ and $a$ as a morphism $a_*: X_* \to S_*$.
In the first step we construct a complex $I$ (resp. $(I,F)$,
$(I,F,W)$)of acyclic (for example flabby) sheaves,
quasi-isomorphic (resp. filtered, bi-filtered) to $K$ (resp.
 $(K,F)$,  $(K, F, W)$)); we can always take Godement resolutions
(\cite{Iver} Chap. II, \S 3.6 p. 95 or \cite{G} Chap. II, \S 4.3
p.167) for example, then in each degree $p$, $(a_q)_* I^p$ on $S_q
= S$ defines for varying $q$ a cosimplicial sheaf on $S$ denoted
$(a_*)_* I^p$, and a differential graded complex for varying $p$,
which is  a double complex whose associated
 simple complex is denoted $s (a_*)_* I := R a_* K$:
\[  (R a_* K)^n := \oplus _{p+q=n} (a_q)_* I^{p,q};  \quad d x^{pq} =
d_I(x^{pq}) + (-1)^p \sum_{i=0}^{q+1}(-1)^i \delta_i x^{pq} \in (R
a_* K)^{n+1}\] where $q$ is the simplicial index ( $\delta_i
(x^{pq}) \in I^{p,q+1} $ and $p$ is the degree. In particular for
$S$ a point we define the hypercohomology of $K$:
\[ R\Gamma (X_*, K):= sR \Gamma^* (X_*, K); \quad  \H^i  (X_*, K):=
H^i (R\Gamma (X_*, K)). \]
  Respectively, the definition of $R a_* (K, F)$ and  $R a_* (K, F,
  W)$is similar.

 The filtration $L$ on $s (a_*)_* I := R a_* K$ defines a
 spectral sequence:
 \[ E^{pq}_1 = R^q (a_p)_* (K_{|X_p}):= H^q (R (a_p)_* (K_{|X_p}))
  \Rightarrow H^{p+q}(R a_* K):= R^{p+q}a_* K \]
  
\begin{remark}[Topological realization]
 Recall that a morphism of simplices $f: \Delta_n \to \Delta_m$ has a geometric
realization $|f|: |\Delta_n| \to |\Delta_m|$ as the affine map
defined when we identify a simplex $\Delta_n$ with the vertices of
its affine realization in $\R^{\Delta_n}$. We construct the
topological realization of a topological semi-simplicial space
$X_*$ as the quotient of the topological space $Y = \coprod_{n
\geq 0} X_n \times |\Delta_n|$ by the equivalence relation $\RR$
generated by the  identifications:
 \[\forall  f: \Delta_n \to \Delta_m, x \in X_m, a \in |\Delta_n|,
\quad (x,|f|(a))\equiv (X_*(f)(x), a) \]
  The topological realization  $|X_*|$ is the quotient space of
  $Y$, modulo the relation $\RR$, with its quotient topology.
  The construction above of the cohomology amounts to the computation of the
  cohomology of the topological space $|X_*|$ with coefficient in
  an abelian group $A$:
  \[ H^i  (X_*, A)\simeq H^i  (|X_*|, A). \]
  \end{remark}

\subsubsection{Cohomological descent} Let $a: X_* \to S $ be an
augmented simplicial scheme; any abelian sheaf $F$ on $S$, lifts
to a sheaf $a^* F$ on $X_*$ and we have a natural morphism:
\[\varphi(a):F \to  R a_* a^* F \,\,{\rm in }\,\, D^+(S). \]

\begin{definition}[cohomological descent] The morphism $a: X_* \to S $ is of
cohomological descent if the natural morphism $\varphi(a)$ is an
isomorphism in $D^+(S)$ for all abelian sheaves $F$ on $S$.
\end{definition}

 The definition
amounts to the following conditions:
\begin{equation*}
 F \xrightarrow {\sim}Ker (a_{0*}a_0^*F
\xrightarrow{\delta_1 -\delta_0} a_1^*F); \quad  R^ia_*a^* F = 0
\,\, {\rm for} \,\, i > 0.
\end{equation*}
In this case for all complexes $K$ in $D^+(S)$:
\[R\Gamma(S,K) \simeq R\Gamma (X_*, a^* K)\]
and we have a spectral sequence:
\[ E^{pq}_1 = \H^q(X_p,a_p^* K) \Rightarrow \H^{p+q}(S,K),
 \quad d_1=\sum_i (-1)^i \delta_i
: E_1^{p,q} \to E_1^{p+1,q}.\]

 \subsubsection{MHS on cohomology of algebraic varieties} A  simplicial complex variety
 $X_*$ is smooth (resp. proper) if every $X_n$ is smooth (resp.
 compact).

\begin{definition}[NCD] A simplicial Normal Crossing Divisor is a family $Y_n \subset X_n$
 of normal crossing divisors such that the family of open subsets
 $U_n := X_n - Y_n$ form a simplicial subvariety $U_*$ of $X_*$,
 hence
 the family of filtered  logarithmic complexes  $(\Omega^*_{X_n}(Log
 Y_n))_{n \geq 0}, W)$ form a filtered complex on $X_*$.
\end{definition}

The following theorem is admitted here:

\begin{theorem}(\cite{HIII}  6.2.8)\label{descent} For each separated complex variety $S$, \\
i) There exist a simplicial variety proper and smooth  $X_*$ over
$\C$,
 containing a normal crossing divisor $Y_*$ in $X_*$ and
  an augmentation $a: U_* = (X_*-Y_*) \to
 S$ satisfying the cohomological descent property.\\
 Hence for all abelian
 sheaves $F$ on $S$, we have an isomorphism $F \xrightarrow{\sim} R a_* a^*F$.\\
ii)  Moreover, for each morphism $f: S \to S'$, there exists a
morphism
 $f_*: X_* \to X'_*$ of simplicial varieties proper and smooth
  with normal crossing divisors $Y_*$ and $Y'_*$ and   augmented  complements
  $a:  U_* \to  S$ and $a': U'_* \to S'$ satisfying the
   cohomological descent property, with
  $ f( U_* ) \subset U'_*$ and $a' \circ f = a$.
\end{theorem}

The proof is based on Hironaka's desingularisation theorem and on
a general contruction of hypercoverings described briefly by
Deligne  in \cite{HIII} after preliminaries on the general theory
of hypercoverings. The desingularisation is carried at each step
of the construction by induction.

 \begin{remark} We can assume  the Normal Crossing Divisor with smooth
irreducible  components.
\end{remark}

\subsubsection{}
An $A$-Cohomological Mixed Hodge Complex $K$ (CMHC) on a
topological simplicial space  $X_*$ consists of:\\
i) A complex $K_A$ of sheaves of $A-$modules on  $X_*$ such that
$\H^k(X_*,K_A)$ are $A$-modules of finite type, and $\H^*
(X_*,K_A) \otimes \Q
\xrightarrow {\sim} \H^{*} (X_*,K_A \otimes \Q)$,\\
ii) A filtered  complex
 $(K_{A \otimes \Q}, W) $ of filtered  sheaves of $A \otimes
\Q$ modules on $X_*$ with an increasing  filtration $ W$ and an
isomorphism
 $K_{A \otimes \Q}  \simeq K_A \otimes \Q$ in the derived category on
 $X_*$,\\
iii) A bi-filitered  complex $(K_{\C}, W, F) $ of  sheaves of
complex vector spaces on $ X_*$ with an increasing
 (resp. decreasing) filtration $W$ (resp. $F$ ) and an isomorphism
  $\alpha : (K_{A \otimes \Q}, W) \otimes \C \xrightarrow {\sim} (K_{\C}, W)$
   in the derived category on $X_*$.\\
Moreover, the following axiom is satisfied \\
(CMHC) The restriction of $K$ to each $X_n$ is an $A$-Cohomological
Mixed Hodge Complex .

\subsubsection{} Let $X_*$ be a  simplicial complex compact smooth
algebraic variety with $Y_*$ a Normal Crossing Divisor in $X_*$
such that $j: U_*= (X_*-Y_*) \to X_*$ is an open simplicial
embedding, then
\[ (R j_* \Z, (Rj_* \Q, \tau_{\leq}),( \Omega^*_{X_*}(Log
Y_*), W, F))\] is a Cohomological Mixed Hodge Complex  on $X_*$.

\subsubsection{}
   If we apply the global section functor to
an $A$-Cohomological Mixed Hodge Complex  $K$ on $X_*$, we get an
$A$-cosimplicial Mixed Hodge
Complex defined as follows:\\
 1) A cosimplicial complex $R
\Gamma^* K_A$ in the derived category of cosimplicial
$A-$modules,\\
2 ) A filtered cosimplicial complex $R \Gamma^* (K_{A\otimes\Q},
W) $ in the derived category of filtered cosimplicial vector
spaces, and an isomorphism  $(R \Gamma^* K_A )\otimes \Q \simeq R
\Gamma^* (K_{A\otimes\Q})$. \\
3) A bi-filtered cosimplicial complex $R \Gamma^* (K_{\C}, W, F)$
in the derived category  of bi-filtered cosimplicial
vector spaces,\\
4) An isomorphism  $R \Gamma^* (K_{A\otimes\Q}, W) \otimes\C
\simeq R \Gamma^* (K_{\C}, W)$  in the
derived category of filtered cosimplicial vector spaces.\\
The hypercohomology
 of a cosimplicial $A$-cohomological mixed Hodge complex on $X_*$ is
 such a complex.

\subsubsection{ Diagonal filtration} To a cosimplicial  mixed
Hodge complex $K$, we associate here a differential graded complex
which is viewed as a double complex whose associated simple
complex is denoted $sK$. We put on $sK$ a weight filtration by a
diagonal process.

\begin{definition}[Differential
graded $A$-MHC] A differential graded  $DG^+$ ( or a complex of
graded objects) is a bounded below complex  with two degrees, the
first defined by the complex and the second  by the gradings. It
can be viewed as a double
 complex.\\
 A differential graded  $A$-Mixed Hodge Complex is defined by a system of
   $DG^+$-complex
(resp. filtered, bi-filtered):
\[K_A, (K_{A\otimes\Q}, W), K_A \otimes \Q \simeq K_{A\otimes\Q},
(K_\C, W, F), (K_{A\otimes\Q}, W)\otimes \C \simeq (K_{\C}, W )\]
 such that for each  degree
$n$ of the grading, the component  $(K_{\C}^{*,n}, W, F) $ underlies
an $A$-Mixed Hodge Complex.
\end{definition}

 A cosimplicial Mixed
Hodge Complex $(K, W, F)$ defines a $DG^+$-$A$-Mixed Hodge Complex
\[sK_A, (sK_{A\otimes\Q}, W),sK_A \otimes \Q \simeq sK_{A\otimes\Q},
 (sK_{\C}, W, F), (sK_{A\otimes\Q}, W)\otimes \C \simeq (sK_{\C}, W ) \]
the degree of the grading is the cosimplicial degree.

\begin{definition}[Diagonal filtration]  The diagonal filtration
$\delta(W,L)$ of $W$ and $L$ on $sK$ is defined by:
\[ \delta(W,L)_n (sK)^i:= \oplus_{p+q=i} W_{n+q}K^{p,q}. \]
where $L^r(sK) = s (K^{p,q})_{q \geq r}$. For a  bi-filtered
complex $(K, W, F)$  with a decreasing $F$, the sum over $F$ is
natural (not diagonal).
\end{definition}

\subsubsection{Properties} We have:
\[ Gr^{\delta(W,L)}_n (sK) \simeq \oplus_p
Gr^W_{n+p}K^{*,p}[-p]\]
 In the case of a $DG^+$-complex
 defined as the hypercohomology of  a complex $(K,W)$ on a
simplicial space $X_*$, we have:
  \[ Gr^{\delta(W,L)}_n R
\Gamma K \simeq \oplus_p R \Gamma (X_p, Gr^W_{n+p}K^{*,p})[-p]. \]
 and for a  bi-filtered complex  with a decreasing $F$:
 \[Gr^{\delta(W,L)}_n R (\Gamma K , F) \simeq \oplus_p
R \Gamma (X_p, (Gr^W_{n+p}K, F))[-p].\]
 Next we remark:

\begin{lemma}
 If $H = ( H_A, W, F)$ is an  $A$-Mixed Hodge Structure, a filtration $L$
of $H_A$ is  a filtration  of Mixed Hodge Structure, if and only
if, for all $n$, 
$$(Gr^n_L H_A, Gr^n_L(W), Gr^n_L(F))$$ 
is an
$A$-mixed Hodge structure.
\end{lemma}

\begin{theorem}[Deligne (\cite{HIII} thm. 8.1.15)] Let $K$ be a graded differential
$A$-Mixed Hodge Complex (for example,
 defined by a cosimplicial $A$-mixed Hodge complex).\\
 i) Then, $(sK, \delta(W,L), F)$ is an $A$-Mixed Hodge Complex. \\
 The first terms of the
 weight spectral sequence:
\[_{\delta(W,L)}E_1^{pq}(sK\otimes \Q) = \oplus_n
H^{q-n}(Gr^W_n K^{*,p+n})\]
 form the simple complex $(_{\delta(W,L)}E_1^{pq}, d_1)$ of
  $A\otimes \Q$-Hodge structures of
weight $q$
 associated to the double complex where $m = n+p$ and $E_1^{pq}$
 is represented by the sum of the terms on the diagonal :
 \[\begin{array}{ccccc}
H^{q-(n+1)}(Gr^W_{n+1}
K^{*,m+1})&\xrightarrow{\partial}&H^{q-n}(Gr^W_n
K^{*,m+1})&\xrightarrow{\partial}&H^{q-(n-1)}(Gr^W_{n-1}
K^{*,m+1})\\
d''\uparrow &&d''\uparrow&&d''\uparrow\\
H^{q-(n+1)}(Gr^W_{n+1}
K^{*,m})&\xrightarrow{\partial}&H^{q-n}(Gr^W_n
K^{*,m})&\xrightarrow{\partial}&H^{q-(n-1)}(Gr^W_{n-1} K^{*,m})
\end{array}\]
where $\partial$ is a connecting morphism and $d''$ is
simplicial.\\
  ii) The terms $_LE_r$ for $ r
>0$, of the spectral sequence defined by $(sK_{A\otimes \Q}, L)$
are endowed with a natural $A$-Mixed Hodge Structure, with
differentials $d_r$
compatible with such structures. \\
iii) The filtration $L$ on $H^*(sK)$ is a filtration in the
category of Mixed Hodge Structures and:
\[  Gr^p_L(H^{p+q} ((sK), \delta(W, L)[p+q], F) =  ({_LE}_{\infty}^{pq}, W, F).\]
\end{theorem}

\subsubsection{} In the case of a smooth simplicial variety complement of
a normal crossing divisor at infinity,
 the cohomology groups $H^n(U_*,\Z)$ are
 endowed with the Mixed Hodge Structure defined
by the following Mixed Hodge Complex:
 \[ R\Gamma (U_*,\Z), \,\,
 R\Gamma (U_*,\Q),\delta(W,L)), R\Gamma (U_*,\Omega^*_{X_*}(Log
Y_*)),\delta(W,L)), F)\] with natural compatibility isomorphisms,
satisfying:
\[Gr^{\delta(W,L)}_n R \Gamma (U_* , \Q) \simeq \oplus_m
 Gr^W_{n+m}R \Gamma (U_m,, \Q))[-m]
\simeq \oplus_m R \Gamma (Y^{n+m}_m, \Q)[-n-2m]\] where the first
isomorphism corresponds to the diagonal filtration and the second
to the logarithmic complex for the open set $U_m$; recall that
$Y^{n+m}_m$ denotes the disjoint  union of intersections of $n+m$
components of the normal crossing divisor $Y_m$ of simplicial
degree $m$. Moreover:
\[_{\delta(W,L)}E_1^{p,q} = \oplus_n
H^{q-2n}(Y^n_{n+p}, Q)   \Rightarrow H^{p+q}(U_*,\Q) \]
 The filtration $F$
induces on $_{\delta(W,L)}E_1^{p,q}$ a Hodge Structure of weight
$b$ and the differentials $d_1$ are compatible with the Hodge
Structures. The term $E_1$ is the simple complex associated to the
double complex of Hodge Structure of weight $q$ where $G$ is an
alternating  Gysin map:
 \[
\begin{array}{ccccc}
H^{q-(2n+2)}(Y^{n+1}_{p+n+1}, \Q)
&\xrightarrow{G}&H^{q-2n}({Y}^n_{p+n+1},
\Q)&\xrightarrow{G}&H^{q-2(n-2)}({Y}^{n-1}_{p+n+1}, \Q)\\
 \sum_i (-1)^i\delta_i\uparrow &
 &\sum_i (-1)^i\delta_i\uparrow&&\sum_i (-1)^i\delta_i\uparrow\\
H^{q-(2n+2)}(Y^{n+1}_{p+n}, \Q)
&\xrightarrow{G}&H^{q-2n}({Y}^{n}_{p+n},
\Q)&\xrightarrow{G}&H^{q-2(n-2)}({Y}^{n-1}_{p+n}, \Q)
\end{array}
\]
where the Hodge Structure on the columns are  twisted respectively
by $(-n-1), (-n), (-n+1)$, the  lines are defined by the
logarithmic complex, while the vertical differentials are
simplicial. We deduce from the general theory:

\begin{prop} i) The Mixed Hodge Structure on $H^n(U_*,\Z) $ is  defined by the
graded differential mixed Hodge complex associated to the
simplicial Mixed Hodge Complex  defined by the logarithmic complex
on each term of $X_*$  and it is functorial in the
couple $(U_*,X_*)$,\\
ii)  The rational weight spectral sequence degenerates at rank $2$
and the Hodge Structure on $E_2$ induced by $E_1$ is
isomorphic to the Hodge Structure on $Gr_W H^n(U_*,\Q)$.\\
iii)The Hodge numbers $h^{pq} $ of $H^n(U_*,\Q)$ vanish  for $p
\notin[0,n]$ or $q \notin [0,n]$.\\
iv) For $Y_* = \emptyset$, the Hodge numbers $h^{pq}$ of
$H^n(X_*,\Q)$ vanish for $p \notin[0,n]$ or $q \notin [0,n]$ or
$p+q > n$.
\end{prop}

\begin{definition} The Mixed Hodge Structure on the cohomology
of a complex algebraic variety
 $H^n(X,\Z)$  is defined by any logarithmic simplicial resolution of
 $X$  via the isomorphism
 with $H^n(U_*,\Z) $ defined by the augmentation $a: U. \to X$.
\end{definition}

The Mixed Hodge Structure just defined  does not depend on the
resolution and  is functorial in $X$ since we can reduce to the
case where a morphism $f:X \to Z$ is covered by a morphism of
hypercoverings.

\subsubsection{Problems} 1) Let $i: Y \to X $ be a closed subvariety
of $X$ and $j: U := X-Y \to X$ the embedding of the complement.
Then the two long exact sequences of cohomology:
\begin{equation*}
\begin{split}
&\ldots \to H^i( X,X-Y, \Z) \to   H^i( X, \Z) \to  H^i(Y, \Z) \to
 H^{i+1}_Y( X, \Z)\to \ldots \\
 &\ldots \to H^i_Y( X, \Z) \to   H^i( X, \Z) \to  H^i( X-Y, \Z) \to
 H^{i+1}_Y( X, \Z)\to \ldots
 \end{split}
\end{equation*}
underly  exact sequences of Mixed Hodge Structure.

The idea is to use a simplicial  hypercovering of the morphism $i$
in order to define two Mixed Hodge Complexes: $K(Y)$ on $Y$ and
$K(X)$ on $X$ with a well defined morphism on the level of
complexes $i^*: K(X) \to K(Y)$ (resp. $j^*: K(X) \to K(X-Y)$),
then the long exact sequence is associated to the mixed cone
$C_M(i^*)$(resp. $C_M(j^*)$).

In particular, one deduce associated long exact sequences by
taking the graded spaces with respect to the filtrations $F$ and
$W$.

\n 2) {\it K\"{u}nneth formula} \cite{St}. Let $X$ and $Y$ be two
algebraic
 varieties, then the isomorphisms of cohomology vector spaces:
   \[ H^i( X \times Y, \C) \simeq \oplus_{r+s = i}H^r(X, \C)
   \otimes H^s(Y, \C)\]
   underly isomorphisms of $\Q$-mixed Hodge structure.
   The answer is in two steps:\\
   i) Consider the tensor product of two mixed Hodge complex defining
   the mixed Hodge structure of $X$
   and $Y$ and deduce the right term, direct sum of tensor product
   of mixed Hodge structures.\\
   ii) Construct a quasi-isomorphism of the tensor product with a
   mixed Hodge complex defining the mixed Hodge structure of $X\times Y$.\\
   iii) Deduce  that the cup product on  the cohomology of an algebraic
    variety is compatible with mixed Hodge structure.

\subsection{MHS on the
cohomology of a complete embedded algebraic  variety}
  For embedded varieties into smooth varieties,
   the mixed Hodge structure on  cohomology can be  deduced by a simple
   method using
exact sequences, once the mixed Hodge structure for normal
crossing divisor has been constructed, which should easily
convince of the natural aspect of this theory. The technical
ingredients consist of Poincar\'e duality and its dual the trace
(or Gysin ) morphism.

Let $p: X' \rightarrow X$  be a proper morphism of complex smooth
varieties of same dimension,  $Y$ a closed subvariety of $X$ and
$Y'$ = $p^{-1}(Y)$. We suppose that $Y'$ is a Normal Crossing
Divisor in $X'$ and the restriction of $p$ induces  an isomorphism
$p_{/X'-Y'} : X'-Y' \xrightarrow{\sim} X-Y$:
$$
\begin{array}{ccccc}Y'&\xrightarrow{i'}&X'&\xleftarrow{j'}&X'-Y'\\
\quad\downarrow{p_Y}&&\quad\downarrow{p}&&\quad\downarrow{p_{X'-Y'}}\\
Y&\xrightarrow{i}&X&\xleftarrow{j}&X-Y
\end{array}$$
The  trace  morphism $Tr p$ is defined as Poincar\'e dual to the
inverse image $p^*$ on cohomology, hence the $Tr p$ is compatible
with the Hodge Structures. It can be defined at the level of sheaf
resolutions of $\Z_{X'}$ and $\Z_{X} $ as constructed by Verdier
\cite{V2}, that is in derived category $ Tr \, p \colon R p_*
\Z_{X'} \rightarrow \Z_X $
   hence we deduce
  by restriction  morphisms  depending on
the embeddings of $Y$ and $Y'$ into $X'$.\\
$(Tr \, p)/_Y : Rp_* \Z_{Y'}
  \rightarrow
  \Z_Y, \, (Tr\, p)/_Y : H^{i} (Y',\Z) \rightarrow
H^i(Y,\Z)$,  and  $Tr p:  H_c^{i} (Y',\Z) \rightarrow
H_c^i(Y,\Z)$.
\begin{remark}  Let $U$ be a neighbourhood of $Y$ in $X$, retract by deformation
onto $Y$ such that $U' = p^{-1}(U)$ is a retract by deformation onto
$Y'$; this is the case if $Y$ is a  sub-variety of $X$. Then the
morphism $(Tr\, p)/_Y$ is deduced from $Tr(p/_U)$ in the diagram:
$$\def\normalbaselines{\baselineskip20pt
\lineskip3pt \lineskiplimit3pt }
\def\mapright#1{{\mathbb S}mash{
\mathop{\longrightarrow}\limits^{#1}}}
\def\mapdown#1{\Big\downarrow \rlap{$\vcenter{\hbox{${\mathbb S}criptstyle#1$}}$}}
\begin{array}{ccc}H^{i}(Y',\Z)&\xleftarrow{\sim}&
H^{i}(U',\Z)\\\quad \downarrow{(Tr\,p)/_Y}&&\quad\downarrow{Tr(p/_U)}\\
H^i(Y,\Z)&\xleftarrow{\sim}&H^i(U,\Z)\end{array}$$
\end{remark}
Consider now the diagram:
$$\begin{array}{ccccc}R\Gamma_c(X' - Y',\Z)& \xrightarrow{j'_*}&
R\Gamma(X',\Z)&
\xrightarrow{{i'}^*}& R\Gamma(Y',\Z)\\
Tr p\downarrow{}&&\downarrow{} Tr p&&\downarrow{} (Tr p)_{|Y}\\
R\Gamma_{c}(X-Y,\Z)& \xrightarrow{j_*}& R\Gamma(X,\Z)&
\xrightarrow{i^*}& R\Gamma(Y,\Z)\end{array}$$

\begin{prop} \cite{E}  i) The morphism $p_Y^* \colon H^i(Y,\Z)
\rightarrow H^i (Y',\Z)$ is injective with retraction $(Tr\,
p)_{/Y}$.\\
 ii) We have a quasi-isomorphism of
$i_* \Z_Y$ with the cone $C({i'}^* - Tr \, p)$ of the morphism
${i'}^* - Tr \, p$. The long exact sequence  associated to the
cone splits into short exact sequences:
  $$ 0 \rightarrow H^i(X',\Z) \xrightarrow{ {i'}^*-Tr \, p}
H^i(Y',\Z) \oplus H^i(X,\Z) \xrightarrow{(Tr \, p)_{/Y}+i^*}
 H^i(Y,\Z) \rightarrow 0.$$
 Moreover ${i'}^* - Tr \, p$ is a  morphism of mixed Hodge structures.
 In particular,
the weight of $ H^i(Y,\C)$ vary in the interval
 [$0,i$] since this is true for $Y'$ and $X$.
 \end{prop}

\begin{definition} The mixed Hodge structure of  $Y$ is defined as  cokernel of
${i'}^*-Tr \, p$ via its isomorphism with $ H^i(Y,\Z)$, induced by
$(Tr \, p)_{/Y}+i^*$. It coincides with Deligne's mixed Hodge
structure.
 \end{definition}

 This result
shows the uniqueness of the theory of mixed Hodge structure, once
the Mixed Hodge Structure of the normal crossing divisor $Y'$ has
been constructed. The above  technique  consists  in the
realization of the Mixed Hodge Structure on the cohomology of $Y$
as relative cohomology with Mixed Hodge Structures on $X$, $X'$
and $Y'$ all smooth proper or normal crossing divisor. Notice that
the Mixed Hodge Structure on $H^i(Y, \Z)$ is realized as a
quotient and not as an extension.

\begin{prop} Let $ p: X' \to X$ be a desingularization of a
complete variety $X$, then for all integers $i$, we have
\[ W_{i-1} H^i ( X , \Q) = Ker \,(H^i ( X , \Q) \xrightarrow{p^*}
 H^i ( X' , \Q)\]
\end{prop}

Let $i: Y \to X$ be the subvariety of singular points in $X$ and
let $Y' := p^{-1} (Y)$, $ i': Y' \to X'$, then we have a long
exact sequence:
\[   H^{i-1}(Y',\Q) \to H^i(X,\Q) \xrightarrow{ (p^*, - i^*)}
H^i(X',\Q) \oplus H^i(Y,\Q) \xrightarrow{i'^* + p_{|Y'}^*}
 H^i(Y',\Q)   \ldots \]
 Since the weight of $ H^{i-1}(Y',\Q)$ is $\leq i-1$, we deduce an
 injective morphism $Gr^W_i H^i(X,\Q) \xrightarrow{ (p^*, - i^*)}
Gr^W_i  H^i(X',\Q) \oplus Gr^W_i  H^i(Y,\Q)$. It is enough to
prove  for any element $ a \in Gr^W_i H^i(X,\Q)$ such that $Gr^W_i
(p^*)(a) = 0$, we have $Gr^W_i(i^*) (a) = 0$; which follows from $
Gr^W_i(p_{|Y'}^* \circ i^*) (a)) = 0$ if we prove $ Gr^W_i
(p_{|Y'}^*)$ is injective. By induction on dim.$Y$, we may assume
the injectivity for a resolution $\tilde Y \to Y$. There exists a
subvariety $Z \subset Y'$generically  covering $Y$, then a
desingularization $\tilde Z$ of $ Z' \times_Y \tilde Y$ is a
ramified covering of $\tilde Y$, hence $H^i(\tilde Y, \Q)$ injects
into $H^i(\tilde Z, \Q)$. We deduce that $H^i( Y, \Q) \to
H^i(\tilde Z, \Q)$ is injective and in particular  the factor
$H^i( Y, \Q) \to H^i(Y', \Q)$ is also injective.

\begin{remark}
[Mixed Hodge Structure on the cohomology of an embedded algebraic
variety] The construction still  apply for non proper varieties if
we construct
the Mixed Hodge Structure of an open normal crossing divisor.\\
{\it Hypothesis}. Let  $i_Z: Z \to X$ a closed embedding and $i_X:
X \to P$
 a closed embedding in a  projective space (or any proper smooth complex
 algebraic variety). By Hironaka desingularization we
 construct a diagram:
 $$\begin{array}{ccccc}
     Z'' & \rightarrow & X'' & \rightarrow & P'' \\
     \downarrow & &\downarrow  &  &\downarrow  \\
    Z' & \rightarrow & X' & \rightarrow & P' \\
     \downarrow & & \downarrow  &  & \downarrow  \\
     Z & \rightarrow& X & \rightarrow & P  \\
   \end{array}$$
first by blowing up centers over $Z$ so to obtain a smooth space
$p: P' \to P$ such that   $Z':= p^{-1}(Z)$ is a normal crossing
divisor; set  $X':= p^{-1}(X)$, then:
$$p|:X'-Z' \xrightarrow{\sim} X-Z, \quad p|:P' - Z'
\xrightarrow{\sim} P - Z$$
 are isomorphisms since the modifications are all over $Z$.
 Next, by blowing up centers over $X'$ we obtain a smooth space
$q: P'' \to P'$ such that $X'':= q^{-1}(X')$ and  $Z'':=
q^{-1}(Z')$
 are  normal crossing divisor, and $q|:P'' - X''
\xrightarrow{\sim} P' - X'$. Then, we deduce the diagram:
$$\begin{array}{ccccc}
    X'' - Z'' & \xrightarrow{i''_X} & P'' - Z'' &
     \xleftarrow{j''} & P'' - X''\\
     q_X\downarrow & &q\downarrow  &  &q\downarrow\!\wr  \\
   X' - Z' & \xrightarrow{i'_X} & P' - Z' & \xleftarrow{j'} & P' - X'\\
   \end{array} $$
Since all modification are above $X'$, we still have an
isomorphism induced by $q$ at right.  For dim.$P = d$ and all
integers $i$, the morphism $q^*: H^{2d-i}_c( P'' - Z'', \Q) \to
H^{2d-i}_c( P' - Z', \Q)$ is well defined on cohomology with
compact support since $q$ is proper; its Poincar\'e dual is called
the trace morphism $Tr q: H^i( P'' - Z'', \Q) \to H^i( P' - Z',
\Q)$ and satisfy the relation $ Tr q\, \circ \, q^*\, = \, Id $.
Moreover, the trace morphism is defined as a morphism of sheaves
$q_* \Z_{P'' - Z''}
 \to \Z_{P' - Z'}$ \cite{V2}, hence an induced trace morphism
 $(Tr q)|(X''-Z''): H^i( X'' - Z'', \Q) \to H^i( X' - Z', \Q)$ is well
 defined.

\begin{prop} With the  notations of the above diagram, we have
short exact sequences:
\begin{equation*}
\begin{split}
0 \to H^i( P'' -  Z'', \Q) \xrightarrow{(i''_X)^* - Tr q}&
 H^i( X'' -  Z'', \Q)\oplus
  H^i( P' -  Z', \Q) \\
  & \xrightarrow{(i'_X)^* - (Tr q)|(X''- Z'')}
 H^i( X' -  Z', \Q) \to 0\\
 \end{split}
\end{equation*}
\end{prop}

Since we have a vertical  isomorphism $q$ at right of the above
diagram, we deduce a long exact sequence of cohomology spaces
containing the sequences of the proposition; the injectivity of
$(i''_X)^* - Tr q$ and  the surjectivity of $(i'_X)^* - (Tr
q)|X'')$ are  deduced from $ Tr q \circ q^* = Id $ and $ (Tr
q)|(X''- Z'') \circ q^*|(X'- Z') = Id $, hence the long exact
sequence  splits into short exact sequences.

\begin{corollary} The cohomology $H^i( X -
 Z, \Z)$ is isomorphic to $H^i( X' -
 Z', \Z)$ since $X-Z \simeq X'-Z'$ and then carries the
 Mixed Hodge Structure isomorphic to  the cokernel of
$(i''_X)^* - Tr q$ acting as a morphism of  Mixed Hodge Structures.
\end{corollary}

The left term carry a  Mixed Hodge Structure as the special case
of the complementary of the normal crossing divisor: $Z''$ into
the smooth proper variety $P''$ , while the middle term is the
complementary of the normal crossing divisor: $Z'' $  into the
normal crossing divisor: $X'' $. Both cases can be treated by the
above special cases without the general theory of simplicial
varieties. This shows that the Mixed Hodge Structure is uniquely
defined by the construction on an open normal crossing divisor and
the logarithmic case for a smooth variety.
\end{remark}

 \end{document}